\documentclass{book}
\usepackage{index}
\usepackage[arrow, matrix, curve]{xy}
\usepackage{graphicx}
\usepackage[utf8]{inputenc}
\usepackage{longtable}
\usepackage{psboxit}
\usepackage{vmargin,hhline,pifont}
\usepackage{enumerate,fancyhdr}
\usepackage{amsmath}
\usepackage{amsthm}
\usepackage{amsfonts}
\usepackage{amssymb}
\usepackage{chapterbib}
\usepackage{marvosym}
\usepackage{verbatim}

\newtheorem{theorem}{Theorem}[chapter]
\newtheorem{proposition}[theorem]{Proposition}
\newtheorem{lemma}[theorem]{Lemma}

\theoremstyle{definition}
\newtheorem{definition}[theorem]{Definition}
\newtheorem{example}[theorem]{Example}

\newtheorem{notation}{Notation}[section]
\newtheorem{corollary}[theorem]{Corollary}
\theoremstyle{remark}
\newtheorem{remark}[theorem]{Remark}

\numberwithin{section}{chapter}
\numberwithin{equation}{chapter}

\newcommand{\ndash}{\nobreakdash-\hspace{0pt}}

\DeclareMathOperator{\Mexp}{Mexp}

\makeindex

\begin{document}

\frontmatter
\title{Affine Kac-Moody symmetric spaces}
\author{Walter Freyn}

\maketitle

\tableofcontents
\mainmatter

\chapter{Introduction}

In this work we develop a theory of affine Kac-Moody symmetric spaces. Those spaces are the closest infinite dimensional analogue to finite dimensional Riemannian symmetric spaces. Finite dimensional symmetric spaces and affine Kac-Moody symmetric spaces share many features of their geometry; additionally their classification follows the same principles and exhibits a similar structure and, most importantly, their relations to other classes of objects like polar actions, isoparametric submanifolds and (twin) buildings are similar~\cite{Freyn10a}.

In this work we introduce affine Kac-Moody symmetric spaces and describe their algebraic, geometric and functional analytic structure; we prove them to be quotient spaces of tame Fr\'echet Kac-Moody Lie groups in a similar way as finite dimensional Riemannian symmetric spaces (besides $\mathbb{R}^n$) are quotient spaces of semisimple Lie groups. 

From a functional analytic point of view, the central observation of our approach is the choice of the functional analytic setting we work in: we construct Kac-Moody groups as torus extension of groups of holomorphic loops form $\mathbb{C}^*$ into complex reductive Lie groups $G_{\mathbb{C}}$. This setting allows for a nice complexification of the resulting Kac-Moody groups; this complexification is a necessary step for the construction of symmetric spaces of the non-compact type. As a result we get Kac-Moody symmetric spaces of the ``compact type'' and the ``non-compact type'' together with a duality relation between simply connected Kac-Moody symmetric spaces of both types. In each case we get two sub-classes, one corresponding to Kac-Moody groups, the other corresponding to involutions of compact Kac-Moody groups. We show that the isotropy representations of Kac-Moody symmetric spaces coincide with the classical examples of $P(G,H)$\ndash actions introduced in \cite{HPTT}. All those results have counterparts in the theory of finite dimensional Riemannian symmetric spaces~\cite{Helgason01, Berndt03}.

\section{Kac-Moody geometry - The origin of the problem}

Together with proper Fredholm isoparametric submanifolds in Hilbert spaces, twin cities and proper Fredholm polar actions Kac-Moody symmetric spaces are one of the three important differential geometry objects of Kac-Moody geometry (for an overview~\cite{Freyn10a} and references therein). 

Kac-Moody geometry is the investigation of the geometry inherent to affine Kac-Moody algebras and affine Kac-Moody groups. Around 1967  Kac-Moody algebras were introduced in the works of V.\ G.\ Kac~\cite{Kac68}, R.\ V.\ Moody~\cite{Moody69}, I.\ L.\ Kantor~\cite{Kantor68} and D.-N.\ Verma (unpublished). From a formal, algebraic point of view, Kac-Moody algebras are understood as Lie algebra realizations of generalized Cartan matrices~\cite{Kac90, Moody95}. In this way Kac-Moody algebras appear as a natural generalization of simple Lie algebras. Similar to the bijection between indecomposable Cartan matrices and simple complex Lie algebras we get a new, wider bijection between generalized indecomposable Cartan matrices and simple complex Kac-Moody algebras; in this wider context Cartan matrices correspond to simple Lie groups, affine Cartan matrices to affine Kac-Moody algebras. In spite of this similarity there is an ample difference: The theory of simple Lie algebras starts with abstract Lie algebras; then structure results are proven; in the end we arrive at the concept of a Cartan subalgebra and a Cartan matrix. Finally a bijection between simple complex Lie algebras and indecomposable Cartan matrices is established. In contrast the construction of affine Kac-Moody algebras starts with an affine Cartan matrix; then one constructs a Cartan subalgebra $\mathfrak{h}$ as a realization of this Cartan matrix and proceeds then to the construction of the whole Kac-Moody algebra using the root spaces associated to $\mathfrak{h}$. Hence the usual construction and definition of Kac-Moody algebras is representation theoretic in nature and breaks manifestly the symmetry inherent to Kac-Moody algebras, thus darkening the geometry behind the definition. Because of this very abstract definition an explicit construction of Kac-Moody algebras was important: In this direction V.\ G.\ Kac, R.\ V.\ Moody and I.\ L.\ Kantor pointed out the subclass of affine Kac-Moody algebras allowing for an explicit description in terms of $2$\ndash dimensional extensions of algebras of polynomial loops into simple Lie algebras. This point of view links the theory of affine Kac-Moody algebras to the theory of simple Lie algebras in yet another very elementary geometric way. Furthermore this description suggests functional analytic completions of the loop algebras with respect to various norms and opens thus the way to the use of functional analytic methods~\cite{PressleySegal86}.

The next step towards a geometric theory is the construction of Kac-Moody groups. This was done by Jacques Tits around 1980; he developed the idea of twin buildings and twin BN-pairs and showed that Kac-Moody groups can be constructed as groups with a twin $BN$\ndash pair. Several works of various authors investigate sundry completions. The most important observation from a geometric viewpoint is, that a big class of complex Kac-Moody groups (namely the $2$-spherical ones) containing all affine Kac-Moody groups can be understood as amalgamations of $SL(2,\mathbb{C})$-groups. Hence similarly to a simple Lie group a complex Kac-Moody group has many $SL(2,\mathbb{C})$-subgroups. The well-known two- or three dimensional geometry of $SL(2,\mathbb{C})$ and its real forms $SU(2)$ and $SL(2,\mathbb{R})$ will thus play an important role in understanding the geometry of (infinite dimensional) Kac-Moody groups. For the subclass of affine Kac-Moody algebras we can construct the associated Kac-Moody groups as torus extensions of groups of polynomial loops. Unfortunately --- as for all Kac-Moody groups --- there is no exponential function with nice properties for algebras of polynomial loops and their loop groups; this problem can be partly solved using suitable completions~\cite{PressleySegal86}; nevertheless there is no completions allowing for the definition of Kac-Moody symmetric spaces such that the resulting exponential map gets a local diffeomorphism. This leads to various delicate functional analytic problems - some of them we will encounter later in these notes.

In the investigation of the geometry associated to Kac-Moody groups one is guided by a finite dimensional blueprint, namely the geometry associated to semisimple Lie groups. Let us thus review shortly the finite dimensional blueprint. The main classes of objects of this well-developped theory are the following:

\begin{enumerate}
\item {\bf Riemannian Symmetric spaces} are Riemannian manifolds with a huge symmetry group: For every point $p$ there is a symmetry $\sigma_p$, which is $-Id$ on the tangential space at $p$. A consequence is that Riemannian symmetric spaces are homogeneous.
\item {\bf Polar representations} are representations of a Lie group $G$ on a vector space $V$ such that there is a subspace $\Sigma\subset V$ called a section that meets each orbit orthogonally.
\item {\bf Spherical buildings over \boldmath$\mathbb{R}$ or \boldmath$\mathbb{C}$} are very special simplicial complexes having a huge set of subcomplexes, called apartments. They are a geometric translation of the Bruhat decomposition of Lie groups.
\item {\bf Isoparametric submanifolds} are submanifolds $V\subset\mathbb{R}^n$ such that their normal bundle is flat and the principal curvatures are constant.
\end{enumerate}

In the first two examples homogeneity is intrinsic, in the last two examples homogeneity is a priori an additional assumption. However, homogeneity can be proven under the assumption of sufficiently high rank \cite{AbramenkoBrown08, Berndt03}. Those four classes of objects are closely related: Start with a symmetric space $M$. We define the isotropy group $K_p$ to be the group of all isometries of $M$, fixing $p$. The isotropy representation of $M$ is the natural representation $K_p:T_pM\longrightarrow T_pM$ induced by the action of $K_p$ on $M$. Because of homogeneity the isotropy representations at different points are isomorphic. It is straight forward to check that the isotropy representation of a symmetric space is a polar representation. Conversely J.\ Dadok proved every polar representation on $\mathbb{R}^n$ to be orbit equivalent to the isotropy representation of a symmetric space~\cite{Dadok85}.  Principal orbits of polar representations are isoparametric submanifolds. Conversely a result of G.\ Thorbergsson shows any full irreducible isoparametric submanifold of $\mathbb{R}^n$ of rank at least three to be equivalent to a principal orbit of some isotropy representation~\cite{Thorbergsson91, Berndt03} and references therein. The boundary of a symmetric space of non-compact type can be identified with a building~\cite{Eberlein96}. Furthermore the building can be embedded into the unit sphere of the representation space of the isotropy representation and hence be seen geometrically in the tangent space of the corresponding symmetric spaces.

The crucial observation for the development of Kac-Moody geometry was the discovery of a close link between affine Kac-Moody algebras and infinite dimensional differential geometry around 1980~\cite{HPTT}. Using the description of affine Kac-Moody algebras as loop algebras, E.\ Heintze, R.\ Palais, C.-L.\  Terng, and G.\ Thorbergsson defined $H^0$\ndash Sobolev completions of Kac-Moody algebras and $H^1$\ndash Sobolev completions of Kac-Moody groups and constructed various classes of polar actions in complete analogy to the various classes of the (known) classification of finite dimensional polar actions.  C.-L.\ Terng proved furthermore that polar representations on Hilbert spaces can be described using similar completions of Kac-Moody algebras; furthermore proper Fredholm isoparametric submanifolds are principal orbits of polar representations~\cite{terng89,terng95}. Conversely all (full irreducible) proper Fredholm submanifolds of codimension not $1$ are principal orbits of polar representations by a result of E.\ Heintze and X.\ Liu~\cite{HeintzeLiu}. These surprising connections parallel the finite dimensional theory where the theorem of Dadok shows that polar representations correspond to symmetric spaces.   All those developments hint to the idea that there is a rich infinite dimensional geometry lurking behind affine Kac-Moody algebras and groups, that parallels in its broad lines the finite dimensional theory of geometric objects whose structures are described by semisimple Lie groups~\cite{Heintze06}.

Inspired by the finite dimensional blueprint besides generalizing the concepts of isoparametric submanifolds and polar actions to Hilbert spaces, C.-L.\ Terng conjectured in her article~\cite{terng95} the existence of infinite dimensional symmetric spaces as the fundamental geometric object necessary for the generalization of the finite dimensional blueprint to affine Kac-Moody algebras, remarking (remark~3.4) that severe technical
problems make the rigorous definition of those spaces difficult. 

It turns out that the crucial
point is to find an analytic framework that allows all algebraic constructions
which are needed for the description of the geometric theory. A recent review of the theory of isoparametric submanifolds
and polar actions on Hilbert spaces, which contains additional
references, is given in~\cite{Heintze06}.

Important progress towards the construction of Kac-Moody symmetric
spaces was achieved by B.\ Popescu only in 2005/2006 in his
thesis~\cite{Popescu05} where he considers weak Hilbert symmetric spaces,
modeled as loop spaces of $H^1$-Sobolev loops, equipped with a $H^0$ scalar
product. As the differential of $H^1$ loops is only in $H^0$ this approach does not allow a convincing definition of the torus
bundle extension corresponding to the $c$ and $d$ parts of the Kac-Moody algebra. To
remedy this he investigates also the framework of smooth loops, which allows the construction of symmetric Fr\'echet
spaces of the ``compact type''. However, there is no convincing definition of a complexification for those groups. Hence, the definition of the dual non-compact symmetric spaces fails completely.
As about half of all symmetric spaces are of non-compact type this is a serious detriment.

This problem was overcome only in the author's thesis~\cite{freyn09} using holomorphic loops defined on
$\mathbb{C}^*=\mathbb{C}\setminus\{0\}$.
In this ansatz we have to tackle the obstacle that the exponential
function in general defines no longer a diffeomorphism between neighborhoods of
the $0$-element in $\mathfrak{g}$ and the unit element in $G$. Therefore, one
requires methods from infinite dimensional Lie theory to define Lie group-
(resp.\ manifold-) structures on those groups and their quotients. A comprehensive review
of this topic including further references is given in the article~\cite{Neeb06},
an investigation of Lie groups of holomorphic maps is contained in the
article~\cite{Neeb07}. See also~\cite{Kriegl97}.

Hence, the theory of Kac-Moody symmetric spaces presented in this work is situated at the
melting point of 3 different areas of current research:

\begin{itemize}
  \item[-] The differential geometry of loop groups, polar actions and isoparametric
    submanifolds yields the geometric intuition.
  \item[-] The theory of infinite dimensional analysis and Lie groups gives the functional analytic framework.
  \item[-] The theory of Kac-Moody algebras and Kac-Moody groups rules the structure theory of Kac-Moody symmetric spaces and imposes precise requirements on the kind of algebraic operations that have to be defined. This imposes severe restrictions on the functional analytic framework we work in.
\end{itemize}

\noindent Using the geometric description of
extensions of loop groups we describe Kac-Moody groups as tame Fr\'echet Lie groups and construct Kac-Moody symmetric spaces of the compact and of the non-compact type as tame Fr\'echet manifolds.

By work done during the last 20 years a generalization of the finite dimensional blueprint to affine Kac-Moody algebras is now well understood. The fundamental (global) object are Kac-Moody symmetric spaces. Their isotropy representations induce polar representations on Hilbert spaces. They describe the local structure in each tangent space. Principal orbits of polar representations are proper Fredholm isoparametric submanifolds in Hilbert space. Twin cities, the appropriate completion of twin buildings, can be embedded equivariantly into the isotropy representation \cite{Freyn10d}. Chambers in twin cities correspond to points in isoparametric submanifolds; the twin city can be seen in the tangent space of a Kac-Moody symmetric space~\cite{Freyn10d}. A version of Thorbergsson's result is available in this setting, proved by E.\ Heintze and X.\ Liu~\cite{HeintzeLiu}; the only big missing link is a proof of an appropriate version of Dadok's theorem. In this direction there exist partial results by C.\ Gorodski, E.\ Heintze and K.\ Weinl.

Besides in Kac-Moody geometry itself Kac-Moody symmetric spaces appear in various places in mathematics and theoretical physics. Let us just mention supergravity theories or M-theory. The literature about this field is vast. For a start see for example~\cite{DamourKleinschmidtNicolai07}, \cite{NicolaiSamtleben05}.

Let us remark nevertheless that there are various other approaches to the construction of infinite dimensional symmetric spaces. See for example \cite{delaHarpe72}, \cite{Neretin96}, or \cite{Klotz11}. Nevertheless, all those approaches do not generalize the important structure properties of finite dimensional symmetric spaces.

\section{Affine Kac-Moody symmetric spaces}

Affine Kac-Moody symmetric spaces are tame Fr\'echet Lorentz symmetric spaces whose isometry group contains a transitive acting affine Kac-Moody group. To state our main results about the geometry of Kac-Moody symmetric spaces let us fix some notation: we denote by 
$G_{\mathbb{C}}$ a complex semisimple Lie group and by $G$ a compact real
form of $G_{\mathbb{C}}$. Furthermore let $\sigma$ be a diagram automorphism  of order $n$ of $\mathfrak{g}_{\mathbb{C}}$. (Hence $n\in \{1,2,3\}$, the case $n=1$ (the identity) is allowed, and $\omega:=e^{\frac{2\pi i}{n}}$. Now, we define the holomorphic loop spaces 
\begin{displaymath}
MG_{\mathbb C}^{\sigma}:= \{f: \mathbb C^* \rightarrow
G_{\mathbb C}| f \textrm{ is holomorphic and } \sigma\circ f (z)=f(\omega z)\}
\end{displaymath}
and
\begin{displaymath}
MG_{\mathbb
  R}^{\sigma}:=\{f:\mathbb C^* \rightarrow G_{\mathbb C} | f(S^1)
\subset G, \textrm{f is holomorphic and } \sigma \circ f (z)=f(\omega z)\}\,.
\end{displaymath}

The complex Kac-Moody groups $\widehat{MG}_{\mathbb{C}}^{\sigma}$ are now constructed as certain $(\mathbb{C}^*)^2$-bundles over $MG_{\mathbb{C}}^{\sigma}$. To simplify the notation we omit the superscript $\sigma$ whenever possible.

\noindent Let $\rho_*$ denote a suitable involution of the second kind (see~\cite{Heintze09}).

Then we can summerize the main results as follows:

\begin{theorem}[affine Kac-Moody symmetric spaces of the ``compact'' type]
Both the  Kac-Moody groups $\widehat{MG}^{\sigma}_{\mathbb R}$
equipped with their $Ad$-invariant metric (called Kac-Moody symmetric space of type $II$), and the quotient spaces
$X= \widehat{MG}^{\sigma}_{\mathbb R}/\textrm{Fix}(\rho_*)$ equipped with its
$Ad(\textrm{Fix}(\rho_*))$-invariant metric (called Kac-Moody symmetric space of type $I$) are tame Fr\'echet
symmetric spaces of the ``compact'' type with respect to their natural $Ad$-invariant
metric. Their curvature tensors satisfy 
\begin{displaymath}
\langle R(X,Y)X, Y\rangle \geq 0\,.
\end{displaymath}
\end{theorem}

\begin{theorem}[affine Kac-Moody symmetric spaces of the ``non-compact'' type]

Both quotient spaces $X=\widehat{MG}^{\sigma}_{\mathbb{C}}/\widehat{MG}^{\sigma}_{\mathbb R}$ (called Kac-Moody symmetric space of type $IV$) and  
$X=H/\textrm{Fix}(\rho_*)$, where $H$ is a non-compact real form of $\widehat{MG}^{\sigma}_{\mathbb{C}}$ equipped with their $Ad$-invariant metrics (called Kac-Moody symmetric space of type $III$),
are tame Fr\'echet symmetric spaces of the ``non-compact'' type. Their curvature tensors satisfy

\begin{displaymath}
\langle R(X,Y)X, Y\rangle \leq 0\,.
\end{displaymath}

\noindent Furthermore Kac-Moody symmetric spaces of the non-compact type are
diffeomorphic to a vector space.
\end{theorem}

\noindent Define the notion of duality as for finite dimensional Riemann symmetric spaces; hence two symmetric spaces are dual iff their curvature tensors differ only by sign. For example the $n$\ndash sphere $S^n$ (curvature $K\equiv 1$ and hyperbolic $n$\ndash space ($K\equiv -1$) are dual. Then we obtain in complete analogy to the theory of finite dimensional Riemannian symmetric spaces:

\begin{theorem}[Duality]
Affine Kac-Moody symmetric spaces of the compact type are dual to the Kac-Moody symmetric spaces of the non-compact type and vice versa.
\end{theorem}

Kac-Moody symmetric spaces have several conjugacy classes of flats. For our purposes the most important class is the one of flats of finite type. A flat is called ``of finite type'' iff it is finite dimensional. A flat is called of exponential type iff it lies in the image of the exponential map and it is called maximal iff it is not contained in another flat. Similar to a result of B.\ Popescu (see~\cite{Popescu05}) in another setting we find:

\begin{theorem}
All maximal flats of finite exponential type in a Kac-Moody symmetric space are conjugate.
\end{theorem}

\noindent 
We show that all Kac-Moody symmetric spaces are Lorentz symmetric spaces. This is a purely formal consequence of the structure of the $Ad$\ndash invariant scalar product. This is unique up to a global scaling factor.

In the finite dimensional case no complete classification
of pseudo Riemann symmetric spaces is known. However, there
are important partial results: Marcel Berger achieved in 1957 a
complete classification of pseudo Riemann symmetric spaces of
``semisimple'' type~\cite{Berger57}. Ines Kath and Martin Olbrich gave a
classification of pseudo Riemann symmetric spaces of index $1$ and $2$
and described structure results that indicate that a general
classification of pseudo Riemann symmetric spaces is out of 
reach~\cite{Kath04}, \cite{Kath06}. 

 As Kac-Moody groups are the natural infinite
dimensional analogue of semisimple Lie groups, at first sight it is tempting to
interpret Kac-Moody symmetric spaces as an infinite dimensional
analogue of the subclass of finite dimensional ``semisimple'' Lorentz
symmetric spaces.

In contrast to this point of view the structure theory of Kac-Moody symmetric spaces shows that they are a direct generalization of finite dimensional Riemannian symmetric spaces.  This concept is highlighted by the similar structure theory and geometric structure as by the similar classification including the duality relation. This point of view is further strengthened as the isotropy representations of Kac-Moody symmetric spaces induce polar actions on Hilbert spaces. Thus in the infinite dimensional setting Kac-Moody symmetric spaces take over the role played by Riemann symmetric spaces. In the end the Lorentz metric on those spaces turns out to be elementary consequence of the fact that affine Weyl groups are not finite.

Let us now turn to the structure of this work:
 
The material is ordered as follows:
\begin{enumerate}
\item[-] In chapter~\ref{chap:frechet} we collect the analytic fundamentals.  We review basic material about tame Fr\'echet spaces, prove that various spaces, we will encounter in later chapters are tame and prove various technical results, we need. The most important one among those results are an implicit function theorem and a chacterization of inverse images of ``regular points'' of certain tame maps as tame Fr\'echet manifolds.  
\item[-] In chapter~\ref{chap:alg} we investigate the algebraic structures which we need for Kac-Moody symmetric spaces. 
 Following the blueprint of the finite dimensional theory we describe
a classification of Kac-Moody symmetric spaces by the classification
of their indecomposable orthogonal symmetric affine Kac-Moody algebras
(OSAKA). After the definition of OSAKAs this is a direct
application of the classification of affine Kac-Moody algebras 
(see~\cite{Kac90}) and their involutions (see for 
example~\cite{MessaoudRousseau03} and the recent work of E.\ Heintze
achieving a complete classification from the geometric point of
view~\cite{Heintze09}.
\item[-] In chapter~\ref{chap:tame} we describe tame Fr\'echet completion. We study first the tame Fr\'echet- and $ILH$-structure on loop groups and loop algebras and investigate properties of the exponential map. Furthermore we study polar actions on certain Fr\'echet spaces. This will be needed at several places to understand the conjugacy properties of flats (part~\ref{chap:symm}). Besides that it is a prerequisite for the theory of cities~\cite{Freyn10d}.We describe tame Fr\'echet realizations of the twisted and non-twisted Kac-Moody algebras and Kac-Moody groups. 
\item[-] In chapter~\ref{chap:symm} we describe the construction of Kac-Moody symmetric spaces: we start with the definition of tame Fr\'echet symmetric spaces and construct the Kac-Moody symmetric spaces of Euclidean type, compact type and non-compact type. Furthermore we have a short look on Lie triple systems, duality and the isotropy representation.
\item[-] The last chapter is devoted to the description of some aspects of the geometry of Kac-Moody symmetric spaces. We prove, that finite dimensional flats are conjugate, review in detail the geometry of Kac-Moody symmetric spaces of type II and give some remarks about the geodesic. Furthermore we show that there are no Hermitian Kac-Moody symmetric spaces, thus pointing out a difference to the theory of finite dimensional symmetric spaces.
\end{enumerate}

\chapter{Analytic foundations}
\label{chap:frechet}

In this section we describe the analytic foundations of the theory of Kac-Moody symmetric spaces. As Kac-Moody symmetric spaces are tame Fr\'echet manifolds, we start by reviewing the theory of tame Fr\'echet structures following the presentation given by Richard Hamilton~\cite{Hamilton82}; furthermore we describe the strongly related concept of ILH- (resp.\ ILB-) structures developed by Hideki Omori~\cite{Omori97}. On the way we prove various technical results about Fr\'echet manifolds.

\section{Tame Fr\'echet manifolds}

\subsection{Fr\'echet spaces}

\noindent This introductory section collects some standard results about Fr\'echet spaces, Fr\'echet manifolds and Fr\'echet Lie groups. Further details or omitted proofs can be found in Hamilton's article~\cite{Hamilton82}.

\begin{definition}[Fr\'echet space]
\index{Fr\'echet space}
A Fr\'echet vector space is a locally convex topological vector space which is complete, Hausdorff and metrizable.
\end{definition}

\begin{lemma}[Metrizable topology]
\index{metrizable topology}
A topology on a vector space is metrizable iff it can be defined by a countable collection of seminorms.
\end{lemma}

\noindent Let us look at some easy examples:

\begin{example}[Fr\'echet spaces]~
\label{frechetexamples}
\begin{enumerate}
\item Every Banach space is a Fr\'echet space. The countable collection of norms contains just one element.
\item Let $\textrm{Hol}(\mathbb{C}, \mathbb{C})$ denote the space of holomorphic functions $f: \mathbb{C} \longrightarrow \mathbb{C}$. Let furthermore $K_n$ be a sequence of simply connected compact sets in $\mathbb C$, such that $K_n \subset K_{n+1}$ and $\bigcup K_n=\mathbb{C}$. Let $\|f\|_n:= \displaystyle\sup_{z\in K_n} |f(z)|$. Then $\textrm{Hol}(\mathbb {C}, \mathbb {C}; \|\hspace{3pt}\|_n)$ is a Fr\'echet space.
\item More generally, for every Riemann surface $S$ the sheaf of holomorphic functions carries a Fr\'echet structure which is defined similarly as in the special case $S=\mathbb{C}$ discussed in the second example.
\end{enumerate}
\end{example}
 
\begin{definition}
\index{Fr\'echet manifold}
 A Fr\'echet manifold is a (possibly infinite dimensional) manifold with charts in a Fr\'echet space such that the chart transition functions are smooth. 
\end{definition}

While it is possible to define Fr\'echet manifolds in this way, there are two strong impediments to the development of analysis and geometry of those spaces:
\begin{enumerate}
\item In general there is no inverse function theorem for smooth maps between Fr\'echet spaces. For counterexamples and examples showing features special to Fr\'echet spaces, see~\cite{Hamilton82}.
\item In general the dual space of a Fr\'echet space is in general not a Fr\'echet space. 
\end{enumerate}

Dealing with those problems  is difficult:  
take as example the space of holomorphic functions $\textrm{Hol}(\mathbb{C}, \mathbb{C})$.  $\mathbb{C}$ can be interpreted as a direct limit of the sets $K_n:=B_n(0)$ with respect to inclusion; the space of functions on a direct limit is an inverse limit; $\textrm{Hol}(\mathbb{C}, \mathbb{C})$ should thus be interpreted as the inverse limit of a sequence of function spaces $\textrm{Hol}(K_n, \mathbb{C})$, where 
\begin{displaymath}\textrm{Hol}(K_n, \mathbb{C}):=\bigcup_{K_n\subset U_{n}^m}\{\textrm{Hol}(U_{n}^m, \mathbb{C})\}\, .\end{displaymath}

 By a choice of appropriate norms on the spaces $\textrm{Hol}(K_n, \mathbb{C})$, one can give them structures as Bergmann-(or Hardy-) spaces. See for example~\cite{HKZ00} and \cite{Duren00}. Hence $\textrm{Hol}(\mathbb{C}, \mathbb{C})$ can be interpreted as inverse limit of Hilbert spaces. By category theory, the categorical dual space of an inverse limit is a direct limit and vice versa. Thus it is clear that we cannot expect the dual space of a Fr\'echet space to be Fr\'echet. It will be Fr\'echet iff the projective limit and the inductive limit coincide, that is, if they both stabilize. This is the case exactly for Banach spaces~\cite{Schaefer80}.

The solution to the first problem is based on a more refined control of the projective limits. Using this structure, the inverse function theorems on the Hilbert-(resp.\ Banach-) spaces in the sequence piece together to give an inverse function theorem on the limit space; this is the famous Nash-Moser inverse function theorem. In the next sections we will formalise those concepts. 

We have to note that there are other ways of dealing with those analytic problems. A recent example is the concept of bounded Fr\'echet geometry developed by Olaf M\"uller~\cite{Muller06}. 

The solution to the second problem consists in avoiding dual spaces. Let us remark that the theory could also be formulated by defining dual spaces as direct limit locally convex topological vector spaces.

\subsection{Tame Fr\'echet spaces}

The central problem for all further structure theory of Fr\'echet spaces is a better control of the set of seminorms. For Fr\'echet spaces $F$ and $G$ and a map $\varphi: F\longrightarrow G$ this will be done by imposing estimates similar in spirit to the concept of quasi isometries relating the sequences norms $\|\varphi(f)\|_n$ and $\|f\|_m$. We will then define tame Fr\'echet spaces as Fr\'echet spaces that are ``tame equivalent'' to some model space of holomorphic functions.
Our presentation follows the article~\cite{Hamilton82}. 

The prerequisite for estimating norms under maps between Fr\'echet spaces are estimates of the norms on the Fr\'echet space itself. This is done by a grading:

\begin{definition}[grading]
\index{grading}
Let $F$ be a Fr\'echet space. A grading on $F$ is a collection of seminorms $\{\|\hspace{3pt}\|_{n}, n\in \mathbb N_0\}$ that define the topology and satisfy
$$\|f \|_0\leq \|f\|_1 \leq \|f\|_2 \leq\| f \|_3 \leq \dots \,.$$
\end{definition}

\begin{lemma}[Constructions of graded Fr\'echet spaces]~
\begin{enumerate}
	\item A closed subspace of a graded Fr\'echet space is a graded Fr\'echet space.
	\item Direct sums of graded Fr\'echet spaces are graded Fr\'echet spaces.
\end{enumerate}
\end{lemma}

\noindent Every Fr\'echet space admits a grading. Let  $(F, \|\hspace{3pt}\|_{n, n\in \mathbb{N}})$ be a Fr\'echet space. Then 
\begin{displaymath}
\left(\widetilde{F}, \widetilde{\|\hspace{3pt}\|}_{n, n\in \mathbb{N}}\right)
\end{displaymath} 
such that $\widetilde{F}=F$ as a set and $\widetilde{\|\hspace{3pt} \|}_{n}:=\displaystyle\sum_{i=1}^n \|\hspace{3pt}\|_i$ is a graded Fr\'echet space. The Fr\'echet topologies of $F$ and $\widetilde{F}$ coincide. The existence of a grading is thus not a property of the Fr\'echet space as a metrizable topological space. It is an additional structure which is geometric in nature.

\begin{definition}[Tame equivalence of gradings]
\index{tame equivalence of gradings}
Let $F$ be a graded Fr\'echet space, $r,b \in \mathbb{N}$ and $C(n), n\in \mathbb{N}$ a sequence with values in $\mathbb{R}^+$. The two gradings  $\{\|\hspace{3pt}\|_n\}$ and $\{\widetilde{\|\hspace{3pt}\|}\}$ are called $\left(r,b,C(n)\right)$-equivalent iff 
\begin{equation*}
 \|f\|_n \leq C(n) \widetilde{\|f\|}_{n+r} \text{ and }  \widetilde{\|f\|}_n \leq C(n)\|f\|_{n+r} \text{ for all } n\geq b
\,.
\end{equation*}
They are called tame equivalent iff they are $(r,b,C(n))$-equivalent for some $(r,b,C(n))$.
\end{definition}

\noindent The following example is basic:

\begin{example}
Let $B$ be a Banach space with norm $\| \hspace{3pt} \|_B$. Denote by $\Sigma(B)$ the space of all exponentially decreasing sequences $\{f_k\}$, ${k\in \mathbb N_0}$ of elements of $B$.
On this space, we can define different gradings:

\begin{align}
\|f\|_{l_1^n} &:= \sum_{k=0}^{\infty}e^{nk} \|f_k\|_B\\
\|f\|_{l_{\infty}^n}&:= \sup_{k\in \mathbb N_0} e^{nk}\|f_k\|_B
\end{align}
\end{example}

\begin{lemma}
On the space $\Sigma(B)$ the two gradings $\|f\|_{l_1^n}$ and $\|f\|_{l_{\infty}^n}$ are tame equivalent.
\end{lemma}

\noindent For the proof~\cite{Hamilton82}.

\begin{example}
\index{space of exp.\ decreasing functions}
The space of exponentially decreasing sequences of elements in $B=\mathbb C^2$ equipped with the Euclidean norm and the space of exponentially decreasing sequences of elements in $B=\mathbb C^2$ together with the supremum-norm $\|(c_1, c_1')\|_B :=\sup(|c_1|,|c_1'|)$ are tame Fr\'echet spaces.
\end{example}

Let us give an intuitive characterization of tame spaces

\begin{lemma}
Let $B$ be a complex Banach space. The space $\Sigma(B)$ of exponentially decreasing sequences in $B$ is isomorphic to the space $\textrm{Hol}(\mathbb{C},B)$ of $B$\ndash valued holomorphic functions.
\end{lemma}

\begin{proof}
We have to prove both inclusions:
Let first $f\in \textrm{Hol}(\mathbb{C},B)$ with its Taylor series expansion:
\begin{displaymath}
f(z)=\sum_{k=0}^{\infty}f_k z^k.
\end{displaymath} 
The coefficients $f_k$ are elements of $B$ and we have to show that $(f_k)$ is an exponentially decreasing sequence. We want to use the $l_{\infty}$\ndash norm.
For the estimate of $f_k$ we need three ingredients:
\begin{enumerate}
 \item As $f(z)$ is entire the expansion $\|f(z)\|=\sum_{n=0}^{\infty}f_k z^k$ converges for all $z\in \mathbb{C}$ - hence $\sup_{z\in B(0,e^n)}|f(z)|<\infty$ of all $n$.
\item Differentiation of $f$ yields the identity: $f_k=\frac{1}{k!}f^{(k)}(0)$
\item The Cauchy inequality~\cite{Berenstein91}, 2.1.20: Let $f$ be holomorphic on $B(0,r)$. Then
\begin{displaymath}
 |f^(k)(0)|\leq k!\frac{\sup_{z\in B(0,r)}|f(z)|}{r^k}
\end{displaymath}
\end{enumerate}
Putting those ingredients together we get for $n\in \mathbb{N}$
\begin{align*}
 \sup_{k}|f_k|e^{nk}&=\sup\left|\frac{f^{(k)}(0)}{k!} \right|e^{kn}\leq\\
&\leq\sup_k\frac{e^{kn}}{k!}\left|k!\frac{\sup_{z\in B(0,e^n)}|f(z)|}{e^{nk}}\right|=\\
&=\sup_{z\in B(0,e^n)}|f(z)|\leq \infty
\end{align*}
Hence 
\begin{displaymath}
 \|f\|_{l_{\infty}^{n}}<\infty\, .
\end{displaymath}

As $f(z)$ is entire the expansion $\|f(z)\|=\sum_{n=0}^{\infty}f_k z^k$ converges for all $z\in \mathbb{C}$ yielding the result for all $n$.\\

Conversely for any exponentially decreasing sequence $(f_n)_{n\in \mathbb{N}}$, $f_n\in B$ we define the function
\begin{displaymath}
f(z):=\sum_{n=0}^{\infty}f_n z^n\, .
\end{displaymath} 
We have to prove that this defines an entire holomorphic function as we get for any $z\in B(0,e^n)$ the estimate
\begin{displaymath}
 |f(z)|=\left|\sum_k f_k z^k\right|\leq \sum_k |f_k||z^k|\leq \sum |f_k|e^{kn}<\infty
\end{displaymath}
\end{proof}

\begin{corollary}
Let $B$ be a real Banach space. Then $\Sigma(B)$ consists of all $B\otimes \mathbb{C}$\ndash valued holomorphic functions on $\mathbb{C}$ such that 
\begin{displaymath}
 \overline{f(z)}=f(\overline{z})\, .
\end{displaymath}
  
\end{corollary}

\begin{example}
 The easiest example is the choice $B=\mathbb{C}$. Then we find
\begin{displaymath}
 \Sigma(B):=\left\{(a_k)_{k\in \mathbb{N}_0}|\sum_k |a_k|e^{kn}<\infty \forall n\in \mathbb{N}\right\}
\end{displaymath}
For fixed $n$ the sequence $(a_k)$ corresponds to the coefficients of the Laurent series of a holomorphic function converging on the disc
$B(0,e^n)$. For $f(z):=\sum a_k z^k$ this is proved by the estimates
\begin{displaymath}
 |f(z)|=|\sum_k a_k z^k|\leq \sum_k |a_k||z^k|\underbrace{\leq}_{\textrm{for}\quad z\in B(0,e^n)} \sum |a_k|e^{kn}<\infty
\end{displaymath}
As this condition is satisfied by assumption for all $n\in \mathbb{N}$ we get that the function defined by the series $(a_k)_{k\in \mathbb{N}}$ is holomorphic on $\mathbb{C}$. Hence $\Sigma(B)=\textrm{Hol}(\mathbb{C},\mathbb{C})$.
\end{example}

Let $F$, $G$, $G_1$ and $G_2$  denote graded Fr\'echet spaces.

\begin{definition}[Tame linear map]
\index{tame map}
A linear map $\varphi: F\longrightarrow G$ is called $(r,b,C(n))$-tame if it satisfies the inequality
$$\|\varphi(f)\|_n \leq C(n)\|f\|_{n+r}\,.$$
$\varphi$ is called tame iff it is $(r,b,C(n))$-tame for some $(r,b, C(n))$.
\end{definition}

Let us give an example showing that tameness of maps depends on the grading. Let us look at the following variant of an example of~\cite{Hamilton82}:

\begin{example}
Let $\mathcal{P}$ be the space of entire holomorphic functions periodic with period
$2\pi i$ and bounded in each left half-plane. Define $L: \mathcal{P} \longrightarrow \mathcal{P}$ by $Lf(z) = f(2z)$.
\begin{enumerate}
\item Define first the grading on $\mathcal{P}$ by:
\begin{displaymath}
\|f \|_n = \sup{|f(z)|:\Re(z) = n}\,.
\end{displaymath}
 Then $\| Lf\|_n\leq \| f \|_{2n}$, hence $L$ is not tame.
 \item Define now the grading on $\mathcal{P}$ by
\begin{displaymath}
\|f \|_n = \sup{|f(z)|:\Re(z) = 2^n}\,.
\end{displaymath}
Then $\| Lf\|_n\leq \| f \|_{n+1}$, hence $L$ is $(1,0,1)$-tame.
 \end{enumerate}
Clearly our gradings are not tame equivalent. 
\end{example}

\begin{definition}[Tame isomorphism]
\index{tame isomorphism}
A map $\varphi:F\longrightarrow G$ is called a tame isomorphism iff it is a linear isomorphism and $\varphi$ and $\varphi^{-1}$ are tame maps.
\end{definition}

\begin{definition}[Tame direct summand]
\index{tame direct summand}
$F$ is a tame direct summand of $G$ iff there exist tame linear maps $\varphi: F\longrightarrow G$ and $\psi: G \longrightarrow F$ such that $\psi \circ \varphi: F \longrightarrow F$ is the identity.
\end{definition}

\begin{definition}[Tame Fr\'echet space]
\index{tame Fr\'echet space}
$F$ is a tame Fr\'echet space iff there is a Banach space $B$ such that $F$ is a tame direct summand of $\Sigma(B)$.
\end{definition}

\begin{lemma}[Constructions of tame Fr\'echet spaces]~
\label{constructionoftamespaces}
\begin{enumerate}
	\item A tame direct summand of a tame Fr\'echet space is tame.
	\item A cartesian product of two tame Fr\'echet spaces is tame.
\end{enumerate}
\end{lemma}

\noindent There are many different examples of tame Fr\'echet spaces. One can show that all examples described in \ref{frechetexamples} are tame Fr\'echet spaces. For proofs and additional examples see~\cite{Hamilton82}. In section \ref{Some_tame_Frechet_spaces} we will study the tame Fr\'echet spaces of holomorphic functions which we need for the construction of Kac-Moody symmetric spaces in detail.

\begin{definition}[Tame Fr\'echet Lie algebra]
\index{tame Lie algebra}
\label{tamefrechetLie algebra}
A Fr\'echet Lie algebra $\mathfrak{g}$ is tame iff it is a tame vector space and
$\textrm{ad}(X)$ is a tame linear map for every $X\in \mathfrak{g}$.
\end{definition}

\begin{example}
 Any finite dimensional Lie algebra is tame.
\end{example}

\begin{example}
The realizations of the Kac-Moody algebras $\widehat{L}(\mathfrak{g},\sigma)$ are tame Fr\'echet Lie algebras for $H^0$\ndash Sobolev  loops, smooth loops and holomorphic loops --- compare section \ref{holomorphicstructuresonkacmoodyalgebras}.
\end{example}

\noindent We now give some definitions for nonlinear tame Fr\'echet objects:

\begin{definition}
A nonlinear map $\Phi: U\subset F \longrightarrow G$ is called $(r, b, C(n))$-tame iff it satisfies the inequality 
\begin{displaymath}
\|\Phi(f)\|_{n}\leq C(n)(1+\| f\|_{n+r})\,\forall n>b\,.
\end{displaymath}
$\Phi$ is called tame iff it is $(r,b,C(n))$-tame for some $(r,b, C(n))$.
\end{definition}

\begin{example}
 Suppose $F$ and $G$ are Banach space (hence the collection of norms consists of one norm) and $\Phi:F\longrightarrow G$ is a $(r_1,b_1,C_1)$\ndash tame isomorphism with a $(r_2, b_2,C_2)$\ndash tame inverse $\Phi^{-1}$. If $b_1\geq 2$ and $b_2\geq 2$ the condition on the norms vanishes. If $b_1=r_1=b_2=r_2=0$ we get 
\begin{displaymath}
\|\Phi(f)\|\leq C_1(1+\| f\|)\,\quad\textrm{and}\quad \|\Phi^{-1}(g)\|\leq C_2(1+\| g\|)\, .
\end{displaymath}
After some manipulations we find:
\begin{displaymath}
 \frac{1}{C_2}\|f\|-1\leq \|\Phi(f)\|\leq C_1(1+\|f\|)
\end{displaymath}
which is similar to the definition of a quasi isometry~\cite{Burago01}
\end{example}

\begin{lemma}[Construction of tame maps]~
\label{constructionoftamemaps}
\begin{enumerate}
\item Let $\Phi: U\subset F \longrightarrow G_1\times G_2$ be a tame map. Define the projections $\pi_i:G_1\times G_2 \longrightarrow G_i, i=1,2$. The maps
\begin{displaymath}
\Phi_i:=\pi_i \circ \Phi: U\longrightarrow G_i
\end{displaymath}
are tame as well.
\item Let $\Phi_i: U \subset F\longrightarrow G_i, i\in \{1,2\}$ be $(r_i, b_i, C_i(n))$-tame maps. Then the map
\begin{displaymath}
\Phi:=(\Phi_1, \Phi_2): U \longrightarrow G_1 \times G_2
\end{displaymath}
is $(\max(r_1, r_2),\max(b_1, b_2), C_1(n)+C_2(n))$-tame.
\end{enumerate}
\end{lemma}

\begin{proof}~
\begin{enumerate}
\item Projections onto a direct factor are $\left(0,0,(1)_{n\in \mathbb{N}}\right)$-tame. The composition of tame maps is tame. Thus $\Phi_{i}$ is tame.
\item Let $f\in U\subset F$. \begin{align*}
\|\Phi(f)\|_n &= \|\Phi(f)\|^1_n+\|\Phi(f)\|^2_n \leq\\
              &\leq C_1(n)(1+\| f\|_{n+r_1})+ C_2(n)(1+\|f\|_{n+r_2}) \leq\\
							&\leq C_1(n)(1+\| f\|_{n+\max(r_1, r_2)})+ C_2(n)(1+\|f\|_{n+\max(r_1, r_2)})= \\
							&\leq (C_1(n)+C_2(n))(1+\| f\|_{n+\max(r_1, r_2)})
\end{align*}
for all $n\geq \max(b_1, b_2)$.
\end{enumerate}
\end{proof}

\subsection{Tame Fr\'echet manifolds}

\begin{definition}[Tame Fr\'echet manifold]
\index{tame Fr\'echet manifold}
A tame Fr\'echet  manifold is a Fr\'echet manifold with charts in a tame Fr\'echet space
such that the chart transition functions are smooth tame maps.
\end{definition}

\begin{example}
Every Banach manifold is a tame Fr\'echet manifold.
\end{example}

\begin{definition}
Let $M$ and $N$ be two tame Fr\'echet manifolds modelled on $F$ resp. $G$. A map $f: M\longrightarrow N$ is tame iff for every pair of charts $\psi_i:V_i \subset N \longrightarrow V_i'$ and $\varphi_j: U_i \subset M \longrightarrow U_i'$, the map $\psi_i \circ f \circ \varphi_j^{-1}$ is tame whenever it is defined. 
\end{definition}

For the construction of tame structures on Kac-Moody groups we need to introduce the following new notion:

\begin{definition}[Tame Fr\'echet submanifold of finite type]
\index{tame submanifold of finite type}
\label{Tamesubmanifoldoffinitetype}
Let $n\in \mathbb{N}$. A subset $M\subset F$ is a $n$-codimensional smooth submanifold of $F$ iff for every $m\in M$ there are open sets $U(m)\subset F$, $V(m)\subset G\times \mathbb{R}^n$ and a tame Fr\'echet chart $\varphi_m:U(m)\longrightarrow V(m)\subset G\times \mathbb{R}^n$ such that
$$\varphi_{M}(M\cap U(m))= G\cap V(m)\,.$$
\end{definition}

\begin{lemma}
A tame submanifold of finite type is a tame Fr\'echet manifold.
\end{lemma}

\begin{proof}~
We have to find an atlas, hence we have to show that there are charts and that the chart transition functions are smooth tame maps. Charts can be defined via the maps $\varphi: U(m)\longrightarrow V(m)$. Hence we have to show that $\varphi_{M,m}:M\cap U(m) \longrightarrow G\cap V(m)$ is a tame diffeomorphism.

Let $m\in M$, $\varphi_m: U(m)\longrightarrow V(m)$ as in definition~\ref{Tamesubmanifoldoffinitetype}.
Define the projection  $\pi_G: G\times \mathbb{R}^n \longrightarrow G$. Lemma~\ref{constructionoftamemaps} tells us that the map $\varphi_{m,G}:=\pi_{G}\circ \varphi_m: U\longrightarrow V_G\subset G$ is tame. Restricting the domain to $M\cap U$ gives us a  tame function $\varphi_{M,G}: G\times \mathbb{R}^n \longrightarrow G$. $\varphi_m$ is supposed to be a tame chart. Thus it is invertible and the inverse function is tame. Let $\varphi^{-1}_m: V(m)\longrightarrow U(m)$ be the inverse function. Restricting the domain to $V(M)\cap G$ yields a tame function $\varphi^{-1}_{m,G}: V(M)\cap G \longrightarrow U(m)\cap M$. 

We have to show that it is a smooth isomorphism. As $\varphi(m)$ is a chart, the maps $\varphi(m)$ and $\varphi^{-1}(m)$ are injective and smooth. Hence the restrictions $\varphi_{M,m}$ and $\varphi^{-1}_{m,G}$ are injective and smooth. Furthermore $\varphi_{m,M}$ is surjective onto $V(m)\cap G$. Hence also $\varphi^{-1}_{m,G}$ is surjective.
\end{proof}

\begin{lemma}
\label{mapinfrechetsubmanifold}
Let $M\subset F$ be a tame Fr\'echet submanifold of finite type. Let $H$ be a tame Fr\'echet space. A map $\varphi_M:H \longrightarrow M$ is tame if it is tame as a map $\varphi_F:H\rightarrow F$, where $\varphi_F$ is defined via the embedding $M\subset F$. 
\end{lemma}

This means, that we chose the tame structure on a tame Fr\'echet submanifold of finite type compatible with the tame structure on the vector the manifold is embedded in. We will need this result in the proof that affine Kac-Moody groups carry a tame structure. 

\begin{proof}
Let $\varphi$ be tame as a map $\varphi_F:H\rightarrow F$. Then the cocatenation $\varphi_i\circ \varphi$ is tame for any chart $\varphi$ of $M$. Thus $\varphi$ is tame as a map into $M$.   
\end{proof}

\noindent The reason for introducing the category of tame Fr\'echet spaces and tame maps is the Nash-Moser inverse function theorem. 
We cite the version of~\cite{Hamilton82}.

\begin{theorem}[Nash-Moser inverse function theorem]
\index{Theorem, Nash-Moser inverse function}
Let $F$ and $G$ be tame Fr\'echet spaces and $\Phi: U\subseteq F \longrightarrow G$ a smooth tame map. Suppose that the equation for the derivative $D\Phi(f)h=k$ has a unique solution $h=V\Phi(f)k$ for all $f\in U$ and all $k$ and that the family of inverses $V\Phi:U\times G \longrightarrow F$ is a smooth tame map. Then $\Phi$ is locally invertible, and each local inverse $\Phi^{-1}$ is a smooth tame map.
\end{theorem}

\noindent A description of this theorem a proof and some of its applications is the subject of the article~\cite{Hamilton82}.
In comparison to the classical Banach inverse function theorem the important additional assumption is that the invertibility of the differential is assumed not only in a single point $p$ but in a small neighborhood $U$ around $p$. This additional condition is necessary~\cite{Hamilton82} because
 in contrast to the Banach space situation it is not true in the Fr\'echet space case that the existence of an invertible differential in one point leads to invertibility in a neighborhood. Let us note the following result~\cite{Hamilton82}, theorem 3.1.1.\ characterizing the family of smooth tame inverses.

\begin{theorem}
 Let $L: (U \subseteq F) \times H \longrightarrow K$ be a smooth tame family of linear
maps. Suppose that the equation 
\begin{displaymath}
L(f)h = k
\end{displaymath} has a unique solution $h$ for all $f$ and
$k$ and that the family of inverses $V(f)k = h$ is continuous and tame as a map from $K$ to$H$
Then $V$ is also a smooth tame map 
\begin{displaymath}
V:(U\subseteq F)\times K\longrightarrow H.
\end{displaymath}
\end{theorem}

We will discuss this in remark~\ref{the concept of inverse limits}.

\section{Tame Fr\'echet spaces for Kac-Moody symmetric spaces}
\label{Some_tame_Frechet_spaces}

In this section we prove some technical results about the spaces of holomorphic loops on $\mathbb{C}^*$ which we will need for the construction of the loop algebra realization of Kac-Moody algebras in section~\ref{Lie_algebras_of_holomorphic_maps}. The main point is to prove that they are tame Fr\'echet spaces.

\begin{notation}~
\label{A_n}
\begin{itemize}
\item[-] Let $A_n$ denote the annulus $A_n:= \{z\in \mathbb{C}^*|e^{-n}\leq |z| \leq e^n\}$ and define the boundaries of $A_n$ by 
\begin{displaymath}
\partial A_n^+:= \{z|\hspace{3pt}|z|=e^n\}\quad \textrm{and} \quad \partial A_n^-:= \{z|\hspace{3pt}|z|=e^{-n}\},
\end{displaymath}
\item[-] Let $A_n'$ denote the set $A_n':= \{z\in \mathbb C | -n \leq \Re(z)\leq n, 0\leq \Im(z)\leq 2\pi i\}$,
\item[-] Let $B_n$ denote the disc $B_n:=\{z\in \mathbb C| \hspace{3pt} |z|\leq  e^n\}$.
\end{itemize}
\end{notation}

\begin{lemma}
\label{tame_equivalence_of_gardings_on_holcc}
Let $F:= \textrm{Hol}(\mathbb C, \mathbb C)$. The two gradings
\begin{align*}
\|f\|_{L_1^n}:=& \frac{1}{2\pi} \int_{\partial B_n}|f(z)|dz\\
\intertext{and}
\|f\|_{L_{\infty}^n}:=& \sup_{z \in B_n }\|f(z)\|
\end{align*}
are tame equivalent.
\end{lemma}

\begin{proof}
~\cite{Hamilton82}.
\end{proof}

\begin{lemma}
\label{holccistame}
$F:= \textrm{Hol}(\mathbb C, \mathbb C)$ is a tame Fr\'echet space.
\end{lemma}

\begin{proof}
~\cite{Hamilton82}.
\end{proof}

\begin{corollary}
The space $F: =\textrm{Hol}(\mathbb C,\mathbb C^n)$ is a tame Fr\'echet space.
\end{corollary}

Our aim is now to prove that $\textrm{Hol}(\mathbb{C}^*, \mathbb{C})$ is a tame Fr\'echet space. Our strategy parallels the one used by Hamilton for the proof of lemma~\ref{holccistame} and \ref{tame_equivalence_of_gardings_on_holcc}. At some points we apply the Cauchy integral formula and the maximum principle for holomorphic functions.

\begin{lemma}
Let $F:=\textrm{Hol} (\mathbb C^*, \mathbb C)$. The two gradings
\begin{align*}
\|f\|_{L_{\infty}^n}&:= \sup_{z\in A_n} |f(z)|\\
\intertext{and}
\|f\|_{L_1^n}&:= \frac{1}{2\pi} \sup \left\{\int_{\partial A_n^+}|f(z)|dz,\int_{\partial A_n^-}|f(z)|dz   \right\}
\end{align*}
are tame equivalent.
\end{lemma}

\begin{proof}~
\begin{enumerate}
	\item We show: $\|f\|_{L_{1}^n}\leq \|f\|_{L_{\infty}^n}$.
\begin{alignat*}{1}
\|f\|_{L_{1}^n}&=\frac{1}{2\pi} \sup \left\{\int_{\partial A_n^+}|f(z)|dz,\int_{\partial A_n^-}|f(z)|dz   \right\}\leq\\
&\leq \frac{1}{2\pi} \sup \left\{\int_{\partial A_n^+}\sup_{\zeta\in A_n^+}|f(\zeta)|dz,\int_{\partial A_n^-}\sup_{\zeta \in A_n^-}|f(\zeta)|dz   \right\} \leq\\
&\leq \sup \left\{ \sup_{z \in \partial A_n^+} |f(z)|, \sup_{z\in  \partial A_n^-}|f(z)| \right\} \leq\\
&\leq \sup_{z\in A_n} |f(z)|=\|f\|_{L_{\infty}^n}
\end{alignat*}

	\item We show: $\|f\|_{L_{\infty}^n}\leq \frac{2}{r} \|f\|_{L_{1}^n}$.

	To this end, we identify the space $\textrm{Hol}(\mathbb C^*, \mathbb C)$ with the space $\textrm{Hol}_{2\pi i}(\mathbb C, \mathbb C)$ of $2\pi i$ periodic functions. Under this identification $A_n$ will be identified with $A_n'$.
\end{enumerate}
\begin{alignat*}{1}
\|f\|_{L_{\infty}^n}&= \sup_{z \in A_n'} |f(z)|=\\
&= \sup_{z \in A_n'}\left|\frac{1}{2\pi i} \left\{ \int_{n+r}^{n+r+2\pi i} \frac{f(\zeta)}{z-\zeta}d\zeta - \int_{-n-r}^{-n-r+2\pi i} \frac{f(\zeta)}{z-\zeta}d\zeta \right\}\right|\leq\\
&\leq \sup_{z \in A_n'}\frac{1}{2\pi} \left\{ \int_{n+r}^{n+r+2\pi i} \left|\frac{f(\zeta)}{r}\right|d\zeta + \int_{-n-r}^{-n-r+2\pi i} \left|\frac{f(\zeta)}{r}\right|d\zeta\right\}=\\
&=\sup_{z\in A_n'}\frac{1}{2\pi r}\left\{ \int_{n+r}^{n+r+2\pi i} \left|f(\zeta)\right|d\zeta + \int_{-n-r}^{-n-r+2\pi i} \left|f(\zeta)\right|d\zeta\right\}\leq\\
&\leq \frac{2}{r}\frac{1}{2\pi} \sup\left\{\int_{n+r}^{n+r+2\pi i}|f(\zeta)|, \int_{-n-r}^{-n-r+2\pi i}|f(\zeta)| \right\}=\\
&=\frac{2}{r}\|f\|_{L_{\infty}^{n+r}}
\end{alignat*}
\end{proof}

\begin{lemma}
\label{holc*cisfrechet}
$F:= \textrm{Hol}(\mathbb C^*, \mathbb C)$ is a tame Fr\'echet space.
\end{lemma}

\begin{proof}
Let $f:= \sum_{k\in \mathbb Z} c_k z^k$ and set $f_0^+:=\sum_{k\in \mathbb N_0} c_k z^k$  and $f^-:=\sum_{-k\in \mathbb N} c_k z^k$. Clearly $f_0^+(z)$ and $f^-(\frac{1}{z})$ are holomorphic functions on $\mathbb C$.
Let
$$
 \begin{matrix}
\varphi: &\textrm{Hol}(\mathbb C^*, \mathbb C) &\longrightarrow &\Sigma(\mathbb C^2)\\
       &(f) &\mapsto &\left(\{c_k\}_{k\geq 0}, \{c_k\}_{k < 0} \right) \,.
\end{matrix}
$$

We use the notation $\widetilde{c}_k:=(c_k, c_{-k})\subset \mathbb C^2$ and use the supremum-norm on $\mathbb C^2$. 

\begin{enumerate}
\item We show: $\|f\|_{L_{\infty}^n}\leq \|\{\widetilde {c}_k\}\|$.
\begin{alignat*}{1}
\|f\|_{L_{\infty}^n}&= \sup_{z \in A_n}|f(z)|\leq\\
&\leq \sup_{z \in A_n}\left\{|f_0^+(z)|+ |f^-(z)| \right\}\leq\\
&\leq 2  \sup_{z\in A_n}\left\{\sup \left\{|f_0^+(z)|, |f^-(z)|  \right\} \right\}=\\
&= 2 \sup\left\{ \sup_{z \in A_n}|f_0^+(z)|, \sup_{z \in A_n}|f^-(\frac{1}{z})|  \right\}=\\
&= 2  \sup\left\{\|f_0^+\|_{L_{\infty}^n}, \|f^-\|_{L_{\infty}^n} \right\}\leq\\
&\leq 2   \sup \left\{ \|\{c_k\}\|_{L_1^n}, \|\{c_{-k}\|_{L_1^n}\}\right\}\leq\\
&\leq 2  \|\{\widetilde{c}_k\}\|_{L_{\infty}^n} 
\end{alignat*}

\item We show: $\|\{c_k\}\|_{L_{\infty}^n}\leq \|f\|_{L_1^n}$.
\begin{alignat*}{1}
\|\{\widetilde{c}_k\}\|_{L_{\infty}^n} &= \sup_k e^{nk}|\widetilde{c}_k|=\\
&=\sup_k e^{nk}\left|\sup\left\{c_k, c_{-k} \right\} \right|=\\
&\leq \sup \left\{\sup_k e^{nk}|c_k|, \sup_k e^{nk}|c_{-k}| \right\}\leq\\
&\leq \sup \left\{\sup_k e^{nk}\frac{1}{2\pi}\left|\int_{n}^{n+2\pi i} e^{-kz} f(z) dz \right|, \sup_k e^{nk} \frac{1}{2\pi} \left| \int_{-n}^{-n+2\pi i} e^{kz} f(z)\right|| \right\} \leq\\
&
\begin{aligned}
\leq \sup &\left\{\sup_k \frac{1}{2}\int_{n}^{n+2\pi i} | e^{nk} e^{-kn}|| e^{-k i\Im(z)}| |f(z)|dz\right.,\\
& \hphantom{\{}\left. \sup_k \frac{1}{2\pi} \int_{-n}^{-n+2\pi i} |e^{nk} e^{-kn}||e^{ki\Im(z)} ||f(z)|dz \right\}\leq
\end{aligned} \\
&\leq \sup \left\{\sup_k \frac{1}{2}\int_{n}^{n+2\pi i} |f(z)|dz, \sup_k \frac{1}{2\pi} \int_{-n}^{-n+2\pi i} |f(z)|dz \right\} \leq\\
&\leq \sup \left\{\sup_k \frac{1}{2}\int_{n}^{n+2\pi i} \sup_{\zeta \in \partial {A'_n}^+}|f(\zeta )|dz, \sup_k \frac{1}{2\pi} \int_{-n}^{-n+2\pi i} \sup_{\zeta \in \partial{A'_n}^+}|f(\zeta)|dz 
\right\} =\\
&= \sup\left\{ \sup_{\zeta \in \partial {A'_n}^+} |f(z)|,  \sup_{\zeta \in \partial {A'_n}^-} |f(z)|  \right\}\leq\\
&\leq \sup_{\zeta \in A'_n |f(\zeta)|}= \|f\|_{L_{\infty}^n}
\end{alignat*}
\end{enumerate}
\end{proof}

The following result will be needed in the proof that an affine Kac-Moody algebra is a tame Fr\'echet Lie algebra:

\begin{lemma}
\label{differentialistame}
The differential 
\begin{displaymath}
\frac{d}{dz}:\textrm{Hol}(\mathbb{C}^*,\mathbb{C}) \longrightarrow \textrm{Hol}(\mathbb{C}^*,\mathbb{C}), \quad f \mapsto f'
\end{displaymath}
is a tame linear map.
\end{lemma}

\noindent For the proof we need the following result:

\begin{lemma}
\label{coroflang}
Let $f$ be analytic on a closed disc $\overline{D}(z_0, R), R>0$. Let $0 < R_1 < R$. Denote by $\|f\|_R$ the supremum norm of $f$ on the circle of radius $R$. Then for $z \in \overline{D}(z_0, R_1)$, we have:
$$|f^{(k)}(z)|\leq \frac{k! R}{(R-R_1)^{k+1}}\|f\|_R\,.$$
\end{lemma}

\begin{proof}
This lemma is an application of Cauchy's integral formula. For details~\cite{Lang99}, pp.~131.
\end{proof}

\begin{proof}[Proof of lemma~\ref{differentialistame}]~
\begin{enumerate}
\item $\frac{d}{dz}$ is linear.
\item To prove tameness we use lemma~\ref{coroflang}. Let $z\in A_n$ and choose $R= e^{-(n+1)}(e-1)$. Thus $\overline{D}(z, R)\subset A_{n+1}$. Hence $\|f\|_R\leq \|f\|_{n+1}$. 

In the notation of lemma~\ref{coroflang} we can use that $R_1=0$ and $k=1$  and calculate in this way: 
\begin{displaymath}
\|f'(z)\|_n\leq \frac{R}{R^{2}}\|f\|_{n+1}=\frac{e^{n+1}}{e-1}\|f\|_{n+1}\,.
\end{displaymath}

This description is independent of $z$. Thus 
\begin{displaymath}
\|f'\|_n\leq \frac{e^{n+1}}{e-1}\|f\|_{n+1}\,.
\end{displaymath}

Thus the differential is $(1,0,\frac{e^{n+1}}{e-1})$-tame.
\end{enumerate}
\end{proof}

A further class of spaces that will be important for the description of twisted Kac-Moody algebras (compare definition~\ref{definitiontwistedloopalgebra}) are
spaces of holomorphic functions that satisfy some functional equation. We describe first the general setting and specialize then to the two most important cases, namely symmetric and antisymmetric holomorphic functions.

\begin{lemma}[Subspaces of $\textrm{Hol}(\mathbb{C}^*, \mathbb{C})$]
\label{subspacesofmg}
Let $k,l\in \mathbb{N}$ and $\omega=e^{\frac{2\pi i}{k}}$. The spaces 
\begin{displaymath}
\textrm{Hol}^{k,l}(\mathbb{C}^*, \mathbb C):=\{ f\in \textrm{Hol}(\mathbb{C}^*, \mathbb C)| f(\omega z)=\omega^{l}f(z)\}
\end{displaymath}
 are tame Fr\'echet spaces.
\end{lemma}

\begin{proof}
As usual for $f\in \textrm{Hol}^{k,l}(\mathbb{C}^*, \mathbb C)$, we put $\|f\|_n:=\displaystyle\sup_{z\in A_n}|f(z)|$. As  $\textrm{Hol}^{k,l}(\mathbb{C}^*, \mathbb C)$ is a closed subspace of $\textrm{Hol}(\mathbb{C}^*, \mathbb{C})$, it is a tame Fr\'echet space as a consequence of lemma~\ref{constructionoftamespaces}.
\end{proof}

Twisted affine Kac-Moody algebras arise as fixed point algebras of diagram automorphisms of non-twisted affine Kac-Moody algebras. The list of possible diagram automorphism shows that nontrivial diagram automorphisms have order $k=2$ or $k=3$~\cite{Carter05}. Thus the values of $k$ which are important for us are $k=2$ and $k=3$.

\noindent For $k=2$, lemma~\ref{subspacesofmg} has the corollaries:

\begin{corollary}[symmetric and antisymmetric loops]~
\label{holc*csaisfrechet}
\begin{itemize}
\item[-] The space $\textrm{Hol}^s(\mathbb{C}^*, \mathbb C):=\{ f\in \textrm{Hol}(\mathbb{C}^*, \mathbb C), f(z)=f(-z)\}$ is a tame Fr\'echet space. 
\item[-] The space $\textrm{Hol}^a(\mathbb{C}^*, \mathbb C):=\{ f\in \textrm{Hol}(\mathbb{C}^*, \mathbb C), f(z)=-f(-z)\}$ is a tame Fr\'echet space.
\end{itemize}
\end{corollary}

\noindent Lemmata~\ref{constructionoftamespaces} and~\ref{holc*cisfrechet} and corollary~\ref{holc*csaisfrechet} include the following result:

\begin{corollary}
\label{holc*cnisfrechet}
$F:= \textrm{Hol} (\mathbb C^*, \mathbb C^n)$, $F^s:= \textrm{Hol}^s (\mathbb C^*, \mathbb C^n)$ and $F^a:= \textrm{Hol}^a (\mathbb C^*, \mathbb C^n)$ are tame Fr\'echet spaces.
\end{corollary}

Let $V^n$ be a $n$-dimensional complex vector space. We want a tame structure on $\textrm{Hol}(\mathbb{C}^*, V^n)$. Using the corollary~\ref{holc*cnisfrechet} we get a tame Fr\'echet structure on the space $F:= \textrm{Hol} (\mathbb C^*, \mathbb C^n)$. This yields a tame structure on the spaces $\textrm{Hol}(\mathbb{C}^*, V^n)$, $\textrm{Hol}^s(\mathbb{C}^*, V^n)$ and $\textrm{Hol}^a(\mathbb{C}^*, V^n)$ only after the choice of an identification of $V^n$ with $\mathbb{C}^n$, hence a choice of a basis. As this construction uses this identification of $V^n$ with $\mathbb{C}^n$ we have to prove that the resulting tame structure is independent of it.

\noindent This is the content of the following lemma:

\begin{lemma}
Let $V^n$ be a $n$-dimensional complex vector space equipped with two norms $|\hspace{3pt},\hspace{3pt}|$ and $|\hspace{3pt},\hspace{3pt}|'$. Study the spaces $\textrm{Hol}^{k,l}(\mathbb{C}^*, V^n)$. Define gradings $\|f\|_n:=\displaystyle\sup_{z\in A_n}|f(z)|$ and $\|f\|_n':=\displaystyle\sup_{z\in A_n}|f(z)|'$. Those gradings are tame equivalent. 
\end{lemma}

\begin{proof}
Any two norms on a finite dimensional vector space are equivalent (see any book about elementary analysis, i.e.~\cite{Koenigsberger00}). Thus there exist constants $c_1$ and $c_2$ such that $|x|\leq c_1|x|'$ and $|x|'\leq c_2|x|$. Then $\|f\|_n:=\displaystyle\sup_{z\in A_n}|f(z)|\leq \displaystyle\sup_{z\in A_n}c_1|f(z)|'=c_1 \|f\|_n'$ and $\|f\|_n':=\displaystyle\sup_{z\in A_n}|f(z)|'\leq \displaystyle\sup_{z\in A_n}c_2|f(z)|' =c_2 \|f\|_n$. Thus they are tame equivalent. 
\end{proof}

\begin{corollary}
Any identification of  $\varphi:\mathbb{C}^n\longrightarrow V^n$ gives a norm $1$\ndash norm $\|v\|=\sum v_i$ on $V^n$. The tame structures induced on $\textrm{Hol}(\mathbb{C}^*, \mathbb{C})$ by two different choices are equivalent. Hence the tame structure on $V^n$ is independent of the identification of $V^n$ with $\mathbb{C}^n$.
\end{corollary}

As a consequence we have proved the following theorem:

\begin{theorem}
 The space $\textrm{Hol}(\mathbb{C}^*, V^n)$ with the family of norms 
\begin{displaymath}
\|f\|_n=\sup_{z\in A_n}|f(z)|_2
\end{displaymath}
where $|f_z|_2$ denotes the usual Euclidean norm on $\mathbb{C}^n$ is tame. 
\end{theorem}

\section{An implicit function theorem for tame maps}

The aim of this section is to prove that the inverse image of a ``regular'' point of a tame map $\varphi$ of a tame Fr\'echet space $F$ into $\mathbb{R}^n,n\in \mathbb{N}$ is a tame Fr\'echet submanifold of finite type. Following the finite dimensional blueprint we prove this result as a consequence of an implicit function theorem.

Morally inverse function theorems and implicit function theorems come in pairs.
In the literature there appear different implicit function theorems for classes of Fr\'echet spaces admitting smoothing operators (i.e.\ tame spaces). See for example~\cite{Krantz02}, \cite{Sergeraert73} and \cite{Poppenberg99}. An implicit function theorem for maps from topological vector spaces to Fr\'echet spaces using metric estimates is described in \cite{Glockner07}.

Richard Hamilton~\cite{Hamilton82} p.~212 proves the following theorem:

\begin{theorem}[Hamilton's implicit function theorem]
\index{Hamilton's implicit function theorem}
Let $F$, $G$ and $H$ be tame spaces and let $\Phi$ be a smooth tame map defined on an open subset $U$ in $F\times G$ to $H$.
$$\Phi:U\subset F\times G \longrightarrow H$$ 
Suppose that whenever $\Phi(f,g)=0$, the partial derivative $D_f\Phi(f,g)$ is surjective, and there is a smooth tame map $V(f,g)h$ linear in $h$,
$$V:(U\subset F\times G)\times H \longrightarrow F\,,$$
and a smooth tame map $Q(f,g\{h,k\})$, bilinear in $h$ and $k$, such that for all $(f,g)$ in $U$ and all $h\in H$ we have: 
$$D_f\Phi(f,g)V(f,g)h=h+Q(f,g)\{\Phi(f,g),h\}\,,$$
so that $V$ is an approximate right inverse for $D_f\Phi$ with quadratic error $Q$. Then if $\Phi(f_0,g_0)=0$ for some $(f_0, g_0\in U)$ we can find neighborhoods of $f_0$ and $g_0$ such that for all $g$ in the neighborhood of $g_0$ there exists an $f$ in the neighborhood of $f_0$ with $\Phi(f,g)=0$. Moreover the solution $f=\Psi(g)$ is defined by a smooth tame map $\Psi$.
\end{theorem}

For details and a proof we refer to \cite{Hamilton82}.
 
For the application we have in mind a much easier theorem is sufficient: 

\begin{theorem}[Implicit function theorem for tame maps]
\index{implicit function theorem}
\label{ImplicitefunctiontheoremfortameFrechetmaps}
Let $F$ be a tame space, $V\simeq W\simeq \mathbb{R}^n$ a finite dimensional vector space and $\varphi:U\subset F\times V \longrightarrow W$ a smooth tame map, such that the partial derivative $\frac{\partial}{\partial (y)}\varphi(z)$ is invertible for all $z=(x,y)\in U$. Suppose $\varphi(z_0)=0$.
Then there exist open sets $U'$ and $U''$ and a smooth tame map $\psi:U'\longrightarrow U''$ such that
$$ \varphi(x,y)=0 \Leftrightarrow y=\psi(x)\,. $$
\end{theorem}

For the proof we use the Nash-Moser inverse function theorem. Our proof follows the one described by Konrad K\"onigsberger~\cite{Koenigsberger00} for the finite dimensional implicit function theorem. The main difference is that we have to prepare the use of the Nach-Moser inverse function theorem and not the Banach inverse function theorem. 

\begin{proof}~
\begin{enumerate}
\item {\bf Prepare use of the Nash-Moser inverse function theorem} 
Study the map $\Phi:F\times V \longrightarrow F\times W$ defined by $\Phi(x,y):=(x, \varphi(x,y))$.
As Cartesian products of tame spaces are tame (lemma~\ref{constructionoftamespaces}), $F\times W$ is a tame space; furthermore $\Phi$ is a tame map (lemma~\ref{constructionoftamemaps}). 

The differential of $\Phi$ is given by
$$D\Phi(x,y)(h',h'')=\left(h', \frac{\partial}{\partial (x)}\varphi(x,y)h'+\frac{\partial}{\partial (y)}\varphi(x,y)h''\right)\,.$$

As the operator $\frac{\partial}{\partial (y)}\varphi(x,y)$ is invertible the equation $D\Phi(x,y)(h',h'')=(k',k'')$ has 
for every $k=(k',k'')$ a unique solution  $(h',h'')=V\Phi(x,y)(k',k'')$. As $\Phi$ is smooth also $D\Phi(x,y)(h',h'')$ is. By assumption $\frac{\partial}{\partial (y)}\varphi(z)$ is invertible for all $z=(x,y)\in U$. By finite dimensionality the inverse $\frac{\partial}{\partial (y)}\varphi(z)^{-1}$ is smooth. 
Hence $V\Phi(x,y)(k',k'')=(k', -\frac{\partial}{\partial (y)}\varphi(z)^{-1} \frac{\partial}{\partial (x)}\varphi(x,y) k' ,\frac{\partial}{\partial (y)}\varphi(z)^{-1}k'')$ is a smooth family of inverses.
Hence, we are now in the situation to apply the Nash-Moser inverse function theorem.

\item {\bf Apply the Nash-Moser inverse function theorem }
Application of the Nash-Moser inverse function theorem gives us open neighborhoods $U_0\subset U$ and $\widetilde{U}_0\subset F\times W$ such that $z_0\in U_0$ and $\Phi(z_0)=(x_0,0) \subset \widetilde{U}_0$ together with an inverse $\Phi^{-1}: \widetilde{U}_0 \longrightarrow U_0$ that is a smooth tame map.

Without loss of generality suppose $\widetilde{U}_0=\widetilde{U}_0'\times \widetilde{U}_0''$ such that $\widetilde{U}_0''=\widetilde{U}_0\cap W$ and $\widetilde{U}_0'=\widetilde{U}_0\cap F$. Thus we can put $\Phi^{-1}(w_1, w_2)=(w_1, \psi(w_1, w_2))$ for some $\psi: F\times W \rightarrow V$. $\varphi$ is tame as $V$ is a finite dimensional vector space.
Thus $$\Phi(x,y)=0 \Leftrightarrow y=\psi(x)\,.$$
This completes the proof.\qedhere
\end{enumerate}
\end{proof}

Thus the main difference to the finite dimensional implicit function theorem is that it is not enough to suppose invertibility in just one point and get the extension to an open neighborhood for free. As it is the case for the Nash-Moser inverse function theorem, we have always to take care of invertibility in an open neighborhood. 

This phenomenon appears again in the definition of a regular value:

\begin{definition}[tame regular value]
\index{tame regular value}
Let $\varphi: F\rightarrow \mathbb{R}^n$ be a tame map. $g\in \mathbb{R}^n$ is a regular value iff there is an open set $U\subset \mathbb{R}^n$ containing $g$, such that the differential has full rank for all $g\in U$.
\end{definition}

\noindent A direct consequence is the theorem:

\begin{theorem}
\label{submanifoldsoffrechetspace}
Let $F$ be a tame space, $W \simeq \mathbb{R}^n$ and $\varphi: F \rightarrow W$ a tame map. Let $g\in W$ be a regular value for $\varphi$.
Then $\varphi^{-1}(g)$ is a tame Fr\'echet submanifold of finite type. 
\end{theorem}

\begin{proof}
The proof follows  the finite dimensional blueprint:
\begin{enumerate} 
\item  {\bf Prepare use of theorem~\ref{ImplicitefunctiontheoremfortameFrechetmaps}} 
Let $\widetilde{U}''\subset W$ be an open set containing $g$ such that the differential of $\varphi$  has full rank on $\widetilde{U}''$. Choose now an arbitrary element $h\in \varphi^{-1}(g)$. 
 To apply the implicit function theorem~\ref{ImplicitefunctiontheoremfortameFrechetmaps} we need an open set $U\subset F$ containing $h$ and a decomposition $U = U'\times U''$ such that the partial differential $\frac{\partial}{\partial (y)}\varphi(z)$ is invertible for all $z=(x,y)\in  U$. To this end, choose a vector $v_1\in T_hV$ such that $\frac{\partial \varphi}{\partial v_1}(0)\not=0$. 
By continuity of the derivative there exists an open set $U_1\subset U$, such that $\frac{\partial \varphi}{u_1}(x)\not=0$ for all $x\in U_1$. 
Now choose $u_2$ such that $\frac{\partial \varphi}{\partial u_2}(0)\not=0$ and such that $\frac{\partial \varphi}{\partial u_1}(0)$ and $\frac{\partial \varphi}{\partial u_2}(0)$ are linearly independent. Those conditions are satisfied on an open set $U_2$. Successively we choose $n$ vectors $\{u_1,\dots, u_n\}$ that are linearly independent with non-vanishing differentials on an open set $U_0 \subset U\subset F$.
Define now $U_0''=\textrm{span}\{u_1, \dots, u_n\}\cap U_0$. Then $U_0=U_0'\times U_0''$, where $U_0'$ is a complementary subspace to $U_0''$, is the desired decomposition. Now we are in a position to apply theorem~\ref{ImplicitefunctiontheoremfortameFrechetmaps}. 

\item {\bf Construct charts} 
Theorem~\ref{ImplicitefunctiontheoremfortameFrechetmaps} gives us an implicit function $\psi: U_0' \longrightarrow U_0''$, such that $\Phi(x,y)=g$ iff $y=\psi(x)$. This defines a chart in $U_0''$. This chart is a tame map. Thus $\varphi^{-1}(g)$ is a tame submanifold.\qedhere
\end{enumerate}
\end{proof}

\begin{example}
Let $F$ be a tame Fr\'echet space. The $n$\ndash unit spheres $S_F^{n}:=\{x\in F| \|x\|_n=1 \}$ are tame Fr\'echet submanifolds. 
\end{example}

\section{Inverse limit constructions}

Following Omori~\cite{Omori97}, we define 
\begin{definition}[$ILE$-chain]
\index{$ILE$-chain}
A family of locally convex topological vector spaces $\{E, E^k; k\in \mathbb{N}\}$ is called an $ILE$-chain iff:
\begin{enumerate}
	\item  $E^k$ is continuously embedded in $E^{k+1}$; its image is dense.
	\item $E=\bigcap E^k$; the topology on $E$ is the inverse limit topology (i.e.\ the weakest topology such that the embedding $E\hookrightarrow E^k$ is continuous for every $k$). 
\end{enumerate}
If the spaces $E^k$ are Hilbert spaces, it is called an $ILH$-chain, if they are Banach spaces an $ILB$-chain.
\index{$ILB$-chain}
\index{$ILH$-chain}

\end{definition}

From now on, we suppose every $ILE$-chain to be at least $ILB$.

\begin{lemma}
Every Fr\'echet $F$ space admits an $ILB$-chain.
\end{lemma}

\begin{proof}
As any Fr\'echet space admits a grading, we can suppose $(F, \|\hspace{3pt}\|_n, n\in \mathbb{N})$ to be a graded Fr\'echet space. Define $B_n$ to be the completion of $F$ with respect to $\|\hspace{3pt}\|_n$. The grading assures that the embedding $B_{n+1} \hookrightarrow  B_n$ is bounded and thus continuous (a linear map between Banach spaces is continuous iff it is bounded~\cite{Conway90}. As $F$ is dense in $B_n$, the image of $B_{n+1}$ in $B_n$ is dense too.  
Thus $\displaystyle F =\lim_{\longleftarrow} B_n$. This inverse limit will be called the \emph{canonical $ILB$-system associated to $F$}. 
\end{proof}

The following definition is a slight variant of a concept due to Hideki Omori:  

\begin{definition}[$ILB$-regular map]
\index{$ILB$-regular map}
Let $\varphi:E\longrightarrow F$ be a linear map between two $ILB$-systems.
$\varphi$ is called $ILB$-$(r,b)$-regular iff for every $n$, $\varphi$ extends to a continuous map $\varphi_n:E^n\longrightarrow F^{n+r}$ for all $n>b$. It is called $ILB$-regular if it is $ILB$-$(0,0)$-regular. It is called weak $ILB$-regular iff there are $(r,b)$ such that it is $ILB$-$(r,b)$-regular \end{definition}

\begin{remark}~
\begin{enumerate}
\item Suppose $E\not=F$ and $\varphi:E\longrightarrow F$ is a weak $ILB$-regular map. After re-indexing the canonical $ILB$-system associated to $E$ and $F$ one may suppose $\varphi$ to be $ILB$-regular. Thus without loss of generality we can suppose every weak $ILB$-regular map $\varphi:E\longrightarrow F$ to be $ILB$-regular.
\item Caution! This is in contrast to the case $E=F$. In this case a map is $ILB$-regular iff it is $ILB$-$(0,b)$-regular.
\end{enumerate}
\end{remark}

\begin{definition}
Let $\{F, B^k; k\in \mathbb{N}\}$ and $\{\widetilde{F}, \widetilde{B}^k; k\in \mathbb{N}\}$ be two $ILB$-systems. They are called $ILB$-equivalent iff there are $ILB$-regular maps:
\begin{align*}
\varphi:\{F, B^k; k\in \mathbb{N}\}&\longrightarrow \{\widetilde{F}, \widetilde{B}^k; k\in \mathbb{N}\} \\
\nu: \{\widetilde{F}, \widetilde{B}^k; k\in \mathbb{N}\} &\longrightarrow \{F, B^k; k\in \mathbb{N}\}\,.
\end{align*}
such that $\varphi\circ \nu= Id$ and $\nu \circ \varphi = Id$.
\end{definition}

\noindent Now we investigate the relationship with tame structures: so from now on we suppose $E$ and $F$ denote tame Fr\'echet spaces.

\begin{lemma}
\label{tame=ilbregular}
A tame map $\varphi:E\longrightarrow F$ induces an $ILB$-$(r,b)$-regular map between the canonical $ILB$-system associated to $E$ and $F$.
\end{lemma}

\begin{proof}
Suppose $\varphi$ is $(r,b,C(n))$-tame. We have to show that $\varphi$ induces a family of maps $\varphi_n:E^n\longrightarrow F^{n+r}$ for all $n\geq b$. Let $e\in E^n$. Take a sequence $e_k\in E$ such that $\displaystyle\lim_{k\rightarrow \infty}e_k=e$. Define $\varphi_n(e):= \displaystyle\lim_{k\rightarrow \infty}\varphi(e_k)$. Clearly $\varphi_n(e)$ is well defined. As $\|\varphi(e_k)\|_{n+r}\leq C(n)\|e_k\|$ we get $\|\varphi_n(e)\|_{n+r}\leq C(n)\|e\|+\epsilon$ for all $\epsilon >0$. Thus $\varphi_n(e)\subset F^{n+r}$ and $\varphi_n$ is a bounded linear map and thus continuous. 
\end{proof}

\noindent Tame equivalence of gradings translates directly into a property of the associated $ILB$-systems.

\begin{definition}[Tame equivalence of $ILB$-systems]
Let $F$ be a Fr\'echet spaces with two gradings $\|\hspace{3pt}\|_n$ and $\widetilde{\|\hspace{3pt}\|}$. Let $(F,B_n; n\in \mathbb N)$ and $(F,\widetilde{B}_n; n\in \mathbb N)$ be the canonical $ILB$-systems associated to $(F,\|\hspace{3pt}\|_n)$ and $(F, \widetilde{\|\hspace{3pt}\|})$. 

$(F,B_n; n\in \mathbb N)$ and $(F,\widetilde{B}_n; n\in \mathbb N)$ are called $(r,b, C(n))$-equivalent iff there are $ILB$-$(r,b)$ regular maps $\varphi: (F,B_n) \longrightarrow (F, \widetilde{B_n})$ and $\nu:(F,\widetilde{B_n}) \longrightarrow (F, B_n)$ such that 
\begin{equation*}
  |\nu(b)|_n \leq C(n) \widetilde{|b|}_{n+r} \text{ and } \widetilde{|\varphi(b)|}_n \leq C(n)|b|_{n+r} \text{ for all } n\geq b\,.
\end{equation*} 
Two inverse limit systems are called tame equivalent if they are $(r,b, C(n))$-equivalent for some set $(r,b, C(n))$.
\end{definition}

\begin{lemma}
Let $F$ be a Fr\'echet space. Tame equivalence of gradings on $F$ and tame equivalence of the canonical $ILB$-systems associated to $F$ is equivalent.
\end{lemma}

\begin{proof}
We have to prove two directions:
\begin{enumerate} 
	\item Suppose first $(F,B^n; n\in \mathbb N)$ and $(F,\widetilde{B}^n; n\in \mathbb N)$ are $(r,b, C(n))$-equivalent. The $n$-norms $\|\hspace{3pt}\|_n$ (resp. $\widetilde{\|\hspace{3pt}\|}_n$ on $F$ are the restriction of the norm $|\hspace{3pt}|_n$ on $B_n$ (resp. $\widetilde{|\hspace{3pt}|}_n$ on $\widetilde{B}^n$) to $F$. Thus the equivalence of the two gradings is trivial.
	\item Suppose now the spaces $(F, \|\hspace{3pt}\|_n)$ and $(F, \widetilde{\|\hspace{3pt}\|}_n)$ are  $(r,b, C(n))$-tame. Using lemma \ref{tame=ilbregular} we get maps $\varphi: (F,B_n) \longrightarrow (F, \widetilde{B_n})$ and $\nu:(F,\widetilde{B}^n) \longrightarrow (F, B^n)$ that are both $(r,b, C(n))$-equivalent. This completes the proof.\qedhere
\end{enumerate}
\end{proof}

\begin{corollary}
Let $F$ be a Fr\'echet space. Tame equivalent gradings on $F$ have  $ILB$-equivalent $ILB$-systems iff the gradings are $(0,b, C(n))$-tame equivalent.
\end{corollary}

\begin{remark}
\label{the concept of inverse limits}
The concept of inverse limits sheds also new light on the Nash-Moser inverse function theorem.
It gives an intuition where the additional condition -- the existence of an open set $U$ such that the differential is invertible on $U$ -- has its origins: 

Let $F^n$ and $G^n$ be the inverse limit systems associated to $F$ and $G$. Suppose $\phi$ is $(r, 0, C(n))$-tame. Then $\phi$ induces a family of maps $\phi_{n}: F^{n} \longrightarrow G^{n+r}$ for all $n\geq b$. Suppose now, we know that for $f\in F$ the derivative is invertible. Thus by the Banach inverse function theorem we get a family of open sets $U_n\subset F^n$ such that $\phi_n|_{U_n}$ is invertible. The whole family -- and thus $\phi$ -- is invertible on the intersection $\bigcap U_n$ -- which as an intersection of infinitely many open sets may not be open. Thus the Nash-Moser inverse function theorem can be rephrased: $\phi$ is invertible if there exists an open set $U \subset \bigcap U_n$ such that $\phi_n$ is invertible on $U$.
\end{remark}

\noindent We define some nonlinear $ILB$-concepts. Our reference is again the monograph~\cite{Omori97}:

\begin{definition}[nonlinear $ILB$-mappings]
A function $f:U\cap E \longrightarrow F$ is called a $C^{\infty}-ILB$-mapping iff for each $n\in \mathbb{N}$, $f$ extends to a $C^{\infty}$ function $f^n:U \cap E^n \longrightarrow F^n$.
\end{definition}

\noindent There is an inverse function theorem and an implicit function theorem for $C^{\infty,r}$-$ILB$ normal mappings, a slight variant of $ILB$-regular maps. But as its assumptions are  too restricted for our applications we omit the details (for a statement of the theorem and a proof see the monograph~\cite{Omori97}).

\begin{definition}[$ILB$-manifold]
\index{$ILB$-manifold}
Let $\{F; F^n\}$ be an $ILB$-system. A Fr\'echet manifold $M$ modelled on $F$ is an $ILB$-manifold modeled on $\{F; F^n\}$ iff there are an open covering $\{U_{i}\}$ and a family of mappings $\varphi_{i}$ satisfying the following conditions:
\begin{enumerate} 
	\item For each $i$ there is an open subset $V_{i}\subset F^n$ and $\varphi_{i}$ is a homeomorphism of $V_{i}\cap F$ onto $U_{i}$.
	\item For any $U_{i}$, $U_{j}$ with $U_{i}\cap U_{j}\neq \emptyset$ there are open subsets  $V_{i ,j}$ and $V_{j,i}$ of $F^n$ such that the chart transition functions $\varphi_{i,j}$ are $C^{\infty}$-$ILB$-functions.
\end{enumerate}
\end{definition}

\begin{proposition}
\label{f0manifoldisilb}
A manifold $M$ modelled on the tame space $F$ is an $ILB$-manifold iff it has an atlas $\mathcal{A}$ such that its chart transition functions are $(0, b, C(n))$-tame.
\end{proposition}

\begin{proof}~
\begin{enumerate}
	\item We first show that existence of an atlas with $(0, b, C(n))$-tame chart transition function is necessary.
	Suppose there is no such atlas. Thus there is always a pair of charts: $\varphi_{i}:U_i \longrightarrow V_i$ and $\varphi_{j}:U_j \longrightarrow V_j$ such that the chart transition functions $\varphi_{i,j}$ are $(r\not=0, b, C(n))$-tame maps.

Lemma~\ref{tame=ilbregular} tells us that  it induces an $ILB$-$(r\not=0,b)$-regular map of the canonical $ILB$-systems associated to $V_i\subset F$ and $V_j \subset F$. Thus this chart transition functions are not $C^{\infty}$-$ILB$-functions as required by the definition.
	\item Now we construct the $ILB$-manifold associated to $M$. Let $\{F; F^n\}$ be the canonical $ILB$-system associated to $F$. As the chart transition functions for $M$ are supposed to be $(0, b, C(n))$-tame, they induce for each $n\in \mathbb{N}$ a continuous map $\varphi_{i,j}^n: V^n_j \longrightarrow V^n_i$. Hence, the chart transition functions are $ILB$-maps. Next we have to check that the maps  $\varphi_{i,j}^n$ are smooth functions. Using that $\varphi_{i,j}$ is supposed to be a smooth function, we get smooth tame derivatives of any order. Another application of Lemma~\ref{tame=ilbregular} tells us that they induce maps $\varphi_{i,j}^{n,(k)}$ of the canonical associated $ILB$-systems. Those maps coincide with the derivatives of the $ILB$-maps as they coincide on a dense subspace after the embedding of $F\hookrightarrow F^{n}$. 
\end{enumerate}\qedhere
\end{proof}

\begin{remark}
An $ILB$-manifold $M$ whose chart transition functions are $(0, b, C(n))$-tame has an associated series of Banach manifolds $M^n, n\in \mathbb{N}$. $M^n$ is defined to be the space of charts $U_{i}^n, i\in I$ identifying points in different charts via the chart transition functions. 
This is a manifold. The set of spaces $U_{i}^n, i\in I$ defines charts, whose chart transition functions are smooth maps. It is furthermore Hausdorff as $M$ is. 

\noindent Thus one can imagine those manifolds to be surrounded by a cloud of Hilbert manifolds. 
\end{remark}

\chapter{Algebraic foundations}
~\label{chap:alg} 
\section{Affine Kac-Moody algebras}

\subsection{Algebraic approach to Kac-Moody algebras}

The theory of Kac-Moody algebras was developed in the 60's independently by V.\ G.\ Kac~\cite{Kac68}, R.\ V.\ Moody~\cite{Moody69}, I.\ L.\ Kantor~\cite{Kantor68} and D.-N.\ Verma (unpublished) as a generalization of semisimple Lie algebras.
The classical references are the books~\cite{Kac90} and \cite{Moody95} ; the book~\cite{Kumar02} contains a short summary for the parts of the theory which are necessary for algebraic Kac-Moody groups; the book~\cite{Carter05} is a very detailed reference for semisimple and affine Kac-Moody algebras.

The main idea of Kac-Moody theory is borrowed from Cartan's classification of complex semisimple Lie algebras which is based on the encoding  of the Lie algebra structure into a matrix, later called Cartan matrix. Reversing this procedure one starts with a $n\times n$-matrix $A$ called a generalized Cartan matrix and constructs a Lie algebra $\mathfrak{g}(A)$ realizing this matrix, i.e.\ a Lie algebra such that the matrix obtained by applying Cartan's procedure is $A$.

We start with the definition of a Cartan matrix:

\begin{definition}[Cartan matrix]
\index{Cartan matrix}
A Cartan matrix $A^{n\times n}$ is a square matrix with integer
coefficients such that
\begin{enumerate}
  \item $a_{ii}=2$ and $a_{i\not=j}\leq 0 \label{geometriccondition}$,
  \item $a_{ij}=0 \Leftrightarrow a_{ji}=0\label{liealgebracondition} $,
  \item There is a vector $v>0$ (component wise) such that $Av>0$ (component wise).
\end{enumerate}
\end{definition}

\begin{definition}
\index{indecomposable Cartan matrix}
A Cartan matrix $A^{n\times n}$ is called decomposable iff $\{1, 2, \dots, n\}$ has a decomposition into two non-empty sets $N_1$ and $N_2$ such that $a_{ij}=0$ for $i\in N_1$ and $j\in N_2$. It is called indecomposable iff it is not decomposable.
\end{definition}

\begin{example}[$2\times 2$-Cartan matrices]

There are -- up to equivalence -- four different $2$-dimensional Cartan matrices:
$$\left(\begin{array}{rr} 2&0\\0&2  \end{array}\right),
\left(\begin{array}{rr} 2&-1\\-1&2  \end{array}\right),
\left(\begin{array}{rr} 2&-1\\-2&2  \end{array}\right),
\left(\begin{array}{rr} 2&-1\\-3&2  \end{array}\right).$$

\noindent The first Cartan matrix is decomposable, the other three Cartan matrices are indecomposable. Those Cartan matrices correspond to the finite dimensional simple Lie algebras
$A_1\times A_1, A_2, B_2, G_2$ respectively. 
Over the field $\mathbb{C}$ we have the equivalences: 
\begin{align*}
\mathfrak{g}(A_1\times A_1)&\cong\mathfrak{sl}(2,\mathbb{C})\times \mathfrak{sl}(2, \mathbb{C}), \\
\mathfrak{g}(A_2)&\cong \mathfrak{sl}(3,\mathbb{C}),\\
\mathfrak{g}(B_2)&\cong\mathfrak{so}(5,\mathbb{C}),\\
\mathfrak{g}(G_2)&\cong\mathfrak{g}_2(\mathbb{C}).
\end{align*}
\end{example}

\noindent Their dimensions are $6$, $8$, $10$, and $14$ respectively. 

\noindent A complete list of indecomposable Cartan matrices consists of the Cartan matrices
\index{classification of Cartan matrices} 
\begin{displaymath}
A_n, B_{n, n\geq 2}, C_{n, n\geq 3}, D_{n, n \geq 4}, E_6, E_7, E_8, F_4, G_2\,. 
\end{displaymath}

\noindent They correspond to Dynkin diagrams of the same name~\cite{Carter05}. 

\begin{definition}[affine Cartan matrix]
\index{affine Cartan matrix}
An affine Cartan matrix $A^{n\times n}$ is a square matrix with integer
coefficients, such that
\begin{enumerate}
  \item $a_{ii}=2$ and $a_{i\not=j}\leq 0 \label{geometricconditionaffine}$.
  \item $a_{ij}=0 \Leftrightarrow a_{ji}=0\label{liealgebraconditionaffine}$.
  \item There is a vector $v>0$ (component wise) such that $Av=0$.
\end{enumerate}
\end{definition}

\begin{example}[$2\times 2$-affine Cartan matrices]
There are -- up to equivalence -- two different $2$-dimensional affine Cartan matrices:
$$\left(\begin{array}{rr} 2&-2\\-2&2  \end{array}\right),
\left(\begin{array}{rr} 2&-1\\-4&2  \end{array}\right)\ .$$

\noindent They correspond to the non-twisted Kac-Moody algebra
$\widetilde{A}_1$ and the twisted algebra $\widetilde{A}_1'$ respectively. Both of those algebras are infinite dimensional.

\end{example}

The reason for the distinction between twisted and non-twisted affine Kac-Moody algebras will become apparent in section~\ref{TheloopalgebraapproachtoKacMoodyalgebras}.

\begin{enumerate}
\item The indecomposable non-twisted affine Cartan matrices are
\begin{displaymath}\widetilde{A}_n, \widetilde{B}_n, \widetilde{C}_n, \widetilde{D}_n, \widetilde{E}_6, \widetilde{E}_7, \widetilde{E}_8, \widetilde{F}_4, \widetilde{G}_2\,.\end{displaymath}
They correspond to Dynkin diagrams of the same name. In fact every non-twisted affine Cartan matrix $\widetilde{X}_l$ can be transformed into the corresponding Cartan matrix $X_l$ by the removal of the first column and the first line. This close relation is reflected in the explicit realizations; see section~\ref{TheloopalgebraapproachtoKacMoodyalgebras}. 
\item The indecomposable twisted affine Cartan matrices are
\begin{displaymath} 
\widetilde{A}_1', \widetilde{C}_l', \widetilde{B}_l^t, \widetilde{C}_l^t, \widetilde{F}_4^t, \widetilde{G}_2^t\,.
\end{displaymath}

They correspond to Dynkin diagrams of the same name. The Kac-Moody algebras associated to those diagrams can be constructed as fixed point algebras of the so-called twisted diagram automorphisms $\sigma$ of a non-twisted Kac-Moody algebra $X$. This construction suggests an alternative notation describing a twisted Kac-Moody algebra by the order of $\sigma$ and the type of $X$.
This yields the following equivalences~\cite{Carter05}, p.~451:
\begin{displaymath}
  \begin{array}{ccl}
  \widetilde{A}_1'   & \hspace{30pt} & ^2\widetilde{A}_{2}\\
  \widetilde{C}_l'   & \hspace{30pt} & ^2\widetilde{A}_{2l}, l\geq 2\\
  \widetilde{B}_l^t  & \hspace{30pt} & ^2\widetilde{A}_{2l-1}, l\geq 3\\
  \widetilde{C}_l^t  & \hspace{30pt} & ^2\widetilde{D}_{l+1}, l\geq 2\\
  \widetilde{F}_4^t  & \hspace{30pt} & ^2\widetilde{E}_6\\
  \widetilde{G}_2^t  & \hspace{30pt} & ^3\widetilde{D}_4\\ 
 \end{array}
\end{displaymath}
\end{enumerate}
In the left column we have characterized the Kac-Moody algebras via their root systems, in the right column via their realizations as fixed point algebras of a diagram automorphism.  Take for example the last line: $\widetilde{G}_2^t$ characterizes the Kac-Moody algebra as one with Weyl group $\widetilde{G}_2$; the entry in the right column tells us, that this algebra can be realized as the fixed point set of a twisted diagram automorphism of order $3$ of the nontwisted affine Kac-Moody algebra of type $D_4$. 

\index{indefinite Cartan matrix}
There are various more general types of generalized Cartan matrices:
\begin{itemize}
	\item[-] hyperbolic Cartan matrices
	\item[-] indefinite Cartan matrices
	\item[-] Borchert's Cartan matrices.
\end{itemize}

As an example let us list all $2$-dimensional hyperbolic Cartan matrices. Those matrices can be parametrized by two integers $a,b$ such that $a,b \in \mathbb{N}, a \geq b$ and $4-ab<0$:
\begin{displaymath}
A_{a,b}:=\left(\begin{array}{cc} 2&-a\\-b&2  \end{array}\right)\, .
\end{displaymath}
\noindent To all those classes of Cartan matrices one can associate Lie algebras, called their algebraic Lie algebra realizations:

\begin{definition}[realization]
\index{realization}
Let $A^{n \times n}$ be a generalized Cartan matrix. The realization of $A$, denoted $\mathfrak{g}(A)$, is the Lie algebra
$$\mathfrak{g}(A^{n\times n}) = \langle e_i, f_i, h_i, i=1,\dots, n| \textrm{R}_1, \dots, \textrm{R}_6\rangle\,,$$ where
\begin{displaymath}
  \begin{array}{cl}
    \textrm{R}_1:& [h_i,h_j]=0\,,\\
    \textrm{R}_2:& [e_i,f_j]=h_i \delta_{ij}\,,\\
    \textrm{R}_3:& [h_i, e_j]=a_{ji}e_j\,,\\
    \textrm{R}_4:& [h_i, f_j]=-a_{ji}f_j\,,\\
    \textrm{R}_5:& (\textrm{ad} e_i)^{1-a_{ji}}(e_j)=0 \hspace{3pt}\quad \textrm{for}\ i\not= j\,,\\
    \textrm{R}_6:& (\textrm{ad} f_i)^{1-a_{ji}}(f_j)=0\hspace{3pt}\quad\textrm{for}\ i\not= j\,.\\
  \end{array}
\end{displaymath}
\end{definition}

\noindent The realization defines a bijection between generalized Cartan matrices and (possibly infinite dimensional) complex Lie algebras. Via this bijection Cartan matrices corresponds to complex simple Lie algebras, affine Cartan matrices correspond to complex affine Kac-Moody and hyperbolic Cartan matrices correspond to hyperbolic Kac-Moody algebras.

For a general description see~\cite{Kac90}, \cite{Moody95}, and references therein. For hyperbolic Kac-Moody algebras see \cite{CCCMNNP10}, \cite{Feingold80}, \cite{FeingoldFrenkel83}, and references therein. Let us note that our theory does not allow for a direct generalization to more general classes of Kac-Moody algebras as it is crucially based on the use of loop group realizations. As of this writing there are no explicit similar realizations known for more general Kac-Moody algebras. Infinite dimensional symmetric spaces might exist for symmetrizable Kac-Moody algebras; in supergravity and string theory coset spaces are studied that from a purely formal point of view resemble symmetric spaces of type $III$ associated to the Kac-Moody groups $E_{10}$ and $E_{11}$. But for those coset spaces there is no description as a symmetric space or even only as a manifold known~\cite{DamourHenneaux01}, \cite{DamourHenneauxNicolai02}, \cite{West01}. 
\index{$E_{10}$}
\index{$E_{11}$}

If a (generalized) Cartan matrix $A^{n+m\times n+m}$ is decomposable into the direct sum of two Cartan matrices $A^{n\times n}$ and $A^{m\times m}$ then a similar decomposition holds for the realizations:
\begin{displaymath}
\mathfrak{g}(A^{n+m\times n+m})=\mathfrak{g}(A^{n\times n})\oplus\mathfrak{g}(A^{m\times m})\,.
\end{displaymath}

This is a crucial observation as this decomposition has its counterparts in all classes of objects associated to those Lie algebras.  
In the special case of finite semisimple Lie algebras this is especially appealing: The following decompositions correspond to each other:

\index{decomposition of realization}
\begin{itemize}
\item[-] Cartan matrix $ \longrightarrow $ sum of indecomposable Cartan matrices,
\item[-] complex semisimple Lie algebra $ \longrightarrow $ direct product of complex simple Lie algebras (ideals),
\item[-] complex semisimple Lie group $ \longrightarrow $ direct product of complex simple Lie groups,
\item[-] simply connected, complete Riemannian manifold  $ \longrightarrow $ simply connected, complete Riemannian manifold with irreducible holonomy,
\item[-] Riemannian symmetric space $ \longrightarrow $ direct sum of irreducible Riemannian symmetric spaces.
\end{itemize}

\noindent In the infinite dimensional situation of Kac-Moody symmetric spaces the situation is slightly more complicated. The central problem is that the direct product construction is very ill-adapted to the geometric situation of Kac-Moody symmetric spaces, as it does not preserve the Lorentz-structure of the spaces. Thus it will be necessary to review this construction in order to define a different concept of composition such that the products are still Lorentzian. In order to distinguish between those two concepts we will refer in any case of ambiguity to Kac-Moody algebras in the usual sense as ``algebraic affine Kac-Moody algebras'' in contrast to ``geometric affine Kac-Moody algebras''.

\subsection{The loop algebra approach to Kac-Moody algebras}
\label{TheloopalgebraapproachtoKacMoodyalgebras}

The loop algebra realization of (algebraic) affine Kac-Moody algebras is developed in the books \cite{Kac90} and \cite{Carter05} from an algebraic point of view. We follow the more geometric approach to loop algebra realizations of Kac-Moody algebras described in the article~\cite{Heintze09}. Nevertheless the class of ``geometric affine Kac-Moody algebras'' we construct is more general. 

\noindent Let $\mathfrak{g}$ be a finite dimensional reductive Lie algebra over $\mathbb{F}=\mathbb{R}$ or $\mathbb{C}$. Hence $\mathfrak{g}$ is a direct product of a semisimple Lie algebra $\mathfrak{g}_s$ with an Abelian Lie algebra $\mathfrak{g}_a$. Let furthermore $\sigma_s \in \textrm{Aut}(\mathfrak{g}_s)$ denote an automorphism of finite order of $\mathfrak{g}_s$ such that the restriction of $\sigma$ to any simple factor $\mathfrak{g}_i$ of $\mathfrak{g}$ is an automorphism of $\mathfrak{g}_i$ and $\sigma|_{\mathfrak{g}_a}=\textrm{Id}$. If $\mathfrak{g}_s$ is a Lie algebra over $\mathbb{R}$ we suppose it to be of compact type. \index{$L(\mathfrak{g}, \sigma)$}
\begin{displaymath}
L(\mathfrak g, \sigma):=\{f:\mathbb{R}\longrightarrow \mathfrak{g}\hspace{3pt}|f(t+2\pi)=\sigma f(t), f \textrm{ satisfies some regularity conditions}\}\label{abstractkacmoodyalgebra}\,.
\end{displaymath}
We use the notation $L(\mathfrak g, \sigma)$ to describe in a unified way constructions that can be realized with explicit constructions of loop algebras satisfying various regularity conditions. The loops may be smooth, real analytic, (after complexification) holomorphic on $\mathbb{C}^*$, holomorphic on an annulus $A_n\subset \mathbb{C}$ or algebraic. If we discuss loop algebras of a fixed regularity we use other precise notations: $M\mathfrak{g}$, $L_{alg}\mathfrak{g}$, $A_n\mathfrak{g}$, $L_{\textrm{alg}}\mathfrak{g}$, $\dots$.

\begin{definition}[Geometric affine Kac-Moody algebra]
\index{geometric affine Kac-Moody algebra}
\label{geometricaffinekacmoodyalgebra}
The geometric affine Kac-Moody algebra associated to a pair $(\mathfrak{g}, \sigma)$ is the algebra:
$$\widehat{L}(\mathfrak g, \sigma):=L(\mathfrak g, \sigma) \oplus \mathbb{F}c \oplus \mathbb{F}d\,,$$
equipped with the lie bracket defined by:
\begin{alignat*}{1}
  [d,f]&:=f'\, ,\\
  [c,c]=[d,d]&:=0\, ,\\
	[c,d]=[c,f]&:=0\, ,\\
  [f,g]&:=[f,g]_{0} + \omega(f,g)c\,.
\end{alignat*}

Here $f\in L(\mathfrak g, \sigma)$ and $\omega$ is a certain antisymmetric $2$-form on $M\mathfrak g$, satisfying the cocycle condition.
\end{definition}

\index{regularity condition for geometric affine Kac-Moody algebra}
Explicit realizations using $k$-times differentiable, smooth, analytic, holomorphic or algebraic loops are common~\cite{HPTT}.
\begin{enumerate}
\item For holomorphic or algebraic loops we define \begin{displaymath}\omega(f,g):=\textrm{Res}(\langle f, g' \rangle)\,.\end{displaymath}
\item For integrable loops we use \begin{displaymath}\omega(f,g)=\frac{1}{2\pi}\int_{0}^{2\pi}\langle f, g'\rangle dt\,.\end{displaymath}
\end{enumerate}
By the formula for residua of holomorphic functions~\cite{Berenstein91} those two definitions of $\omega$ coincide for polynomial functions and holomorphic functions. All functional analytic completions which we study, contain the space of polynomials as a dense subspace. Hence we can use both formulations equivalently.

Let us reformulate this for nontwisted affine Kac-Moody algebras in terms of functions on $\mathbb{C}^*$: Embedding the interval $[0;2\pi]$ into the unit circle of $\mathbb{C}^*$, the parameter $t$ gets replaced by the complex parameter $z:=e^{it}$, which extends to $\mathbb{C}^*$.

\begin{definition}
The complex geometric affine Kac-Moody algebra associated to a pair $(\mathfrak{g}, \sigma)$ is the algebra:
$$\widehat{M\mathfrak g}:=M\mathfrak g \oplus \mathbb{C}c \oplus \mathbb{C}d\,,$$
equipped with the lie bracket defined by:
\begin{alignat*}{1}
  [d,f(z)]&:=izf'(z)\,,\\
  [c,c]=[d,d]&:=0\,,\\
  [c,d]=[c,f]&:=0\,,\\
  [f,g](z)&:=[f(z),g(z)]_{0} + \omega(f(z),g(z))c\,.
\end{alignat*}
\end{definition}

As $\frac{d}{dt}e^{itn}=ine^{itn}=inz^n=iz\frac{d}{dz}z^n$ both definitions coincide.

\begin{definition}[semisimple geometric affine Kac-Moody algebra]
\index{semisimple geometric Kac-Moody algebra}
A geometric affine Kac-Moody algebra $\widehat{L}(\mathfrak{g}, \sigma)$ is called semisimple if $\mathfrak{g}$ is semisimple.
A geometric affine Kac-Moody algebra $\widehat{L}(\mathfrak{g}, \sigma)$ is called simple if $\mathfrak{g}$ is simple.
\end{definition}

\begin{definition}[derived algebra]
\index{derived algebra}
The derived geometric affine Kac-Moody algebra associated to a pair $(\mathfrak{g}, \sigma)$ is the algebra:
\begin{displaymath}\widetilde{L}(\mathfrak g, \sigma):=L(\mathfrak g, \sigma) \oplus \mathbb{F}c\,,\end{displaymath}
with the lie bracket inherited from the affine geometric Kac-Moody algebra.
\end{definition}

The name ``derived algebra'' is due to the fact that for semisimple geometric affine Kac-Moody algebras we have the equality
\begin{displaymath}\widetilde{L}(\mathfrak g, \sigma)\hspace{3pt}=\hspace{3pt}\left[\widehat{L}(\mathfrak g, \sigma), \widehat{L}(\mathfrak g, \sigma)\right]\,.\end{displaymath}

If $\mathfrak{g}$ is a simple Lie algebra then the associated algebraic and geometric Kac-Moody algebras coincide up to completion:

\begin{lemma}
Suppose $\mathfrak {g}$ is simple.
The realization of $\widehat{L}(\mathfrak g, \sigma)$ with algebraic loops $\widehat{L_{alg}\mathfrak{g}}^{\sigma}$ is a simple, algebraic affine Kac-Moody algebra.
\end{lemma}

This lemma is a restatement of the loop algebra realization of affine Kac-Moody algebras described in~\cite{Carter05}.
If one chooses $\widehat{L}(\mathfrak{g}, \sigma)$ to be a realization of a more general regularity class of loops than algebraic loops then $\widehat{L}(\mathfrak g, \sigma)$ can be described as a completion of $\widehat{L_{alg}\mathfrak{g}}^{\sigma}$ with respect to some seminorms or some set of semi-norms. 

Following \cite{PressleySegal86}, p.168 we define two Lie algebra representations $R$ and $R'$ of a Lie algebra $\mathfrak{g}$ to be essentially equivalent to $R'$ if there is an injective map $\varphi: R\longrightarrow R'$ with dense image, equivariant with respect to $\mathfrak{g}$.
\index{essential equivalence}

\begin{lemma}
As representation of $\widehat{L_{alg}\mathfrak{g}}^{\sigma}$ we have:
$\widehat{L_{alg}\mathfrak{g}}^{\sigma}$ is essentially equivalent to $\widehat{L}(\mathfrak{g}, \sigma)$ for all functional analytic regularity conditions.
\end{lemma} 

\begin{proof}
As $\widehat{L}(\mathfrak{g}, \sigma)$ is constructed as a completion of $\widehat{L_{alg}\mathfrak{g}}^{\sigma}$, the embedding $\varphi:\widehat{L_{alg}\mathfrak{g}}^{\sigma}\longrightarrow \widehat{L}(\mathfrak{g}, \sigma)$ has a dense image. Furthermore it is injective and $\widehat{L_{alg}\mathfrak{g}}^{\sigma}$-equivariant.
\end{proof}

\begin{remark}
We remark that one can define a formal completion for algebraic Kac-Moody algebras. From a purely algebraic point of view, this completion has the disadvantage that it does not preserve the symmetries of the Kac-Moody algebra. For example the symmetry between the generators $e_i$ and $f_i$ is destroyed. From the point of view of polynomial realizations this means that the symmetry between $t$ and $t^{-1}$ is broken. It is thus not preserved under involutions mapping $t$ onto $t^{-1}$. From the point of view of holomorphic functions, the problem is that the two poles at $0$ and at $\infty$ are not treated at the same footing.  This causes serious problems with the definition of Kac-Moody symmetric spaces as involutions of second type are ruled out (for further details and proofs see for example~\cite{Kumar02}).
\end{remark}

We will now describe the splitting behaviour of geometric Kac-Moody algebras.
Let us distinguish two concepts:

\begin{definition}[indecomposable Lie algebra]
\index{indecomposable Lie algebra}
A Lie algebra $\mathfrak{g}$ is called indecomposable if it does not admit a decomposition $\mathfrak{g}=\mathfrak{g}_1\oplus \mathfrak{g}_2$ into two ideals $\mathfrak{g}_1$ and $\mathfrak{g}_2$.
\end{definition}

\begin{definition}[irreducible Kac-Moody algebra]
\index{irreducible Kac-Moody algebra}
A Kac-Moody algebra $\mathfrak{g}$ is called irreducible if there is no Kac-Moody algebra $\mathfrak{h}$ such that  $\mathfrak{h}\subset \mathfrak{g}$  and $\mathfrak{h}$ is an ideal in $\mathfrak{g}$.
\end{definition}

\begin{example}
Suppose $\mathfrak{g}=\mathfrak{g}_1\oplus \mathfrak{g}_2$ with $\mathfrak{g}_{i, i=\{1,2\}}$ semisimple. Then $\widehat{L}(\mathfrak{g},\sigma)$ is reducible, as $\widehat{L}(\mathfrak{g}_1,\sigma_1)$ and $\widehat{L}(\mathfrak{g}_2,\sigma_2)$ are Kac-Moody subalgebras. Nevertheless it is indecomposable as the (unique) central direction spanned by $c$ cannot be in all components but only in one component of a decomposition.
\end{example}

We start with the investigation of the loop algebra part.
\begin{lemma}
Let $\mathfrak{g}=\mathfrak{g}_a\oplus \mathfrak{g}_s$ be a reductive Lie algebra and $\sigma=(\sigma_a, \sigma_s)$ an involution. Then $L(\mathfrak{g}_a, \sigma_a)$ and $L(\mathfrak{g}_s, \sigma_s)$ are ideals in $L(\mathfrak{g}_s, \sigma_s)$.
\end{lemma}

\begin{proof}
Decompose $f\in (\mathfrak{g}, \sigma)$ into the component functions $f_a\in (\mathfrak{g}_a, \sigma_a)$ and $f_s\in (\mathfrak{g}_s, \sigma_s)$.
\end{proof}

\begin{lemma}
Let $\mathfrak{g}$ be semisimple and suppose $(\mathfrak{g}, \sigma):=\bigoplus_i (\mathfrak{g_i}, \sigma_i)$.
Then \begin{displaymath}L(\mathfrak{g}, \sigma)=\bigoplus_i L(\mathfrak{g}_i, \sigma_i)\,.\end{displaymath}
Each algebra $L(\mathfrak{g}_i, \sigma_i)$ is an ideal in $L(\mathfrak{g}, \sigma)$.
\end{lemma}

\begin{proof}
Study the decomposition of any loop $f\in L(\mathfrak{g}, \sigma)$ into its component loops $f_i\in L(\mathfrak{g}_i, \sigma_i)$. This yields the direct product decomposition. As the bracket is defined pointwise each $L(\mathfrak{g}_i, \sigma_i)$ is an ideal in $L(\mathfrak{g}, \sigma)$.\qedhere
\end{proof}

 Nevertheless on the level of Kac-Moody algebras the behaviour is different:  to get a direct sum decomposition of a Kac-Moody algebra into indecomposable ones we would need an extension $\mathbb{C}c_i\oplus \mathbb{C}d_i$ for every simple factor $\mathfrak{g}_i$ in the decomposition of the underlying Lie algebra $\mathfrak{g}$:

From the analogy with the affine algebraic Kac-Moody algebra associated to $(\mathfrak{g}, \sigma)$ we would expect $\widehat{L}(\mathfrak{g}, \sigma)$ to be the algebra
$$\bigoplus_{i=1}^n \widehat{L}(\mathfrak{g}_i, \sigma_i)\,.$$

This algebra has a $n$-fold extension. It is exactly the Kac-Moody algebra $\mathfrak{g}(A)$ we would get by completing the realization of a generalized Cartan matrix $A$ which is the  sum of affine Cartan matrices $A_i$ such that $A_i$ is the Cartan matrix of $\widehat{L}(\mathfrak{g}_i, \sigma_i)$. As we will see later the $Ad$-invariant scalar product associated to this extension has index $n$. Nevertheless there is a second possibility to define the extension, which yields an $Ad$-invariant Lorentz scalar product (i.e.\ index 1) -- we defined  $\widehat{L}(\mathfrak{g}, \sigma)$ to be an algebra with only one such extension.  
Thus we trivially get

\begin{lemma}
Suppose $\mathfrak{g}$ is a semisimple, non-simple Lie algebra. Then
$$\widehat{L}(\mathfrak{g}, \sigma)\not= \bigoplus_i \widehat{L}(\mathfrak{g}_i, \sigma_i)\,.$$
\end{lemma}

\begin{proof}
The center of $\widehat{L}(\mathfrak{g}, \sigma)$ is $1$-dimensional. In contrast the dimension of the center of $\bigoplus_i \widehat{L}(\mathfrak{g}_i, \sigma_i)$ is equivalent to the number of simple factors of $\mathfrak{g}$.
\end{proof}

\noindent Thus geometric Kac-Moody algebras do not split into a direct sum of simple algebras. 

\begin{lemma}[Splitting of geometric Kac-Moody algebras]
\index{splitting of geometric affine Kac-Moody algebra}
Let as above $\mathfrak{g}:=\mathfrak{g}_a\oplus\bigoplus_i \mathfrak{g}_i$ be the decomposition of a reductive Lie algebra into its abelian factor $\mathfrak{g}_a$ and its simple factors $\mathfrak{g}_i$. Let $\widehat{L}(\mathfrak{g}, \sigma)$ be the associated Kac-Moody algebra. Then 
\begin{enumerate}
\item $\widetilde{L}(\mathfrak{g}_a, \sigma_a)$ and $\widetilde{L}(\mathfrak{g}_i, \sigma_i)$ are ideals in $\widehat{L}(\mathfrak{g}, \sigma)$.
\item Let $(\mathfrak{g}_i, \sigma_i)$ be a simple factor of $\mathfrak{g}$. $\widetilde{L}_{alg}(\mathfrak{g}_i, \sigma_i)\oplus \mathbb{F}d$ is an indecomposable Kac-Moody subalgebra.
\item There is a Lie algebra homomorphism
\begin{displaymath}\varphi:\left(\widetilde{L}(\mathfrak{g}_a, \sigma_a)\oplus\bigoplus_{i=1}^n \widetilde{L}(\mathfrak{g}_i, \sigma_i)\right)\longrightarrow \widetilde{L}(\mathfrak{g}, \sigma)\,,\end{displaymath}
defined by $\varphi(f_1, r_{c_1}), \dots, (f_n, r_{c_n})=(f_1, \dots, f_n,  r_c=\sum r_{c_i})$.

\item $\widehat{L}(\mathfrak{g}, \sigma):=\left(\bigoplus \widetilde{L}(\mathfrak{g}_i, \sigma_i)\right)/\textrm{Ker}(\varphi)\oplus \mathbb{F}d$.
This defines an exact sequence
\begin{displaymath}1\longrightarrow \mathbb{F}^{n-1}\longrightarrow \left(\bigoplus_{i=1}^n \widetilde{L}(\mathfrak{g}_i, \sigma_i)\right) \longrightarrow \widetilde{L}(\mathfrak{g}, \sigma)\longrightarrow 1\,.\end{displaymath}
\end{enumerate}
\end{lemma}

\begin{proof}~
\begin{enumerate}
\item To check 1., it is sufficient to verify the closedness of the bracket operation: let $f_i\in \widetilde{L}(\mathfrak{g}_i, \sigma_i)$, and $g+\mu d\in \widehat{L}(\mathfrak{g}, \sigma)$.  Then $[f_i, g+\mu d]= [f_i, g]-\mu f_i'$. $f_i'$ is in $\widetilde{L}(\mathfrak{g}_i, \sigma_i)$, the same is true for $[f_i, g]$, as it is true pointwise for elements in $\mathfrak{g}$.

\item 2.\ follows directly from the definition.

\item $\varphi|_{L(\mathfrak{g}, \sigma)}$ is an isomorphism. So we are left with checking the behaviour of the extensions.
\begin{alignat*}{1}
\hspace{3pt}&\hspace{3pt}\varphi[((f_1, r_{c_1}), \dots, (f_n, r_{c_n})); ((\bar{f}_1, \bar{r}_{c_1}), \dots, (\bar{f}_n, \bar{r}_{c_n}))] = \\
= &\hspace{3pt}\varphi[f_1, \dots, f_n; \bar{f}_1, \dots, \bar{f}_n]_0+\varphi\left(\sum_{i=1}^n\omega_i\left(f_i; \bar{f}_i'\right)c_i\right)=\\
= &\hspace{3pt}[\varphi(f_1, \dots, f_n); \varphi(\bar{f}_1, \dots, (\bar{f}_n))]_0+\sum_{i=1}^n \omega\left(\varphi(f_1, \dots, f_n), \varphi(\bar{f}_1, \dots , \bar{f}_n)\right)c=\\
= &\hspace{3pt}[\varphi((f_1, r_{c_1}), \dots, (f_n, r_{c_n}));\varphi((\bar{f}_1, \bar{r}_{c_1}), \dots, (\bar{f}_n, \bar{r}_{c_n}))]
\end{alignat*}
\item is a consequence of 3.\ .\qedhere
\end{enumerate}

\end{proof}

\begin{remark}
One can construct products of geometric affine Kac-Moody algebras. If a generalized geometric affine Kac-Moody algebra is the product of $n$ geometric affine Kac-Moody algebras, it has index $n$. This constructions corresponds to a direct product of Kac-Moody symmetric spaces. Thus from now on we restrict our attention to geometric affine Kac-Moody algebras.
\end{remark}

Now we want to prove a similar splitting theorem for automorphisms of complex geometric affine Kac-Moody algebras.

The following lemma allows to restrict the study to the loop algebras:

\begin{samepage}
\begin{lemma}~ 
\begin{enumerate}
\item Any automorphism $\widehat{\varphi}:\widehat{L}(\mathfrak{g}_{\mathbb {R}},\sigma) \longrightarrow \widehat{L}(\mathfrak{g}_{\mathbb {R}}, \sigma)$ induces an automorphism of the derived algebras $\widetilde{\varphi}:\widetilde{L}(\mathfrak{g}_{\mathbb {R}}, \sigma)\longrightarrow \widetilde{L}(\mathfrak{g}_{\mathbb {R}}, \sigma)$.
\item Any automorphism $\widehat{\varphi}:\widehat{L}(\mathfrak{g}_{\mathbb {R}},\sigma) \longrightarrow \widehat{L}(\mathfrak{g}_{\mathbb {R}}, \sigma)$ induces an automorphism of the loop algebras $\varphi:L(\mathfrak{g}_{\mathbb {R}}, \sigma)\longrightarrow L(\mathfrak{g}_{\mathbb {R}}, \sigma)$.
\end{enumerate}
\end{lemma} 
\end{samepage}

\begin{proof}
see~\cite{Heintze09}.
\end{proof}

\noindent For loop algebras of real or complex type we find:

\begin{lemma}
Let $\mathfrak{g}=\mathfrak{g}_s \oplus \mathfrak{g}_a$ and $(\mathfrak{g_s}, \sigma):= \bigoplus_i (\mathfrak{g_i}, \sigma_i)$, where $(\mathfrak{g_i}, \sigma_i)$ is a simple Lie algebra of the compact type.
Let $L(\mathfrak g, \sigma)$ be the associated loop algebra and $\varphi$ an automorphism of $L(\mathfrak g, \sigma)$. 
Then
\begin{enumerate}
\item $\varphi\left(L(\mathfrak{g}_a, \sigma)\right)=L(\mathfrak{g}_a, \sigma)$ and $\varphi\left(L(\mathfrak{g}_s, \sigma)\right)=L(\mathfrak{g}_s, \sigma)$.
\item $L(\mathfrak{g}_s, \sigma)$ decomposes under the action of $\varphi$ into $\varphi$-invariant ideals of two types:
   \begin{enumerate}
			\item Loop algebras of simple Lie algebras $L(\mathfrak{g}_i,\sigma_i)$ together with an automorphism $\varphi_i$ (called ``type $I$-factors''),
			\item Loop algebras of products of simple Lie algebras $\mathfrak{g_i}=\oplus_{i=1}^m\mathfrak{g_i}'$ together with an automorphism $\varphi_i$ of order $n$, cyclically interchanging the $m$-factors (called ``type $II$-factors''). In this case $\frac{n}{m}=k\in \mathbb{Z}$, and $\sigma$ induces an automorphism of order $k$ on each simple factor.
   \end{enumerate}
\end{enumerate}
\end{lemma}

\begin{proof}
Let $\mathfrak{g}=\mathfrak{g}_a \oplus \mathfrak{g}_s$. Each function $f:\mathbb{R}\longrightarrow \mathfrak{g}$ has a unique decomposition $f:=(f_a, f_s)$, such that $f_s: \mathbb{R}\rightarrow \mathfrak{g}_s$ and $f_a: \mathbb{R}\rightarrow \mathfrak{g}_a$. As $f(t+2\pi)=\sigma f(t)$ is equivalent to $f_a(t+2\pi)=\sigma f_a(t)=f_a(t)$ and $f_s(t+2\pi)=\sigma_s f_s(t)$ and $\sigma|_{\mathfrak{g}_a}=Id$ this induces the decomposition: $L(\mathfrak{g}, \sigma)=L (\mathfrak{g}_a, Id) \oplus L(\mathfrak{g}_s, \sigma_s)$. 

\noindent Let now $\mathfrak{g}_s=\bigoplus_{i=1}^m\mathfrak{g}_i$ be a decomposition of $\mathfrak{g}_s$ such that  
\begin{enumerate}	
 \item $\mathfrak{g}_i$ is invariant under $\varphi = \varphi|_{\mathfrak{g}_i}$.
 \item There is no decomposition $\mathfrak{g}_i=\mathfrak{g}'_i \oplus \mathfrak{g}''_i$ such that  $\varphi|_{\mathfrak{g}'_i}$ and $\varphi|_{\mathfrak{g}''_i}$ are automorphisms and $\mathfrak{g}'_i$ and $\mathfrak{g}''_i$ are invariant under the bracket operation.
\end{enumerate}

Again $f_s$ splits into $m$ component functions $f_s=(f_1, \dots, f_m)$ and the compatibility condition $f_s(t+2\pi)=\sigma f_s(t)$ is equivalent to the $m$ compatibility conditions $f_i(t+2\pi)=\sigma f_i(t), i=1, \dots, m$.

\noindent There are now two cases:
\begin{enumerate}
\item Suppose first, $\mathfrak{g}_i$ is simple. Then $\varphi_i$ is an involution of $L(\mathfrak{g}_i, \sigma_i)$. The pair 
\begin{displaymath}\left(L(\mathfrak{g}_i, \sigma_i), \varphi_i\right)
\end{displaymath} is of type $I$. 
 The finite order automorphisms of simple affine geometric Kac-Moody algebras are completely classified~\cite{Heintze09}.
\item Suppose now $\mathfrak{g}_i$ is not simple and $\varphi|_{L(\mathfrak{g}_i,\sigma_i)}$ is of order $n$. There is a decomposition $\mathfrak{g}_i:=\bigoplus_j \mathfrak{g}_i^j$ such that $\mathfrak{g}_i^j$ is a simple Lie algebra. As there is no subalgebra $L(\mathfrak{h},\sigma_i|_{\mathfrak{h}})\subset L(\mathfrak{g}_i,\sigma_i)$ which is both invariant under $\varphi_{\mathfrak{g}_i}$ and an ideal in $L(\mathfrak{g}, \sigma)$, we find that all $(\mathfrak{g}_i^j, \sigma_i^j)$ are of the same type and the algebras $L(\mathfrak{g}_i^j, \sigma_i^j)$ are permuted by $\varphi_{\mathfrak{g}_i}$. Thus the number of those factors, denoted $m$, is a divisor of $n=km$.
We get $$L(\mathfrak{g}_i, \sigma_i):=\bigoplus_{j=1}^m L(\mathfrak{g}_j,\sigma_j)\,.$$
$\varphi$ induces an automorphism $\bar{\varphi}$ of order $k$ on each simple factor. $\bar{\varphi}$ is again a finite order automorphism of a simple geometric affine Kac-Moody algebra. \qedhere
\end{enumerate}
\end{proof}

\index{normal form of automorphism of loop algebra}
Using the classification result of E.\ Heintze and C.\ Gro\ss~\cite{Heintze09} we know that every automorphism of the loop algebra $L(\mathfrak{g},\sigma)$ is of standard form:
\begin{displaymath}\varphi\left( f(t)\right)=\varphi_t f(\lambda(t))\end{displaymath}
Here $\varphi(t)$ denotes a curve of automorphisms of $\mathfrak{g}$ and $\lambda:\mathbb{R}\longrightarrow \mathbb{R}$ is a smooth function. Not all such automorphisms are extendible to the affine Kac-Moody algebra associated to $L(\mathfrak{g}, \sigma)$. We quote theorem 3.4.\ of~\cite{Heintze09}:

\begin{theorem}[Heintze-Gro\ss, 09]
\index{theorem of Heintze Gro\ss}
Let $\widehat{\varphi}:\widehat{L}(\mathfrak{g}, \sigma)\longrightarrow \widehat{L}(\widetilde{\mathfrak{g}}, \widetilde{\sigma})$  be a linear or conjugate linear map. Then $\widehat{\varphi}$ is an isomorphism of Lie algebras iff there exists $\gamma\in \mathbb{F}$ and a linear (resp.\ conjugate linear) isomorphism $\varphi:L(\mathfrak{g}, \sigma))\longrightarrow L(\widetilde{\mathfrak{g}}, \widetilde{\sigma})$ with $\lambda'=\epsilon_{\varphi}$ constant such that
\begin{alignat*}{2}
\widehat{\varphi}c&= \epsilon_{\varphi}c\\
\widehat{\varphi}d&=\epsilon_{\varphi}d-\epsilon_{\varphi} f_{\varphi}+\gamma c\\
\widehat{\varphi}f&=\varphi(u)+\mu(f) c\,.
\end{alignat*}
\end{theorem} 

Following \cite{Heintze09} we call $\widehat{\varphi}$ of ``first type'' if $\epsilon_{\varphi}=1$ and of ``second type'' if $\epsilon_{\varphi}=-1$.

\begin{definition}
Let $\mathfrak{g}$ be reductive Lie algebra, $\sigma$ a diagram automorphism and $L(\mathfrak{g}, \sigma)$ the associated (twisted) loop algebra.
\begin{enumerate} 
\item An involution $\varphi$ of $L(\mathfrak {g}, \sigma)$ is called locally admissible iff its restriction to any irreducible factor is extendible to the associated simple affine Kac-Moody algebra. 
\item It is called extendible iff it can be extended to $\widehat{L}(\mathfrak {g}, \sigma)$.   
\end{enumerate}
\end{definition}

\begin{remark}
It is important to note that on every simple factor of a loop algebra we can choose any automorphism we want, especially it is possible to use any locally admissible automorphism --- i.e.\ the identity, automorphisms of first type and automorphism of second type --- simultaneously.
\end{remark}

\index{involution of 1.~type}
\index{involution of 2.~type}
This is no longer the case with geometric affine Kac-Moody algebras: study now extensions to $\widehat{L}(\mathfrak{g}, \sigma)$. Here we have an important restriction: study any factor $\widehat{L}(\mathfrak{g}_i, \sigma_i)$. The involution $\varphi_i$ of $L(\mathfrak{g}, \sigma)$ defines a unique involution $\widehat{\varphi}_i$ of $\widehat{L}(\mathfrak{g}_i, \sigma_i)$~\cite{Heintze09}. Thus we find:

\begin{lemma} 
\index{extendible involution}
\index{locally admissible involution}
A locally admissible involution of $\varphi:L(\mathfrak{g}, \sigma)\longrightarrow L(\mathfrak{g}, \sigma)$ is admissible, that is extendible to $\widehat{L}(\mathfrak{g}, \sigma)$, iff every restriction $\varphi_i:L(\mathfrak{g}_i, \sigma_i)\longrightarrow L(\mathfrak{g}_i, \sigma_i)$ has the same extension to $\widehat{L}(\mathfrak{g}_i, \sigma_i)$.
\end{lemma}

Thus we have exactly two possibilities:
\begin{enumerate}
\item Every involution $\varphi_i$ is of the first type or the identity. In this case $\varphi$ is called of first type.
\item Every involution $\varphi_i$ is of the second type. In this case $\varphi$ is called of second type.
\end{enumerate}

\section{Orthogonal symmetric affine Kac-Moody algebras}

\subsection{The finite dimensional blueprint}

The classification of finite dimensional Riemann symmetric spaces proceeds by associating a pair consisting of a real simple Lie algebra and an involution to every irreducible symmetric space~\cite{Helgason01}. Those pairs are called ``orthogonal symmetric Lie algebras''~\cite{Helgason01}, chapter V.

We review the approach used for finite dimensional symmetric spaces~\cite{Helgason01}. 
An orthogonal symmetric Lie algebra is defined to be a pair $(\mathfrak{l},\rho)$, consisting of a real lie algebra $\mathfrak{l}$ and an involutive automorphism $\rho$ of $\mathfrak{l}$ such that the fixed point algebra of $\rho$ is a subalgebra of compact type of $\mathfrak{l}$. We distinguish three types of orthogonal symmetric Lie algebras.

\index{orthogonal symmetric Lie algebra (OSLA)}
Let us describe them more closely: let $\mathfrak{l}=\mathfrak{u}\oplus \mathfrak{e}$ be the decomposition of $\mathfrak{l}$ into the $+1$ and $-1$-eigenspace of $\rho$.
\begin{enumerate}
\item {\bf Euclidean type:}\index{OSLA of Euclidean type}
An effective orthogonal symmetric Lie algebra is of Euclidean type iff $\mathfrak{e}$ is Abelian.
\item {\bf Compact type:} \index{OSLA of compact type}
An orthogonal symmetric Lie algebra is of compact type iff $\mathfrak{l}$ is a compact, semisimple Lie algebra.
There are two classes of irreducible orthogonal symmetric Lie algebras of compact type: the first class consists of compact real forms of simple Lie algebras together with an involution $\rho$, the second class consists of an algebra $\mathfrak{l}=\mathfrak{h}\oplus\mathfrak{h}$ such that $\mathfrak{h}$ is a compact real form of a simple Lie algebra together with an involution $\rho$, interchanging the two factors.
\item {\bf Non-compact type:}\index{OSLA of non-compact type}
An orthogonal symmetric Lie algebra is of non-compact type iff $\mathfrak{e}_0$ is non-compact, semisimple, and $\mathfrak{l}=\mathfrak{u}\oplus \mathfrak{e}$ is a Cartan decomposition (for the definition of a Cartan decomposition see~\cite{Helgason01}, p.~183). There are two classes of irreducible Lie algebras of non-compact type:
the first class consists of complex simple Lie algebras with the involution being conjugation with respect to a compact real form, the second one consists of non compact real forms with the involution defined by the Cartan decomposition.
\end{enumerate}

Every orthogonal symmetric Lie algebra can be decomposed into a direct product of three ideals $\mathfrak{l}=\mathfrak{l}_0\oplus \mathfrak{l}_+ \oplus\mathfrak{l}_-$ such that $\mathfrak{l}_0$ is an orthogonal symmetric subalgebra of Euclidean type, $\mathfrak{l}_+$ is an orthogonal symmetric Lie algebra of compact type and $\mathfrak{l}_-$ is an orthogonal symmetric Lie algebra of non-compact type.

Orthogonal symmetric algebras of compact type and of non-compact type are in duality: let $\mathfrak{l}_+=\mathfrak{u}\oplus \mathfrak{e}$ be an orthogonal symmetric Lie algebra of compact type, then $\mathfrak{l}_-=\mathfrak{u}\oplus i\mathfrak{e}$ is an orthogonal symmetric Lie algebra of non-compact type and vice versa.

Let $\mathfrak{g}_{\mathbb{C}}$ be a simple complex Lie algebra, $\mathfrak{g}_{\mathbb{R}}$ a real form of compact type of $\mathfrak{g}_{\mathbb{C}}$ and denote by $\overline{\phantom{z}}$ the conjugation with respect to $\mathfrak{g}_{\mathbb{R}}$. Let $\rho$ denote an involution of $\mathfrak{g}_{\mathbb{R}}$. We define two involutions:
\begin{equation*}
\begin{array}{cc}
\rho_*: \mathfrak{g}_{\mathbb{C}} \longrightarrow \mathfrak{g}_{\mathbb{C}}, & z \mapsto \overline{z}\\
\rho_0: \mathfrak{g}_{\mathbb{C}} \longrightarrow \mathfrak{g}_{\mathbb{C}}, & z \mapsto \overline{\rho(z)}
\end{array}
\end{equation*}

Using this notation we can describe the compact and non-compact real forms of $\mathfrak{g}_{\mathbb{C}}$  as fixed point algebras of the involutions: namely we get  $\mathfrak{g}_{\mathbb{R}}=\textrm{Fix}(\rho_*)$ and for the non-compact dual $\mathfrak{g}_D=\textrm{Fix}(\rho_0)$. Let $\mathfrak{g}_{\mathbb{R}}=\mathfrak{k}\oplus \mathfrak{p}$ be the decomposition into the $+1$ and $-1$ eigenspaces of $\rho$, then $\mathfrak{g}_D=\mathfrak{k}\oplus i\mathfrak{p}$.
\begin{enumerate} 
\item The pair $(\textrm{Fix}(\rho_*)= \mathfrak{g}_{\mathbb{R}}, \rho_0)$ is an orthogonal symmetric Lie algebra of the compact type, 
\item The pair $(\textrm{Fix}(\rho_0)=\mathfrak{g}_D, \rho_*)$ is an orthogonal symmetric Lie algebra of the non-compact type.
\end{enumerate}

\subsection{Orthogonal symmetric affine Kac-Moody algebras}

In this section we will study the infinite dimensional version of this theory which we need for Kac-Moody symmetric spaces: we introduce orthogonal symmetric affine Kac-Moody algebras. The results in this section hold for twisted and non-twisted affine Kac-Moody algebras. We use the notation of twisted affine Kac-Moody algebras together with the convention that $\sigma$ may denote the identity.

\index{orthogonal symmetric affine Kac-Moody algebra (OSAKA)}
We start with some definitions:

\begin{definition}
A real form of a complex geometric affine Kac-Moody algebra $\widehat{L}(\mathfrak{g}_{\mathbb{C}}, \sigma)$ is the fixed point set of a conjugate linear involution. 
\end{definition}
\index{real form of geometric affine Kac-Moody algebra}

\begin{example}
Let $A$ be an affine Cartan matrix, $\mathfrak{g}_{\mathbb{C}}(A)$ be its Lie algebra realization over $\mathbb{C}$. The realization $\mathfrak{g}_{\mathbb{R}}(A)$ over $\mathbb{R}$ is a real form of $\mathfrak{g}_{\mathbb{C}}(A)$. The conjugate linear involution is just ordinary complex conjugation. Similarly any completion $\widehat{L}(\mathfrak{g}_{\mathbb{R}}, \sigma)$ is a real form of the corresponding completion $\widehat{L}(\mathfrak{g}_{\mathbb{C}}, \sigma)$. 
\end{example}

We have described in section~\ref{TheloopalgebraapproachtoKacMoodyalgebras} that involutions of a geometric affine Kac-Moody algebra restrict to involutions of irreducible factors of the loop algebra. Hence, the invariant subalgebras are direct products of invariant subalgebras in those factors together with the appropriate $2$-dimensional extension.

\begin{definition}[compact real affine Kac-Moody algebra]
\index{compact real form of affine Kac-Moody algebra}
A compact real form of a complex affine Kac-Moody algebra $\widehat{L}(\mathfrak{g}_{\mathbb{C}}, \sigma)$ is defined to be a subalgebra of $\widehat{L}(\mathfrak{g}_\mathbb{C},\sigma)$ that is conjugate to the algebra
$\widehat{L}(\mathfrak{g}_{\mathbb {R}}, \sigma)$
where $\mathfrak{g}_{\mathbb R}$ is a compact real form of
$\mathfrak{g}_{\mathbb C}$. 
\end{definition}

\begin{remark}
A semisimple Lie algebra is called of ``compact type'' iff it integrates to a compact semisimple Lie group.
The infinite dimensional generalization of compact Lie groups are loop groups of compact Lie groups and the Kac-Moody groups, constructed as extensions of those loop groups (see section~\ref{Kac-Moody groups}). Thus the denomination is justified by the fact that ``compact'' affine Kac-Moody algebras integrate to ``compact'' Kac-Moody groups. Of course ``compact'' Kac-Moody groups are neither compact nor even locally compact.
\end{remark}

To define a loop group of the compact type we define an infinite dimensional version of the Cartan-Killing form:

\begin{definition}[Cartan-Killing form]
\index{Cartan-Killing form}
The Cartan-Killing form of a loop algebra $L(\mathfrak{g}_{\mathbb{C}},\sigma)$ is defined by
\begin{displaymath}B_{(\mathfrak{g}_{\mathbb{C}},\sigma)}\left(f, g\right)=\int_{0}^{2\pi} B\left(f(t), g(t) \right)dt\,.\end{displaymath}
\end{definition}

\begin{definition}[compact loop algebra]
\index{compact loop algebra}
A loop algebra of compact type is a real form of $L(\mathfrak{g}_{\mathbb{C}},\sigma)$ such that its Cartan-Killing form is negative definite. 
\end{definition}

\begin{lemma}
Let $\mathfrak{g}_{\mathbb{R}}$ be a compact semisimple Lie algebra.
Then the loop algebra $L(\mathfrak{g}_{\mathbb{R}}, \sigma)$ is of compact type.
\end{lemma}

\begin{proof}
The Cartan-Killing form on $\mathfrak{g}_{\mathbb{R}}$ is negative definite. Hence, $B_{(\mathfrak{g}_{\mathbb{C}},\sigma)}\left(f, g\right)$ is negative definite.
\end{proof}

\noindent E.\ Heintze and C.\ Gro\ss\ prove in~\cite{Heintze09}, theorem 7.4.\ for indecomposable simple affine Kac-Moody algebras:

\begin{theorem}
\index{uniqueness of compact real form}
$\widehat{L}(\mathfrak{g}_{\mathbb{C}},\sigma)$ has a compact real form which is unique up to conjugation. 
\end{theorem}

Their proof extends directly to semisimple geometric affine Kac-Moody algebras such that $\mathfrak{g}_{\mathbb{C}}$ is semisimple (but need not be simple): we choose the (up to conjugation) unique subalgebra in every factor and have to check that the extensions fit together. This follows from the explicit construction. 

To find non-compact real forms, we need the following result of E.\ Heintze and C.\ Gro\ss\ (Corollary 7.7.\ of~\cite{Heintze09}):

\begin{theorem}
\label{theoremofHeintzegross}
Let $\mathcal{G}$ be a irreducible complex geometric affine Kac-Moody algebra, $\mathcal{U}$ a real form of compact type. The conjugacy classes of real forms of non-compact type of $\mathcal{G}$ are in bijection with the conjugacy classes of involutions on $\mathcal{U}$. The correspondence is given by $\mathcal{U}=\mathcal{K}\oplus \mathcal{P}\mapsto \mathcal{K}\oplus i\mathcal{P}$ where $\mathcal{K}$ and $\mathcal{P}$ are the $\pm 1$-eigenspaces of the involution.  
\end{theorem}

Thus to find non-compact real forms, we have to study automorphisms of order $2$ of a geometric affine Kac-Moody algebra of the compact type. From now on we restrict to involutions $\widehat{\varphi}$ of second type, that is those such that $\epsilon_{\varphi}=-1$.

Now we want to extend theorem \ref{theoremofHeintzegross} to non-irreducible geometric affine Kac-Moody algebras.
\begin{enumerate}
\item Suppose first that the involution $\widehat{\varphi}$ on $\widehat{L}(\mathfrak{g}, \sigma)$ is chosen in a way that every irreducible factor is of type $I$. In this case $\widehat{\varphi}$ restricts to an involution $\widehat{\varphi}_i$ on every irreducible factor $\widehat{L}(\mathfrak{g}_i, \sigma_i)$. Let $L(\mathfrak{g}_i, \sigma_i)=\mathcal{K}_i \oplus \mathcal{P}_i$ be the decomposition of $L(\mathfrak{g}_i, \sigma_i)$ into the eigenspaces of $\varphi_i$. The dualization is performed according to the standard pattern using $\mathcal{K}=\oplus_i \mathcal{K}_i$ and $\mathcal{P}=\oplus_i \mathcal{P}_i\oplus \mathbb {R} c\oplus \mathbb{R}d$.

\item If the decomposition of $\widehat{\varphi}$ contains simple factors of type $II$, we perform the same decomposition procedure: 
let $\widehat{L}(\mathfrak{g}_j \oplus \mathfrak{g}_j', \sigma_j \oplus \sigma_j')$ be irreducible with respect to $\widehat{\varphi}$, then we have the decomposition 
$L(\mathfrak{g}_j\oplus \mathfrak{g}_j', \sigma_j \oplus \sigma_j')=\mathcal{K}_j\oplus \mathcal{P}_j$. Dualization follows the standard pattern.
\end{enumerate}

\noindent We have to investigate iff all real forms can be described in this way:

\begin{enumerate}
\item If $\mathfrak{g}$ is simple it is a result of E.\ Heintze and C.\ Gro\ss\ that every real form of non-compact type can be constructed in this way~\cite{Heintze09}.
\item If $\mathfrak{g}$ is not simple we use  that the restriction to the loop algebra $L(\mathfrak{g}, \sigma)$ of any real form consists of the direct product of real forms in the irreducible components of $L(\mathfrak{g}_{\mathbb{C}}, \sigma)$. Hence, according to the result of E.\ Heintze and C.\ Gro\ss, those are of the described type, and thus the non-compact real form we started with.
\end{enumerate}

\noindent Putting this together we have proved the following two theorems:

\begin{theorem}
Let $\widehat{L}(\mathfrak{g}, \sigma)$ be a geometric affine Kac-Moody algebra of the compact type. Let $\widehat{\varphi}$ be an involution of order $2$ of the second kind of $\widehat{L}(\mathfrak{g}, \sigma)$. Let furthermore  $\widehat{L}(\mathfrak{g}, \sigma)=\mathcal{K} \oplus \mathcal{P}$ be the decomposition into its $\pm 1$-eigenspaces.
Then $\mathcal{G}_D:=\mathcal{K}\oplus i \mathcal{P}$ is the dual real form of the non-compact type. 
\end{theorem}

\begin{theorem}
\label{eithercompactornoncompact}
Every real form of a complex geometric affine Kac-Moody algebra is either of compact type or of non-compact type. A mixed type is not possible.
\end{theorem}

Note, that the combination of the identity on some factors with involutions of second type on other factors is excluded.

\begin{proposition}
Let $\mathfrak{g}$ be semisimple and $\widehat{L}(\mathfrak{g}, \sigma)_D$ be a real form of the non-compact type. Let $\widehat{L}(\mathfrak{g}, \sigma)_D=\mathcal{K}\oplus \mathcal{P}$ be a Cartan decomposition. The Cartan-Killing form is negative definite on $\mathcal{K}$ and positive definite on $\mathcal{P}$.
\end{proposition}

\begin{proof}
Suppose first $\sigma$ is the identity. Let $\varphi$ be an automorphism. Then without loss of generality $\varphi(f)=\varphi_0(f(-t))$~\cite{Heintze09}. Let $\mathfrak{g}=\mathfrak{k}\oplus \mathfrak{p}$ be the decomposition of $\mathfrak{g}$ into the $\pm 1$-eigenspaces of $\varphi_0$. Then
$f\in \textrm{Fix}(\varphi)$ iff its Taylor expansion satisfies
$$\sum_n a_n e^{int}= \sum \varphi_0(a_{-n})e^{int}\,.$$
Let $a_n=k_n\oplus p_n$ be the decomposition of $a_n$ into the $\pm 1$ eigenspaces with respect to $\varphi_0$. 
Hence $$f(t)=\sum_n k_n \cos(nt)+\sum_n p_n \sin(nt)\,.$$ 
Then using bilinearity and the fact that $\{\cos(nt), \sin(nt)\}$ are orthonormal we can calculate $B_{\mathfrak{g}}$:
$$B_\mathfrak{g}=\int_0^{2\pi}\sum_n \cos^{2}(nt)B(k_n, k_n)-\int_0^{2\pi}\sum_n \sin^{2}(nt)B(p_n, p_n)\,.$$
Hence $B_{\mathfrak{g}}$ is negative definite on $\textrm{Fix}(\varphi)$. Analogously one calculates that it is positive definite on the $-1$-eigenspace of $\varphi$. 
If $\sigma\not=Id$ then one gets the same result by embedding $L(\mathfrak{g}, \sigma)$ into an algebra $L(\mathfrak{h}, \textrm{id})$ which is always possible~\cite{Kac90}. 

\end{proof}

\begin{proposition}
Let $\mathfrak{g}$ be abelian. The Cartan-Killing form of $L(\mathfrak{g})$ is trivial.
\end{proposition}

\begin{proof}
If $\mathfrak{g}$ is abelian its Cartan-Killing form $B_{\mathfrak{g}}$ vanishes. Hence the integral over $B_{\mathfrak{g}}$ vanishes.
\end{proof}

\noindent Now we can define orthogonal symmetric affine Kac-Moody algebras (OSAKA):

\begin{definition}[Orthogonal symmetric Kac-Moody algebra]
An orthogonal symmetric affine Kac-Moody algebra (OSAKA) is a pair $\left(\widehat{L}(\mathfrak{g}, \sigma), \widehat{L}(\rho)\right)$ such that
\begin{enumerate}
	 \item $\widehat{L}(\mathfrak{g}, \sigma)$ is a real form of an affine geometric Kac-Moody algebra,
	 \item $\widehat{L}(\rho)$ is an involutive automorphism of $\widehat{L}(\mathfrak{g}, \sigma)$,
	 \item If $\mathfrak{g}=\mathfrak{g}_s\oplus \mathfrak{g}_a$ then the restriction $\textrm{Fix}(\widehat{L}(\rho))|_{\mathfrak{g}_s}$ is a compact Kac-Moody algebra or a compact loop group and $\textrm{Fix}(\widehat{L}(\rho))|_{\mathfrak{g}_a}=0$.
\end{enumerate}
\end{definition}

\begin{definition}
An OSAKA is called effective if $\textrm{Fix}(\widehat{L}(\rho))\cap\mathfrak{z}=0$ where $\mathfrak{z}$ denotes the center of $\widehat{L}(\mathfrak{g}, \sigma)$.
\end{definition}

\begin{lemma}
Let $\left(\widehat{L}(\mathfrak{g}, \sigma), \widehat{L}(\rho)\right)$ be an effective OSAKA. Then $\widehat{L}(\rho)$ is of second type and $\textrm{Fix}(\widehat{L}(\rho))$ is a compact loop group. 
\end{lemma}

\begin{proof}
We have $c\in \mathfrak{z}$. Hence $c\not\in \widehat{L}(\rho)$. But $\widehat{L}(\rho)(c)\in \{\pm c\}$. Hence $\widehat{L}(\rho)$ is of second type.
By the classification of involutions we know that $\widehat{L}(\rho)(c)=-c$ implies that $\widehat{L}(\rho)$ is quasiconjugate to an involution satisfying $\widehat{L}(\rho)(d)=-d$. Hence $\textrm{Fix}(\widehat{L}(\rho))\subset L(\mathfrak{g}, \sigma)$.
\end{proof}

Following again the presentation of Helgason, we define 3 types of OSAKAs:

\index{OSAKA of Euclidean type}
\index{OSAKA of compact type}
\index{OSAKA of non-compact type}
\begin{definition}[Types of OSAKAs]
\label{types of OSAKAs}
Let $(\widehat{L}(\mathfrak{g}, \sigma), \widehat{L}(\rho))$ be an OSAKA. Let $\widehat{L}(\mathfrak{g}, \sigma)=\mathcal{K}\oplus \mathcal{P}$ be the decomposition of $\widehat{L}(\mathfrak{g}, \sigma)$ into the eigenspaces of $\widehat{L}(\mathfrak{g}, \sigma)$ of eigenvalue $+1$ resp.\ $-1$.
\begin{enumerate}
	\item If $\widehat{L}(\mathfrak{g}, \sigma)$ is a compact real affine Kac-Moody algebra, it is said to be of the compact type.
	\item If $\widehat{L}(\mathfrak{g}, \sigma)$ is a non-compact real affine Kac-Moody algebra, $\widehat{L}(\mathfrak{g}, \sigma)=\mathcal{U}\oplus \mathcal{P}$ is a Cartan decomposition of $\widehat{L}(\mathfrak{g}, \sigma)$.
	\item If $L(\mathfrak{g},\sigma)$ is abelian it is said to be of Euclidean type.
\end{enumerate}
\end{definition}

\begin{definition}[semisimple OSAKA]
\index{semisimple OSAKA}
An OSAKA $(\widehat{L}(\mathfrak{g}, \sigma), \widehat{L}(\rho))$ is called semisimple if $\widehat{L}(\mathfrak{g}, \sigma)$ is a semisimple geometric affine Kac-Moody algebra.
\end{definition}

OSAKAs of the compact type and of the noncompact type are semisimple, OSAKAs of the Euclidean type are not semisimple.

\index{irreducible OSAKA}
\begin{definition}[irreducible OSAKA]
An OSAKA $(\widehat{L}(\mathfrak{g}, \sigma), \widehat{L}(\rho))$ is called irreducible iff it has no non-trivial Kac-Moody subalgebra invariant under $\widehat{L}(\rho)$.
\end{definition}

\noindent Thus we can describe the different classes of irreducible OSAKAs of compact type.

\index{OSAKA of type I}
\index{OSAKA of type II}
\begin{enumerate}
\item The first class consists of compact real forms $\widehat{L}(\mathfrak{g},{\sigma})$, where $\mathfrak{g}$ is a simple Lie algebra together with an involution of the second kind. A complete classification is available: As each Kac-Moody algebra has a (up to conjugation unique real form), a classification consists in running through all pairs consisting of affine Kac-Moody algebras and conjugate linear involutions of the second kind. Lists of involutions may be found in~\cite{Heintze08}. Those OSAKAs are irreducible OSAKAs of type $I$. We will see in chapter~\ref{chap:symm} that they correspond to Kac-Moody symmetric spaces of type $1$. A paper of Tripathy and Pati gives a list of Satake diagrams~\cite{Tripathy06}. 
\item Let $\mathfrak{g}_{\mathbb{R}}$ be a simple real Lie algebra of the compact type. The second class consists of pairs of an affine Kac-Moody algebra
$\widehat{L}(\mathfrak{g}_{\mathbb{R}}\times \mathfrak{g}_{\mathbb{R}})$ together with an involution $\widehat{\varphi}$ of the second kind, that switches the two factors. 
\begin{align*}
\varphi:\quad \widehat{L}(\mathfrak{g}_{\mathbb{R}}\times \mathfrak{g}_{\mathbb{R}}, \sigma\oplus\sigma)&\longrightarrow \widehat{L}(\mathfrak{g}_{\mathbb{R}}\times \mathfrak{g}_{\mathbb{R}},\sigma\oplus\sigma)\\
(f(t), g(t), r_c, r_d)&\mapsto (g(-t), f(-t), -r_c, -r_d)
\end{align*}
Hence the fixed point algebra consists of elements $(f(t), f(-t), 0,0)$, which is isomorphic to $L(\mathfrak{g}, \sigma)$.
Those algebras correspond to Kac-Moody symmetric spaces that are compact Kac-Moody groups equipped with their $Ad$-invariant metrics (type $II$). As each Kac-Moody algebra has a (unique up to conjugation) compact real form, a complete classification consists in running through all affine Kac-Moody algebras.
\end{enumerate}

Besides the OSAKAs of the compact type there are the OSAKAs of the non-compact type.
\index{OSAKA of type III}
\index{OSAKA of type VI}

\begin{itemize}
\item[(3)] Let $\mathfrak{g}_{\mathbb{C}}$ be a complex semisimple Lie algebra, and $\widehat{L}(\mathfrak{g}_{\mathbb{C}},\sigma)$ the associated complex affine Kac-Moody algebra.
This class consists of real forms of the non-compact type that are described as fixed point sets of involutions of type 2 together with a special involution, called Cartan involution. This is the unique involution on $\mathcal{G}$, such that the decomposition into its $\pm 1$-eigenspaces $\mathcal{K}$ and $\mathcal{P}$ yields: $\mathcal{K}\oplus i \mathcal{P}$ is a real form of compact type of $\widehat{L}(\mathfrak{g}_{\mathbb{C}},\sigma)$. For existence and uniqueness of the Cartan involution see~\cite{Heintze09}. Those orthogonal symmetric Lie algebras correspond to Kac-Moody symmetric spaces of type $III$.

\item[(4)] Let $\mathfrak{g}_{\mathbb{C}}$ be a complex semisimple Lie algebra. The fourth class consists of negative-conjugate real forms of $\widehat{L}(\mathfrak{g}_{\mathbb{C}}\oplus\mathfrak{g}_{\mathbb{C}}, \sigma \oplus \sigma)$. The involution is given by the complex conjugation $\widehat{L}(\rho_0)$ with respect to a compact real form of $L(\mathfrak{g}_{\mathbb{C}},\sigma)\cong L(\mathfrak{g}, \sigma)\oplus i L(\mathfrak{g}, \sigma)$. Hence they have the form: $L(\mathfrak{g}, \sigma)\oplus iL(\mathfrak{g},\sigma)\oplus\mathbb{R}c\oplus \mathbb{R}d$.
Let us remark that we can also take the complex Kac-Moody algebra $\widehat{L}(\mathfrak{g}_{\mathbb{C}}, \sigma)$ and define the involution $\widehat{\rho}_0$ as conjugation with respect to the real form of compact type $\widehat{L}(\mathfrak{g}_{\mathbb{R}}, \sigma)$. Nevertheless the pair $(\widehat{L}(\mathfrak{g}_{\mathbb{C}}, \sigma), \widehat{\rho}_0)$ is not an OSAKA in the sense of our definition as the fixed point algebra is not a loop algebra of the compact type and as the subjacent Kac-Moody algebra is not a real Kac-Moody algebra. Both of those algebras describe the same symmetric space as their $\mathcal{P}$-components are the same.
Those algebras correspond to Kac-Moody symmetric spaces of type $IV$.
\end{itemize}

Let us investigate the duality a little more closely. We follow the presentation for finite dimensional Lie algebras given in~\cite{Helgason01}.
Let $(\mathcal{G},\rho)$ be an OSAKA and let $\mathcal{G}=\mathcal{K}\oplus\mathcal{P}$ be the decomposition into the $\pm 1$-eigenspaces of $\rho$. Then $\mathcal{G}^*=\mathcal{K}\oplus\mathcal{P}$ is a real Lie form of $\mathcal{G}_{\mathbb{C}}$. For any element $g\in \mathcal{G}$ let $g=k+p$ be the decomposition into the $\mathcal{K}$-and $\mathcal{P}$-component. Define $\rho^*:k+ip\mapsto k-ip$.

\index{dual OSAKA}
\begin{proposition}
Let $(\mathcal{G},\rho)$ be an OSAKA. 
\begin{enumerate}
\item Then the pair $(\mathcal{G}^*, \rho^*)$ is an OSAKA, called the dual OSAKA.
\item If $(\mathcal{G},\rho)$ is of compact type then $(\mathcal{G}^*, \rho^*)$ is of non-compact type and vice versa.
\end{enumerate}
\end{proposition}

\begin{proof}
Suppose $\mathcal{G}=\mathcal{K}\oplus \mathcal{P}$.  By assumption $\mathcal{G}$ is an OSAKA. Hence $\mathcal{K}=Fix(\rho)$ is a loop algebra of the compact type. $\rho^*$ is an involution on $\mathcal{G}^*$. $Fix(\rho^*)=Fix(\rho)=\mathcal{K}$ is hence also a loop algebra of the compact type. Thus $(\mathcal{G}^*, \rho^*)$ is an OSAKA proving (1).\\ 
If $(\mathcal{G},\rho)$ is of compact type then $(\mathcal{G}^*,\rho^*)$ is of non-compact type by results of~\cite{Heintze08}.
\end{proof}

\begin{corollary}
The OSAKAs of the following types are dual:
\begin{center}\begin{tabular}{ccc}
type I&$\Leftrightarrow$& type III\\
type II&$\Leftrightarrow$&type IV
\end{tabular} 
\end{center}
\end{corollary}

The derived algebras of the last class of OSAKAs --- the ones of Euclidean type  --- are  Heisenberg algebras~\cite{PressleySegal86}. The maximal subgroups of compact type are trivial.  Hence the involution inverts all elements.

\chapter{Tame structures on affine Kac-Moody algebras and Kac-Moody groups}
\label{chap:tame}

In this chapter we introduce the functional analytic setting, we need for the construction of affine Kac-Moody symmetric spaces: We define Kac-Moody groups and Kac-Moody algebras of holomorphic loops and prove that they are tame Fr\'echet Lie groups resp. Lie algebras; we define ILB-structures on them. As Kac-Moody groups (resp.\ algebras) are $2$-dimensional extensions of loop groups we focus our attention to the loop group (resp.\ algebra). With the only exception that the action of group elements in the two dimensional extension has to be well-defined,  the extensions do not lead to new functional analytic problems.

\section{Lie algebras of holomorphic maps}
\label{Lie_algebras_of_holomorphic_maps}

Let $\mathfrak{g}$ be a complex reductive Lie algebra that is a direct product of simple Lie algebras with an abelian Lie algebra. 
\index{reductive Lie algebra} Simple Lie algebras are completely classified by their root systems; we give a complete list:  \index{classification of simple Lie algebra}
\begin{displaymath}
A_n, B_{n, n\geq 2}, C_{n, n\geq 3}, D_{n, n \geq 4}, E_6, E_7, E_8, F_4, G_2\,. 
\end{displaymath}

Abelian Lie algebras are classified by their dimension. We define a 
 complex reductive Lie algebra $\mathfrak{g}_{\mathbb{C}}$ to be a real reductive Lie algebra $\mathfrak{g}$ such that $\mathfrak{g}$ is a direct product of the (up to conjugation) unique compact real forms of the simple factors of $\mathfrak{g}_{\mathbb{C}}$ together with a compact real form of the abelian factor. A compact real abelian Lie algebra of dimension $n$ is just $\mathbb{R}^n$ together with the trivial bracket; but we define the exponential function such that the exponential image is a torus. Hence a compact real Lie algebra can be identified with the purely imaginary part of a complex abelian Lie algebra.

\begin{definition}[complex holomorphic non-twisted Loop algebra]
\index{holomorphic loop algebra}
\label{complex holomorphic non-twisted Loop algebra}
Let $\mathfrak{g}_{\mathbb C}$ be a finite-dimensional reductive complex Lie algebra. 

\begin{enumerate}
\item  The loop algebra $A_n\mathfrak{g}_{\mathbb{C}}$ is the vector space
$$A_n\mathfrak{g}_{\mathbb{C}}:=\bigcup_{A_n\subset U \textrm{open}}\{f:U \longrightarrow
\mathfrak{g}_\mathbb {C}| \textrm{ f is holomorphic}\}\,,$$
equipped with the natural Lie bracket:
$$[f,g]_{L_n\mathfrak{g}}(z):=[f,g]_0(z):=[f(z),g(z)]_{\mathfrak{g}}\,.$$

\item  The loop algebra $M\mathfrak{g}_{\mathbb{C}}$ is the vector space
$$M\mathfrak{g}_{\mathbb{C}}:=\{f:\mathbb C^* \longrightarrow
\mathfrak{g}_\mathbb {C}| \textrm{ f is holomorphic}\}$$
equipped with the natural Lie bracket:
$$[f,g]_{M\mathfrak{g}}(z):=[f,g]_0(z):=[f(z),g(z)]_{\mathfrak{g}}\,.$$
\end{enumerate}
\end{definition}

\begin{lemma}~
\label{mgfrechetagbanach}
\begin{enumerate}
\item $M\mathfrak{g}_{\mathbb{C}}$ is a tame  space.
\item $A_n\mathfrak{g}_{\mathbb{C}}$ is a Banach space. 
\end{enumerate}
\end{lemma}

\begin{proof}
The first assertion is a consequence of corollary~\ref{holc*cnisfrechet}. The second assertion is a consequence of Montel's theorem stating that absolute convergent sequences of holomorphic functions converge to a holomorphic function~\cite{Berenstein91}. 
\end{proof}

\noindent The inclusions $S^1 =A_0 \subset \dots A_n \subset A_{n+1} \subset \dots \subset \mathbb{C}^*$ induce the reversed inclusions on the associated loop algebras:

\begin{displaymath}
M\mathfrak{g}_{_{\mathbb{C}}} \subset \dots \subset A_{n+1}\mathfrak{g}_{\mathbb{C}}\subset A_n\mathfrak{g}_{\mathbb{C}} \subset \dots \subset A_0\mathfrak{g}_{\mathbb{C}}= L_{\textrm{hol}}\mathfrak{g}_{\mathbb{C}}\,.
\end{displaymath}

$L_{\textrm{hol}}\mathfrak{g}_{\mathbb{C}}$ denotes functions holomorphic in a small open set around $S^1\subset \mathbb{C}^*$.

To describe the twisted loop algebras we recall the graph automorphisms of the finite dimensional simple Lie algebras:
the following list contains the simple algebras $A$ with a nontrivial diagram automorphism $\sigma$ and the type of the fixed point algebra (compare \cite{Carter05}). 
\[
\begin{array}{lrccccc}
A&:&A_{2k}&A_{2k+1}&D_{k+1}&D_4&E_6\\
\textrm{Order of }\sigma&:&2&2&2&3&2\\
A^1&:&B_k&C_k&B_k&G_2&F_4
\end{array}
\]

\begin{definition}[(twisted) loop algebra, $\textrm{ord}(\sigma ) =2$]
\index{twisted holomorphic loop algebra}
\label{definitiontwistedloopalgebra}
Let $\mathfrak{g}_{\mathbb C}$ be a finite dimensional semisimple complex Lie
algebra of type $A_{k}$, $D_{k, k \geq 5}$ or $E_6$,  $\sigma$ the diagram automorphism. Let $\mathfrak{g}_{\mathbb C}:=\mathfrak{g}_{\mathbb C}^1 \oplus \mathfrak{g}_{\mathbb C}^{-1}$ be the decomposition into the $\pm$-eigenspaces of $\sigma$. Let $X\in \{A_n, \mathbb{C}^*\}$. If $X=A_n$ holomorphic functions on $X$ are understood to be holomorphic on an open set containing $X$.

\noindent Then the loop algebra $(X\mathfrak{g})^{\sigma}$ is the vector space
$$X\mathfrak{g}^{\sigma}:=\{ f\in X\mathfrak{g}| f(-z)=\sigma(f(z))\}\,,$$
equipped with the natural Lie bracket:
$$[f,g]_{X\mathfrak{g}^{\sigma}}(z):=[f,g]_0(z):=[f(z),g(z)]_{\mathfrak{g}}\,.$$

\end{definition}

\begin{remark}[(twisted) loop algebra, $\textrm{ord}(\sigma ) =3$]
For the algebra of type $D_4$ there exists an automorphism $\sigma$ of order $3$. In this case we get exactly the same results as for the other types. The main difference is that we have three eigenspaces, corresponding to $\{\omega, \omega^2, \omega^3=1\}$ for $\omega= e^{\frac{2\pi i }{3}}$. For a function $f$ in the loop algebra $M\mathfrak{g}$, this results in a twisting condition $f(\omega z)=\sigma f(z)$ (for details compare again~\cite{Carter05}).
\end{remark}

\begin{lemma}[Banach- and Fr\'echet structures on twisted loop algebras]~
\begin{enumerate}
	\item $A_n\mathfrak{g}^{\sigma}$ equipped with the norm $\|\hspace{3pt} \|_n$ is a Banach Lie algebra,
	\item $M\mathfrak{g}^{\sigma}$ equipped with the norms $\|\hspace{3pt} \|_n$ is a tame Fr\'echet Lie algebra. 
\end{enumerate}
\end{lemma}

\begin{proof}
Closed subspaces of Banach spaces are Banach spaces and closed subspaces of tame Fr\'echet spaces are tame Fr\'echet spaces (lemma~\ref{constructionoftamespaces}). 
\end{proof}

\noindent To unify notation suppose the identity to be an involution and define $X\in \{A_n, \mathbb{C}\}$. If $X=A_n$ then holomorphic functions on $X$ are understood to be holomorphic on an open set containing $X$.

\begin{proposition}
\label{nganginverselimit}
The system $\{M\mathfrak{g}_{\mathbb{C}}^{\sigma}; A_n\mathfrak{g}_{\mathbb{C}}^{\sigma}\}$ is an $ILB$-system.
\end{proposition}

\begin{proof}~

\begin{enumerate}
\item We have to check that $A_{n+1}\mathfrak{g}^{\sigma}\hookrightarrow A_n\mathfrak{g}^{\sigma}$ is a continuous, dense embedding. Continuity follows as $\|\hspace{3pt}\|_{n}\leq \|\hspace{3pt}\|_{n+1}$. Thus the embedding is a bounded linear map and thus continuous; the image is dense as polynomials on $\mathbb{C^*}$ are dense in $A_{n}\mathfrak{g}$  for all $n\in \mathbb{N}$.
\item The topology on  $M\mathfrak{g}^{\sigma}$ is the inverse limit topology as it is the topology generated by the set of all norms $\|\hspace{3pt}\|_n$.  
\end{enumerate}
\end{proof}

\begin{theorem}
\index{topology on $M\mathfrak{g}$}
 The following topologies on $M\mathfrak{g}_{\mathbb{C}}$ are equivalent:
\begin{enumerate}
 \item the compact-open topology,
 \item the topology of compact convergence,
 \item the Fr\'echet topology,
 \item the $ILB$\ndash topology.
\end{enumerate}
\end{theorem}

Recall that on a Fr\'echet space $(F, \|\ \|_n)$
\begin{displaymath}
 d(f,g)=\sum_n\frac{1}{2^n}\textrm{\small$\frac{\|f-g\|_n}{1+\|f-g\|_n}$\normalsize}
\end{displaymath}
defines a metric. The Fr\'echet topology is the topology generated by the metric $d(f,g)$~\cite{Hamilton82}.
Let $(F; F^n)$ be an ILB\ndash system. The $ILB$\ndash topology on $F$ is the weakest topology  such that the embedding $F\hookrightarrow F^k$ is continuous for every $k$~\cite{Omori97}.

\begin{proof} We prove $(1)\Leftrightarrow (2)$, $(3)\Leftrightarrow (2)$ and $(3)\Leftrightarrow (4)$.
 \begin{enumerate}
  \item [$(1)\Leftrightarrow$]$ (2)$ The equivalence between the compact-open topology and the topology of compact convergence is well-known for spaces of continuous functions $C(X,Y)$ for $X$, $Y$ metric spaces. The extension to the setting of holomorphic functions is a consequence of Montel's theorem.
\item [$(3)\Leftrightarrow$]$ (2)$ Let $(f_k) \subset M\mathfrak{g}$ be a sequence converging to $f_0$ in the topology of compact convergence. Then for every $n_0\in \mathbb{N}$ and $\epsilon> 0$ there is a $k_0$ such that for all $k\geq k_0$ the estimate $\|f_k-f_0\|_n\leq\epsilon$ is satisfied for all $n\leq n_0$. Hence
\begin{align*}
 d(f_k,f_0)& =\sum_{n=0}^{\infty}\frac{1}{2^n} \textrm{\small$\frac{\|f_k-f_0\|_n}{1+\|f_k-f_0\|_n}$\normalsize}=\\
           & =\sum_{n=0}^{n_0}\frac{1}{2^n}\textrm{\small$\frac{\|f_k-f_0\|_n}{1+\|f_k-f_0\|_n}$\normalsize}+
           \sum_{n=n_0+1}^{\infty}\frac{1}{2^n}\textrm{\small$\frac{\|f_k-f_0\|_n}{1+\|f_k-f_0\|_n}$\normalsize}\leq\\
&\leq \sum_{n=0}^{n_0}\frac{1}{2^n}\frac{\epsilon}{1+\epsilon}+ \sum_{n=n_0+1}^{\infty}\frac{1}{2^n}\leq\\
&\leq 2\epsilon+\left(\frac{1}{2}\right)^{n_0}\, .
\end{align*}
Hence $(f_k) \subset M\mathfrak{g}$ converges in the Fr\'echet topology.
Conversely let $(f_k) \subset M\mathfrak{g}$ be a sequence converging to $f_0$ in the Fr\'echet topology. Then we have
\begin{displaymath}
\lim_{k\rightarrow \infty} d(f_k,f_0)=\lim_{k\rightarrow\infty} \sum_n\frac{1}{2^n}\textrm{\small$\frac{\|f_k-f_0\|_n}{1+\|f_f-f_0\|_n}$\normalsize}=0\, .
\end{displaymath}
As $d(f_k,f_0)\geq \sup_n \frac{1}{2^n}\frac{\|f_k-f_0\|_n}{1+\|f_f-f_0\|_n}$ we conclude that 
\begin{displaymath}
 \lim_{k\rightarrow \infty}\sup_n \frac{1}{2^n}\frac{\|f_k-f_0\|_n}{1+\|f_f-f_0\|_n}=0\, .
\end{displaymath}
As for any compact set $K\subset \mathbb{C}^*$ there is some $n$ such that $K\subset A_n$, this yields compact convergence.
\item [$(3)\Leftrightarrow$]$ (4)$ Suppose $M\mathfrak{g}$ is equipped with the tame Fr\'echet topology. We first check that the embedding 
\begin{displaymath}
 \varphi_k: M\mathfrak{g}\hookrightarrow A_n\mathfrak{g}\, .
\end{displaymath}
is continuous. To this end let $f_n\subset M\mathfrak{g}$ be a sequence converging to $f_0$. This means that $$\lim_{n\rightarrow \infty}\|f_n-f_0\|_k=0$$ for all $k$. Hence $\lim_{n\rightarrow \infty} \varphi(f_n)=\varphi(f_0)$. Thus $\varphi$ is continuous. As a consequence the Fr\'echet topology is stronger than the ILB\ndash topology. Let conversely $M\mathfrak{g}$ be equipped with the ILB\ndash topology. A sequence $(f_n)\in M\mathfrak{g}$ converges if $\varphi_k(f_n)\in A_k\mathfrak{g}$ converges for all $k$. This is equivalent to 
\begin{displaymath}
 \lim_{n\rightarrow \infty}\|f_n-f_0\|_k=0
\end{displaymath}
 and thus to convergence in the Fr\'echet topology. Hence the Fr\'echet topology is weaker than the ILB\ndash topology. This completes the proof.
 \end{enumerate}

\end{proof}

The adjoint action $ad(g):M\mathfrak{g}^{\sigma}\longrightarrow M\mathfrak{g}^{\sigma}$ is $(0,0, 2\|g\|_n)$-tame for each  $g\in M\mathfrak{g}^{\sigma}$. Thus it induces an adjoint action on each algebra $A_n\mathfrak{g}^{\sigma}$ of the associated $ILB$-system. Contrast this with the situation for the affine Kac-Moody algebra  described in section~\ref{holomorphicstructuresonkacmoodyalgebras}.

\noindent Having described the holomorphic complex loop algebras which we will need, we turn now to some associated objects, namely the compact real forms and spaces of differential forms.

\noindent We start with real forms of compact type:

\begin{definition}[compact real form of a holomorphic non-twisted loop algebra]
\index{compact real non-twisted loop algebra}
Let $\mathfrak{g}_{\mathbb C}$ be a finite-dimensional semisimple complex Lie algebra and $\mathfrak{g}$ its compact real form.
 The loop algebra $X\mathfrak{g}_{\mathbb{R}}^{\sigma}$ is the vector space
\begin{displaymath}X\mathfrak{g}_{\mathbb{R}}^{\sigma}:=\{f\in X\mathfrak{g}_{\mathbb{C}}^{\sigma}| f(S^1)\subset \mathfrak{g} \}\,,\end{displaymath}
equipped with the natural Lie bracket:
\begin{displaymath}[f,g]_{X\mathfrak{g}}(z):=[f,g]_0(z):=[f(z),g(z)]_{\mathfrak{g}}\,.\end{displaymath}

\end{definition}

\noindent As a holomorphic function on $X$ can be expanded into its Laurent
series, one can represent every element of a loop algebra by a series 
\begin{displaymath} f(z):= \sum_{n} g_n z^n\end{displaymath}
with $g_n \in \mathfrak{g}$.

\begin{lemma}
The condition $f(S^1)\subset \mathfrak{g}_{\mathbb{R}}$  is equivalent to the condition $g_n=-\bar{g}_{-n}^{t}$.
\end{lemma}

\begin{proof}
Let $z=e^{it}\in S^1\subset \mathbb{C}^*$ and let $g_n=g_n^{r}+ig_n^{i}$ be the decomposition of $g_n$ into its real and imaginary parts. Then we find
\begin{align*}
f(z)&= \sum_{n\in \mathbb{Z}} g_n z^n=\sum_{n\in \mathbb{Z}} g_n e^{itn}=\\
&=a_0+\sum_{n\in \mathbb{N}}g_n e^{itn}+g_{-n}e^{-itn}=\\
&=a_0+\sum_{n\in \mathbb{N}}(g_n^{r}+ig_n^{i}) (\cos(tn)+i\sin(tn))+(g_{-n}^{r}+ig_{-n}^{i})(\cos(-tn)+i\sin(-tn))=\\
&=a_0+\sum_{n\in \mathbb{N}}(g_n^{r}\cos(tn)-g_n^{i}\sin(tn)) +i(g_{n}^{i}\cos(tn)+g_n^r\sin(tn))+\\
&\qquad \qquad(g_{-n}^{r}\cos(-tn)- g_{-n}^{i}\sin(-tn) +i(g_{-n}^{i}\cos(-tn)+(g_{-n}^{r}\sin(-tn))=\\
&=a_0+\sum_{n\in \mathbb{N}}(g_n^r+g_{-n}^r+i(g_{n}^{i}+g_{-n}^{i}))\cos(tn)+(-g_n^{i}+g_{-n}^{i}+i(g_n^r-g_{-n}^r))\sin(tn)
\end{align*}
Now $x\in \mathfrak{g}_{\mathbb{C}}$ is in $\mathfrak{g}$ if $x=-\overline{x}^t$. This gives for the coefficients $g_{n}^r$:
\begin{align*}
g_n^r+g_{-n}^r&=-(g_n^r)^t-(g_{-n}^r)^t\, ,\\
g_n^r-g_{-n}^r&=(g_n^r)^t-(g_{-n}^r)^t\, .
\end{align*}
Adding both we get $g_n^{r}=-(g_{-n}^r)^t$.
In a similar way we get for the imaginary parts of the coefficients $g_n^{i}=(g_{-n}^i)^t$ and thus the result.
\end{proof}

\begin{lemma}
$M\mathfrak{g}_{\mathbb{R}}$ is a tame Fr\'echet space.
\end{lemma}

\begin{proof}
$M\mathfrak{g}_{\mathbb{R}}\subset M\mathfrak{g}_{\mathbb{C}}$ is a closed subspace and thus tame according to lemma~\ref{constructionoftamespaces}.
\end{proof}

\begin{proposition}
The system $\{M\mathfrak{g}_{\mathbb{R}}^{\sigma}, A_n\mathfrak{g}_{\mathbb{R}}^{\sigma}\}$ is an $ILB$-system.

\end{proposition}

\begin{proof}
The proof is as in the complex case.
\end{proof}

\begin{definition}~
\begin{enumerate}
\item $\Omega^1(X,\mathfrak{g}_{\mathbb{C}})$ is the space of $\mathfrak{g}_{\mathbb{C}}$-valued $1$-forms on $X$, that is elements $\omega\in \Omega^1(X,\mathfrak{g}_{\mathbb{C}})$ are of the form $\omega(z)=f(z)dz$ with $f(z)\in X\mathfrak{g}_{\mathbb{C}}$. We define a family of norms  by $\|\omega\|_n:=|f(z)|_n$.
\item $\Omega^1(X,\mathfrak{g}_{\mathbb{C}})_\mathbb{R}$ is the space of $\mathfrak{g}_{\mathbb{C}}$-valued $1$-forms on $X$ such that $f(S^1) \subset \mathfrak{g}_{\mathbb{R}}$.
\end{enumerate}
\end{definition}

\noindent As $M\mathfrak{g}_{\mathbb{C}}$ and $M\mathfrak{g}_{\mathbb{R}}$ are  tame Fr\'echet spaces, also $\Omega^1(X,\mathfrak{g}_{\mathbb{C}})$ and  $\Omega^1(X,\mathfrak{g}_{\mathbb{C}})_\mathbb{R}$ are tame.

\begin{proposition}
Identifying $\Omega^1(M,\mathfrak{g}_{\mathbb{C}})\cong M\mathfrak{g}_{\mathbb{C}}$ and $\Omega^1(A_n,\mathfrak{g}_{\mathbb{C}})\cong A_n\mathfrak{g}_{\mathbb{C}}$ we have:
The system $\{\Omega^1(M,\mathfrak{g}_{\mathbb{C}}),\Omega^1(A_n,\mathfrak{g}_{\mathbb{C}})\}$ is an $ILB$-system.
\end{proposition}

\begin{proof}
This is equivalent to $\{M\mathfrak{g}_{\mathbb{C}}, A_n\mathfrak{g}_{\mathbb{C}}\}$ being an $ILB$-system.
\end{proof}

\begin{remark}
Real forms of the algebras $X\mathfrak{g}^{\sigma}_{\mathbb{C}}$ correspond to conjugate-linear involutions of $X\mathfrak{g}^{\sigma}_{\mathbb{C}}$: assign to a real form the conjugation with respect to it. In the other direction, fixed point algebras of conjugate-linear involutions are real forms.
Hence, real forms are closed subalgebras. Thus by an application of lemma~\ref{constructionoftamespaces} real forms of $M\mathfrak{g}_{\mathbb{C}}^{\sigma}$ are tame, real forms of $A_n{\mathfrak{g}}_{\mathbb{C}}^{\sigma}$ are Banach. 
The proof of proposition~\ref{nganginverselimit} generalizes to this setting. Thus one has $ILB$-systems.
\end{remark}

\section{Lie groups of holomorphic maps}
\label{loopgroups}

Up to now we studied analytic structures on loop algebras but not on the
associated loop groups. In short the main result is that all loop algebras interesting to us are tame Lie algebras allowing for the definition of an $ILB$-structure. In this section we prove similar results for loop groups. Let $G$ be a compact semisimple Lie group and $G_{\mathbb{C}}$ its complexification.

\subsection{Foundations}

Let us recall the definition of a smooth tame Lie group from~\cite{Hamilton82}:

\begin{definition}[smooth tame Lie group]
A smooth tame Lie group is a smooth tame Fr\'echet manifold $G$ equipped with a
group structure such that the multiplication map 
\begin{displaymath}
\varphi: G\times G\longrightarrow G,\quad (g,h)\mapsto gh
\end{displaymath}
and the inverse map
\begin{displaymath}
\varphi: G\longrightarrow G,\quad g\mapsto g^{-1}
\end{displaymath}
are smooth tame maps.
\end{definition}

In this section we will show that the following groups are smooth tame Lie groups:

\begin{definition}[complex loop groups]~ 
\label{A_nGC}
\index{complex loop group}
\begin{enumerate}
\item The loop group $A_nG$ is the group
\begin{displaymath}A_nG_{\mathbb C}:=\{f:A_n\longrightarrow G_{\mathbb C}|
\textrm{ f is holomorphic}\}\,.\end{displaymath}
The multiplication is defined to be $fg(z):= f(z)g(z)$ for $f,g \in A_nG$. 

\item The loop group $MG$ is the group 
$$MG_{\mathbb C}:=\{f:\mathbb C^*\longrightarrow G_{\mathbb C}|
\textrm{ f is holomorphic}\}\,.$$
The multiplication is defined to be $(fg)(z):= f(z)g(z)$ for $f,g \in MG$.
\end{enumerate}
\end{definition}

\begin{definition}[real form of the compact type]~
\index{holomorphic loop group of compact type}
\label{A_nGR}
\begin{enumerate}
\item The real form of the compact type $A_nG_{\mathbb{R}}$ is defined to be 
$$A_nG_{\mathbb R}:=\{f\in A_nG_{\mathbb C}| f(S^1) \subset G_{\mathbb R}\}\,.$$
\item The real form of the compact type $MG_{\mathbb{R}}$ is defined to be 
$$MG_{\mathbb R}:=\{f\in MG_{\mathbb C}| f(S^1) \subset G_{\mathbb R}\}\,.$$
\end{enumerate}
\end{definition}

\index{exponential function for loop group}
There are exponential functions 
\begin{align*}A_n\exp: A_n\mathfrak{g}&\longrightarrow A_nG\quad \textrm{and}\\
\Mexp: M\mathfrak{g}&\longrightarrow MG,
\end{align*} 
defined pointwise using the group exponential function $\exp: \mathfrak{g}\rightarrow G$:
$$
\begin{aligned}
[(A_n\exp)(f)](z)&:=\exp(f(z))\,,\\
[(\Mexp)(f)](z)&:=\exp(f(z))\,.
\end{aligned}$$

Let us remark, that for any $z\in \mathbb{C}^*$ resp. $z\in A_n$, the function  $[(\Mexp)(t\cdot f)](z)$ (resp. $[(A_n\exp)(t\cdot f)]$ is a $1$-parameter subgroup.

The next important object needed to describe the connection between the loop algebras and the loop groups is the definition of the Adjoint action $\textrm{Ad}$:

As usual it is defined pointwise using the Adjoint action of the Lie group $G_{\mathbb{K}}$, $\mathbb{K}\in \{\mathbb{R},\mathbb{C}\}$:

$$
\begin{array}{ll}
(\textrm{Ad}(A_nG)_{\mathbb{K}} \times A_n\mathfrak{g}_{\mathbb{K}})\longrightarrow A_n\mathfrak{g}_{\mathbb{K}}, \hspace{10pt}&(f,h) \mapsto fhf^{-1}\,,\\
(\textrm{Ad}(MG)_{\mathbb{K}}\times M\mathfrak{g}_{\mathbb{K}})\longrightarrow M\mathfrak{g}_{\mathbb{K}}, &(f,h) \mapsto fhf^{-1}\,,\\
\end{array}
$$
where $$fhf^{-1}(z):= f(z)h(z)f^{-1}(z)\simeq \textrm{Ad}(f(z)) (h(z))\, .$$

For the Adjoint action of groups of the compact type to be well-defined we have to check, that the condition $f(S^1)\subset \mathfrak{g}_{\mathbb{R}}$ is preserved. This is a consequence of the adjoint action for finite dimensional compact Lie groups: for all $z\in S^1$ we have $f(z)\in G_{\mathbb{R}}$ and $h(z)\in \mathfrak{g}_{\mathbb{R}}$. Thus  the condition $\textrm{Ad}(f)h(z)\in \mathfrak{g}_{\mathbb{R}}$ is preserved pointwise.

\begin{lemma}
The exponential function and the Adjoint action satisfy the identity: 
\begin{displaymath}\textrm{Ad}\circ X\exp= e^{ad}\quad \textrm{for} \quad X\in \{A_n, \mathbb{C}^*\}\, .\end{displaymath}
\end{lemma}

\begin{proof}
Applying the identity for finite dimensional Lie algebras (resp. Lie groups) we get that it is valid pointwise.\end{proof}

\noindent We will now investigate the functional analytic nature of the groups $A_nG$ and $MG$: to fix some notation let $X\sigma$ denote by abuse of notation an involution of $X\mathfrak{g}$ resp.\ of $XG$. Let $XG_D$ denote a real form of non-compact type of $XG$, and denote by $\textrm{Fix}(X\sigma)$ the fixed point group of an involution $X\sigma$.

E.\ Heintze and C.\ Gro\ss~\cite{Heintze09} show that real forms of the non-compact type of a complex simple Kac-Moody algebra are in bijection with involutions of the compact real form (which is unique up to conjugation).  Let $X\mathfrak{g}_{\mathbb{R}}$ be a compact real form with involution $X\sigma$. We denote by $X\mathfrak{g}_{D, \sigma}$ the real form of non-compact type associated to $X\sigma$.

The groups $A_nG$ are easy to understand: as $A_n$ is compact we can follow the classical strategy to define manifold and Lie group structures.  We start by defining a chart on an open set containing the identity with values in the Lie algebra via the exponential map; then we use left translation to construct an atlas of the whole group. This strategy yields the following basic results:

\begin{theorem}~
\begin{enumerate}
\item $A_nG_{\mathbb{R}}$ and $A_nG_{\mathbb{C}}$ are  Banach-Lie groups.
\item Real forms $A_nG_D$ of non-compact type of $A_nG_{\mathbb{C}}$ are Banach-Lie groups.
\item Quotients $A_nG_{\mathbb{R}}/\textrm{Fix}(A_n\sigma)$ and  $A_nG_{D, \sigma}/\textrm{Fix}(A_n\sigma)$ are Banach manifolds. 
\end{enumerate}
\end{theorem}

 For Banach-Lie groups and Banach manifolds, there is a considerable body of work; for a classical introduction see for example~\cite{Palais68}. 

\noindent For the groups $MG$ the theory is considerably more difficult. The crucial observation is the fact that the exponential map has in general not to be a local diffeomorphism.

\noindent We give an example of this strange phenomenon:

\begin{example}[$MSL(2,\mathbb C)$]
\label{sl2chasnodiffeomorphicexponentialmap}
We study the Lie group $SL(2,\mathbb{C})$. As is well known, $$\exp: \mathfrak{sl}(2,\mathbb {C})\longrightarrow SL(2,\mathbb{C})$$ is not surjective. For example, elements $g\in SL(2, \mathbb{C})$ conjugate to the element \tiny$\left(
\begin{array}{cc}
  -1 & 1\\
  0 & -1
\end{array}
\right)$ \normalsize
are not in the image of $\exp \left(\mathfrak{sl}(2,\mathbb C)\right)$.

We want to show that there exists a sequence $f_n \in MSL(2,\mathbb C)$ which
converges to the identity map in the compact-open-(tame Fr\'echet) topology but is not contained in the image of $\Mexp$. To this end, we have to construct $f_n$ in a way that it contains points that are not in the image of $\exp
\left(\mathfrak{sl}(2,\mathbb C)\right)$.
We take
\begin{align*}
f_n(z)&=
\left(
	\begin{array}{cc}e^{\pi z/n} & -i z/n\\0&e^{- \pi z/n}\end{array}
\right)\,.\\
\intertext{Then}
f_n(in)&=
\left(
	\begin{array}{cc}-1 & 1\\0&-1\end{array}
\right)\,.
\end{align*}
So $f_n$ is not contained in $\Im(\Mexp(\mathfrak{sl}(2, \mathbb C)))$. On the other hand, for all $z_0 \in \mathbb C^*$ fixed
$$\lim_{n\rightarrow \infty}f_n(z_0)=
\left(	\begin{array}{cc}1 & 0\\0&1\end{array}
\right)=\textrm{Id}\,.$$

So in the compact-open topology for every neighborhood $U_k$ of the identity there exist $n_k\in \mathbb N$ such that $\forall n\geq n_k: f_n \in U_k$.

This proves that $f_n$ is not a local diffeomorphism.
\end{example}

\noindent Especially this observation contains the corollary that the groups $MG$ are no locally exponential Lie groups in the sense of Karl-Hermann Neeb~\cite{Neeb06}, that is Lie groups such that $\exp$ is a local diffeomorphism. 

Hence, we have to find another way to define manifold structures on $MG$.
We will start by describing some results about the relationship between $MG$ and $M\mathfrak g$. Then we will show that loop groups satisfy the weaker axioms for pairs of exponential type introduced by Hideki Omori. Finally, we will investigate the structure of $(MG;A_nG)$ as $ILB$-groups.

\begin{definition}
The tangential space $T_pMG$ is defined as the space of path-equivalence classes of continuous paths.
\end{definition}

\noindent The relationship between $M\mathfrak{g}$ and $MG$ is governed by
the following three results:

\index{tangential space}
\begin{theorem}[Tangential space]
\label{tangential space}
Let $\mathfrak{g}$ be the Lie algebra of $G$. Then
$$M\mathfrak{g}= T_e(MG)\,.$$
\end{theorem}

\noindent Moreover, it is isomorphic to the algebra of left-invariant vector fields on $MG$. 

\noindent On the other hand, we find:

\begin{theorem}[Loop groups whose exponential map is no local diffeomorphism]
\label{expnodiffeomorphism}
Let $G_{\mathbb{C}}$ be a complex semisimple Lie group. 
$$\Mexp:M\mathfrak{g}\longrightarrow MG$$
is not a local diffeomorphism.
\end{theorem}

\noindent In contrast we have for nilpotent Lie groups:

\begin{proposition}[Loop groups whose exponential map is a local diffeomorphism]
\label{diffeomorphicexponentialmap}
Let $G_{\mathbb{C}}$ be a complex Lie group such that its universal cover is biholomorphically equivalent to $\mathbb{C}^n$. Then its exponential map $\Mexp$ is a local diffeomorphism.
\end{proposition}

\begin{corollary}
\label{examplesofliegroupswithdiffexp}
Let $G$ be a complex Lie group. If $G$ is nilpotent (i.e.\ abelian) then $\Mexp$ is a diffeomorphism.
\end{corollary}

\noindent Corollary~\ref{examplesofliegroupswithdiffexp} is a direct consequence of proposition~\ref{diffeomorphicexponentialmap}, as abelian and nilpotent Lie groups have $\mathbb{C}^n$ as their universal cover~\cite{Knapp96}, corollary 1.103 and \cite{Varadarajan84}, section 3.6.

This behaviour seems contradictory as the usual intuition tells us there should be a geodesic in every direction. 
From a classical analysis point of view, one can describe the analytic basics of this situation as a consequence of a change of limits, arising because of the non-compactness of $\mathbb C^*$:

\begin{itemize}
\item[-] On the one hand, theorem~\ref{tangential space} describes locally uniform convergence on each compact subset $K\subset \mathbb C^*$.
\item[-] Theorem~\ref{expnodiffeomorphism} shows in contrast that globally on $\mathbb{C}^*$ convergence needs no longer be uniform.
\end{itemize}

\noindent We now give the proofs that we have omitted:

\begin{proof}[Proof of theorem~\ref{tangential space}]~
\begin{enumerate}
\item We prove: $M\mathfrak{g}\subset T_e(MG)$, that is: let $\gamma(t): (-\epsilon, \epsilon) \rightarrow MG$ be a curve such that $\gamma(0)= e\in MG$. Then $\dot{\gamma}(0)$ is in $M\mathfrak g$.

Let $z\in \mathbb{C}^*$ arbitrary but fixed. Then $\gamma_z(t)$, the evaluation of $\gamma(t)$ at $z\in \mathbb{C^*}$, is a curve in $G$ satisfying $\gamma_z(0)=e$. Thus $\frac{d}{dt}\gamma_z(t)$ is an element in $\mathfrak{g}$. So we get a function $w: \mathbb{C}^*\longrightarrow \mathfrak{g}$; we have to prove that $w$ is holomorphic.

So let $\gamma(t):=\sum_n \gamma_n(t) z^n$. Then $\frac{d}{dt}\gamma(t)=\sum_n \frac{d}{dt}\gamma_n(t) z^n$, which is clearly holomorphic.

This shows: $T_e(MG)\subset M\mathfrak{g}$.

\item The other direction is straight forward: let $\gamma\in M\mathfrak{g}$. Then $\Mexp(t\gamma)$ is a curve in $MG$ with tangential vector $\gamma$. So $ M\mathfrak{g}\subset T_e(MG)$.

\end{enumerate}
This completes the proof. 
\end{proof}

\begin{proof}[Proof of theorem~\ref{expnodiffeomorphism}] 
The proof of theorem~\ref{expnodiffeomorphism} relies on the fact that $\Mexp$ is no local diffeomorphism for $SL(2,\mathbb{C})$. 
Let $\mathfrak{h}\simeq \mathfrak{sl}(2,\mathbb{C})\subset \mathfrak{g}$ be a subalgebra of $\mathfrak{g}$. Then there is an $H=SL(2,\mathbb{C})$ subgroup in $G$, such that $\mathfrak{h}$ is the Lie algebra of  $H$. See part VII.5 of the book~\cite{Knapp96}. 
Study the subalgebra $M\mathfrak{h}\subset M\mathfrak g$ and the subgroup $MH\subset MG$. $M\mathfrak{h}$ can be identified with $T_e(MH)$. Moreover $\Mexp: M\mathfrak{h} \subset MH$. But as example~\ref{sl2chasnodiffeomorphicexponentialmap} shows, $\Mexp$ is no local diffeomorphism. This completes the proof. 
\end{proof}

The image of the exponential map consists of those loops whose image is completely
contained in the image of the Lie group exponential map. We will give further
comments about this situation in section~\ref{geodesics}, where we study the
behaviour of geodesics.

\begin{proof}[Proof of theorem~\ref{diffeomorphicexponentialmap}]
Let $\widetilde G$ be the universal cover of $G$; the exponential map $\exp: \mathfrak{g}\longrightarrow \widetilde{G}$ is a biholomorphic map. Thus cocatenation with $\exp$ (resp.\ $\exp^{-1}$) induces a biholomorphic map between $M\mathfrak{g}$ and $M\widetilde{G}$. To get that $\Mexp: M\mathfrak{g}\longrightarrow MG$ is a local diffeomorphism, one uses the fact that each loop in $\widetilde{MG}$ projects onto a loop in $MG$ and, conversely, each loop in $MG$ can be lifted to a loop in $M\widetilde{G}$, which is unique up to Deck transformation and thus locally unique.  This proves that $\Mexp$ is a local diffeomorphism. 
\end{proof}

\subsection{Manifold structures on groups of holomorphic maps}
\label{Manifoldstructuresongroupsofholomorphicmaps}

In this section, we prove that the groups $MG_{\mathbb{R}}$, $MG_{\mathbb{C}}$, and various quotients are tame Fr\'echet manifolds.

\begin{theorem}
\label{mgcistame}
$MG_{\mathbb C}$ is a tame Fr\'echet manifold.
\end{theorem}

The idea of the proof is to use logarithmic derivatives. The idea to use logarithmic derivatives to get charts is quite common: for regular Lie groups it is developed in the book~\cite{Kriegl97}, chapters 38 and 40. 
Furthermore it is used by Karl-Hermann Neeb~\cite{Neeb06} to prove the following theorem~\cite{Neeb06}, theorem III.1.9.:

\begin{theorem}[Neeb]
\label{Neebstheorem}
Let $\mathbb{F}\in \{\mathbb{R}, \mathbb{C}\}$, $G$ be a $\mathbb{F}$-Lie group and $M$ a finite dimensional, connected $\sigma$-compact
$\mathbb{F}$-manifold. We endow the group $C_{\mathbb{F}}^{\infty}(M,G)$ with the compact open $C^{\infty}$-topology, turning it into a topological group. This topology is compatible with a Lie group structure if
$\dim_{\mathbb{F}}M=1$ and $\pi_1(M)$ is finitely generated.
\end{theorem}

The main ingredients for the proof of Neebs theorem are the use of logarithmic derivatives to define charts in $\Omega^1(M, \mathfrak{g})$, the space of $\mathfrak{g}$-valued $1$-forms on $M$ and Gl\"ockners inverse function theorem~\cite{Glockner06}, to take care of the monodromy if $\pi_1(M)$ is nontrivial.

If $\mathbb{F}=\mathbb{C}$ we have the equivalence $C_{\mathbb{F}}^{\infty}(M,G)\simeq Hol(M,G)$. Our situation is the special case $M=\mathbb{C}^*$. Hence, $\pi_1(M)=\mathbb{Z}$ is a finitely generated group.  The compact open $C^{\infty}$-topology coincides for holomorphic maps with the topology of compact convergence. 
Thus Neeb's theorem tells us that  $MG_{\mathbb F}, \mathbb{F} \in \{\mathbb{R}, \mathbb{C}\}$ are locally convex
topological Lie groups.

Nevertheless, we do not get tame structures. Hence we have to prove the theorem completely new. Our presentation follows the proof of Karl-Hermann Neeb for the locally convex case.

We need some definitions: $\alpha \in \Omega^1(M,\mathfrak{g})$ is called integrable iff there exists a function $f\in Hol(M, G)$ such that $\delta(f):=f^{-1}df=\alpha$. The uniqueness of solutions to linear differential equations shows that $\delta(f_1)=\delta(f_2)$ iff $f_1=gf_2$ for some $g\in G$.

The first step of the proof is the following lemma, whose statement and proof can be found in~\cite{Kriegl97} and \cite{Neeb06}:
morally it is a straight forward application of the monodromy principle for holomorphic Pfaffian systems, as described in the article~\cite{Novikov02}.

\begin{lemma}
Let $M$ be a $1$-dimensional complex manifold, $\alpha \in \Omega^{1}(M, \mathfrak{g})$.
\label{integrability of 1 dim }
\begin{enumerate}
	\item $\alpha$ is locally integrable
	\item If $M$ is connected, $M_0\in M$,  then there exists a homomorphism 
	$$ \textrm{per}_{\alpha}: \pi_1(M,m_o)\longrightarrow G$$
        that vanishes iff $\alpha$ is integrable.
\end{enumerate} 
\end{lemma}

\begin{proof} Proof of theorem \ref{mgcistame}
\noindent  We define the embedding:
\[
\begin{array}{ccr}
\varphi: MG_{\mathbb{C}}&\hookrightarrow& \Omega^{1}(\mathbb{C}^*, \mathfrak{g}_{\mathbb{C}}) \times G_{\mathbb{C}}\, ,\\
f &\mapsto& (\delta(f)=f^{-1}df, f(1))\, .
\end{array}
\]
This embedding is injective as $(\delta(f_1), f_1(1))=(\delta(f_2), f_2(1))$ iff $\delta(f_1)=\delta(f_2)$ and $f_1(1)=f_2(1)$. The first condition leads to the relation $f_1=g f_2$ for some $g\in G$. Then the second condition gets $f_1(1)=gf_1(1)$. Hence $g=\textrm{Id}$. 

Compare this embedding with the description of polar actions on Fr\'echet spaces in section~\ref{PolaractionsontameFrechetspaces}.
Let $\pi_1$ and $\pi_2$ denote the projections:
\[
\begin{array}{ccc}
\pi_1: \Omega^{1}(\mathbb{C}^*, \mathfrak{g}_{\mathbb{C}}) \times G_{\mathbb{C}} &\mapsto& \Omega^{1}(\mathbb{C}^*, \mathfrak{g}_{\mathbb{C}})\, ,\\
\pi_2: \Omega^{1}(\mathbb{C}^*, \mathfrak{g}_{\mathbb{C}}) \times G_{\mathbb{C}} &\mapsto& G_{\mathbb{C}}\, .
\end{array}
\]

We construct charts for $MG_{\mathbb{C}}\subset \Omega^{1}(\mathbb{C}^*, \mathfrak{g}_{\mathbb{C}}) \times G_{\mathbb{C}}$ as direct product of charts for $\pi_1 \circ \varphi(MG_{\mathbb{C}})$ and $\pi_2 \circ \varphi(MG_{\mathbb{C}})$.
\begin{itemize}
\item[-] $\pi_2 \circ \varphi$ is surjective; so describing charts for the second factor is easy: we can choose charts for $G$. Via the exponential mapping and left translation, we get charts $\psi_{2,g}: U(g)\longrightarrow V(0)$ defined on an open set $U(g)$ around $g\in G$ taking values in $V(0)\subset \mathfrak{g}_{\mathbb{C}}$. To describe the family of norms, we use the Euclidean metric $$\|\hspace{3pt}\|_n:=\|\hspace{3pt}\|_{Eucl}\, .$$ 

\item[-] The first factor is more difficult to deal with as $\pi_1 \circ \varphi$ is not surjective. While every $\mathfrak{g}_{\mathbb{C}}$-valued $1$-form $\alpha \in \Omega(\mathbb{C}^*, \mathfrak{g}_{\mathbb{C}})$ is locally integrable by lemma \ref{integrability of 1 dim }, the monodromy may prevent global integrability. 
A form $\alpha \in \Omega^1(\mathbb{C}^*, \mathfrak{g}_{\mathbb{C}})$ is in the image of $\pi_1 \circ \varphi$ iff its monodromy vanishes, that is iff
$$e^{\int_{S^1}\alpha}=e\in G_{\mathbb{C}}\,.$$

This is equivalent to the condition $\int_{S^1}\alpha = a_{-1}(\alpha) \subset \frac{1}{2\pi i} \exp^{-1}(e)$ where $a_{-1}(\alpha) $ denotes the $(-1)$-Laurent coefficient of the Laurent series of $\alpha= f(z)dz$. Hence we get the characterization of  $\Im(\pi_1\circ \varphi)$ as the inverse image of $e\in G_{\mathbb{C}}$ via the monodromy map.

Thus we have to show that this inverse image is a tame Fr\'echet manifold. To this end, we use composition with a chart $\psi: U\longrightarrow V$ for $e\in U\subset G$ with values in $G_{\mathbb{C}}$. This gives us a tame map $\Omega(\mathbb{C}^*, \mathfrak{g}_{\mathbb{C}})\longrightarrow \mathfrak{g}_{\mathbb{C}}$. This map satisfies the assumptions of theorem~\ref{submanifoldsoffrechetspace}. Thus its inverse image is a tame Fr\'echet submanifold. This proves that $\pi_1 \circ \varphi$ is a tame Fr\'echet submanifold.
\end{itemize}
\noindent Thus $MG_{\mathbb{C}}$ as a product of a tame Fr\'echet manifold with a Lie group is a tame Fr\'echet manifold.
This completes the proof of theorem~\ref{mgcistame}.
\end{proof}

\begin{theorem}
$MG_{\mathbb{C}}$ is a tame Fr\'echet Lie group.
\end{theorem}

\begin{proof}
From theorem \ref{mgcistame} we know that $MG_{\mathbb{C}}$ is tame  Fr\'echet manifold. Hence we have to check that 
\begin{displaymath}
\psi_1: MG_{\mathbb{C}}\times MG_{\mathbb{C}}\longrightarrow MG_{\mathbb{C}},\quad (g,f)\mapsto gf,\end{displaymath}
and 
\begin{displaymath}
\psi_2:MG_{\mathbb{C}}\longrightarrow MG_{\mathbb{C}},\quad f\mapsto f^{-1}
\end{displaymath}
are tame Fr\'echet maps.
Using the embedding
\[
\begin{array}{ccr}
\varphi: MG_{\mathbb{C}}&\hookrightarrow& \Omega^{1}(\mathbb{C}^*, \mathfrak{g}_{\mathbb{C}}) \times G_{\mathbb{C}}\, ,\\
f &\mapsto& (f^{-1}df, f(1))\, .
\end{array}
\]
we get for $\psi_1$ the description
\begin{align*}
\psi_1: \Omega^{1}(\mathbb{C}^*, \mathfrak{g}_{\mathbb{C}})\times\Omega^{1}(\mathbb{C}^*, \mathfrak{g}_{\mathbb{C}}) \times G_{\mathbb{C}}\times G_{\mathbb{C}}&\longrightarrow \Omega^{1}(\mathbb{C}^*, \mathfrak{g}_{\mathbb{C}}) \times G_{\mathbb{C}}\, \\
(f^{-1}df, g^{-1}dg, f(1), g(1)) &\mapsto (gf^{-1}dfg^{-1}+g^{-1}dg, f(1)g(1))\, .
\end{align*}
which is smooth as a direct product of smooth tame maps. 

Similarly we get for $\psi_2$:
\begin{align*}
\psi_1: \Omega^{1}(\mathbb{C}^*, \mathfrak{g}_{\mathbb{C}})\times G_{\mathbb{C}}&\longrightarrow \Omega^{1}(\mathbb{C}^*, \mathfrak{g}_{\mathbb{C}}) \times G_{\mathbb{C}}\, \\
(f^{-1}df,  f(1))  &\mapsto (f(f^{-1}df) f^{-1}, f^{-1}(1))\, .
\end{align*}
which is again a smooth tame map.

\end{proof}

We now investigate different classes of quotients of loop groups:

\begin{proposition}
$MG_{\mathbb{R}}$ is a tame Fr\'echet Lie group.
\end{proposition}

\begin{proof}
The proof is similar to the proof for $MG_{\mathbb{C}}$. We have only to take care of the reality condition $f(S^1)\subset G_{\mathbb{R}}$ for loops $f\in MG_{\mathbb{C}}$.
Thus the embedding $\varphi$ maps a loop $f$ into $\Omega \left(\mathbb{C}^*, \mathfrak{g}_{\mathbb{R}}\right)\times G_{\mathbb{R}}$, which are both tame Fr\'echet spaces. Now similar arguments apply.
\end{proof}

\begin{proposition}
\label{mgc/mgrtamefrechet}
The group $MG_{\mathbb{C}}/MG_{\mathbb R}$ is a tame Fr\'echet manifold.
\end{proposition}
\begin{proof}

Review the embedding 
$$\psi: MG_{\mathbb{C}}\longrightarrow \Omega^1(\mathbb{C^*}, \mathfrak{g}_{\mathbb{C}})\times G_{\mathbb{C}}\,.$$

A loop $f\cdot g$ is mapped onto $\psi(f\cdot g)=(g^{-1}f^{-1}df g+ g^{-1}dg, [f\cdot g](1))$. This is the well-known gauge-action of the group $MG_{\mathbb{R}}$, denoted by $\mathcal{G}^*(MG_{\mathbb{R}})$.
Thus there is a well defined embedding
$$\psi: MG_{\mathbb{C}}/MG_{\mathbb{R}}\longrightarrow \Omega^1(\mathbb{C^*}, \mathfrak{g}_{\mathbb{C}})/\mathcal{G}^*(MG_{\mathbb{R}})\times G_{\mathbb{C}}/G_{\mathbb{R}}\,.$$

No we study again the projections $\pi_1$ and $\pi_2$ on the first and second factor. $\pi_2$ is surjective; $G_{\mathbb{C}}/G_{\mathbb{R}}$ is a tame manifold; so this factor is no problem.

The projection on the first factor, $\pi_1$, needs a more careful analysis:
the right multiplication of $MG_{\mathbb{R}}$ on $MG_{\mathbb{R}}$ is surjective: using the decomposition $\mathfrak{g}_{\mathbb{C}}=\mathfrak{g}_{\mathbb{R}}+i \mathfrak{g}_{\mathbb{R}}$, we get for $\Omega^1(\mathbb{C}^*, \mathfrak{g}_{\mathbb{C}}):=\Omega^1_{\mathbb{R}}(\mathbb{C}^*, \mathfrak{g}_{\mathbb{C}})+i\Omega^1_{\mathbb{R}} (\mathbb{C}^*, \mathfrak{g}_{\mathbb{C}})$.  

The surjectivity of the right multiplication of $MG_{\mathbb{R}}$ on $MG_{\mathbb{R}}$ translates into the surjectivity of the $MG$-gauge action on the imaginary part $i\Omega^1_{\mathbb{R}} (\mathbb{C}^*, \mathfrak{g}_{\mathbb{C}})\cap \Im(\pi_1 \circ \psi)(MG)$. Thus we can suppose to have chosen a representative $f\in f\cdot MG$, such that the imaginary part $ \pi_1\circ \psi(f)$ is $0$.
So all we have to check is the real part. Here we find that $\exp^{-1}(e)=0$. Thus $a_{-1}=0$. So we can identify the image $\pi_1 \circ \psi (MG_{\mathbb{C}}/MG_{\mathbb{R}})\simeq \Omega^1_{\mathbb{R}}(\mathbb{C^*}, \mathfrak{g}|a_{-1}=0) $. This is a tame Fr\'echet space.

Hence proposition~\ref{mgc/mgrtamefrechet} is proven.
\end{proof}

Having proved that $MG_{\mathbb{C}}$ and $MG_{\mathbb{C}}/MG_{\mathbb{R}}$ are tame Fr\'echet
manifolds we have to check that the same is true for the quotients
$MG_{\mathbb R}/\textrm{Fix}{(\rho)}$ and $MG_D/\textrm{Fix}{(\rho)}$.

To this end, let $MG$ be a loop group and $M\rho$ the loop part of an involution of the second kind.

\begin{proposition}
Let $MG_D$ be a non-compact real form of $MG_{\mathbb{C}}$. $MG_D$ is a tame Fr\'echet manifold. 
\end{proposition}

\begin{proposition}
The quotient spaces $MG_{\mathbb{R}}/\textrm{Fix}(M\rho)$ and $MG_{D, \rho}/\textrm{Fix}(M\rho)$ are tame manifolds.
\end{proposition}

\begin{proof}
The proof is  an argument analogous to the proof that $MG_{\mathbb{C}}/MG_{\mathbb{R}}$ is a tame Fr\'echet space.
\end{proof}

\noindent The next class are twisted loop groups:

\begin{proposition}[Twisted loop groups]
Let $G$ be a compact simple Lie group of type $A_n, D_n$ or  $E_6$ and $\sigma$ a diagram automorphism of order $m\in \{2,3\}$. Let $\omega=e^{\frac{2\pi i}{m}}$.

\begin{itemize}
\item[-] The group $A_nG^{\sigma}:=\{f\in A^nG | \sigma \circ f(z)=f(\omega z)\}$ is a Banach-Lie group. 
\item[-] The group $MG^{\sigma}:=\{f\in MG | \sigma \circ f(z)=f(\omega z)\}$ is a tame manifold. Charts can be taken to be in $\Omega^1(\mathbb{C}^*, G)^{\sigma}$. Furthermore $M\mathfrak{g}^{\sigma}\simeq T_e(MG^{\sigma})$. 
\end{itemize}
\end{proposition}

\begin{proof}
To generalize the proofs of the non-twisted setting to the twisted setting one has to check that the subspaces defined by diagram automorphism are preserved by the logarithmic derivative.

\begin{enumerate}
\item For the exponential map,  we use $\sigma_{\mathfrak{g}}$ resp.\ $\sigma_{G}$ to denote the realization of the diagram automorphism $\sigma$ on $\mathfrak{g}$ resp.\ $G$. Any involution of a semisimple Lie group satisfies the identity: $\sigma_{G} \circ \exp =\exp \circ {\sigma_{\mathfrak{g}}}$
$$[\sigma_{G}\circ \Mexp(f)](z)=\exp(\sigma_{\mathfrak{g}}(f(z)))=\exp(f(\omega z))=[\Mexp(f)](\omega z),$$

\item For the logarithmic derivative we calculate:
$$\delta(\sigma \circ f)=(\sigma f)^{-1} d(\sigma \circ f)= \sigma f^{-1} \sigma df = \sigma (\delta f)\,.$$
Thus we get charts in the $\sigma$-invariant subalgebra of $\Omega^1(\mathbb{C}^*, \mathfrak{g}_{\mathbb{C}})$.
\end{enumerate}
\end{proof}

\noindent The following definition is due to Omori~\cite{Omori97}:

\begin{definition}[Exponential pair]
\label{exponentialtype}
A pair $(G,\mathfrak{g})$ consisting of a Fr\'echet group $G$ and a Fr\'echet space $\mathfrak{g}$ is called a topological group of \emph{exponential type} if there is a continuous mapping:
$$\exp: \mathfrak{g}\longrightarrow G$$
such that:
\begin{enumerate}
	\item For every $X \in \mathfrak{g}$, $\exp(sX)$ is a one-parameter subgroup of $G$.
	\item For $X,Y \in \mathfrak{g}$, $X=Y$ iff $\exp(sX)=\exp(sY)$ for every $s\in \mathbb{R}$.
	\item For a sequence $\{X_n\}\in \mathfrak{g}$, $\displaystyle\lim_{n\rightarrow \infty} X_n$ converges to an element $X\in \mathfrak{g}$ iff $\displaystyle\lim_{n\rightarrow \infty} (\exp{sX_n})$ converges uniformly on each compact interval to the element $\exp(sX)$.
	\item There is a continuous mapping $\textrm{Ad}: G\times \mathfrak{g}\longrightarrow \mathfrak{g} $ with  $h\exp(sX)h^{-1}=\exp s \textrm{Ad}(h)X$ for every $h\in G$ and $X\in \mathfrak{g}$.
\end{enumerate}
\end{definition}

\begin{proposition}[Exponential type]
The pair $(MG_{\mathbb{K}}, M\mathfrak{g}_{\mathbb{K}})$ is of exponential type.
\end{proposition}

\begin{proof}
We proved that $MG$ is a tame Fr\'echet Lie group, thus a topological group.
To prove that $(MG_{\mathbb{K}}, M\mathfrak{g}_{\mathbb{K}})$ is of exponential type, we have to check the four conditions given in definition~\ref{exponentialtype}:

\begin{enumerate}
	\item The first condition can be checked by a pointwise analysis: for $f\in M\mathfrak{g}$ and for every $z\in \mathbb{C}^*$, the curve $\exp{sf(z)}$ is a $1$-parameter subgroup in $G$. This pieces together for all $z\in \mathbb{C^*}$, to yield the condition.
	\item The second condition follows analogously: let $X,Y\in M\mathfrak{g}$. $X=Y$ iff $X(z)=Y(z)$ for all $z \in \mathbb{C}^*$. The finite dimensional theory tells us that this is equivalent to the curves $\exp(sX(z))\subset G_{\mathbb{K}}$ and $\exp(sY(z))\subset G_{\mathbb{K}}$ to be equivalent for all $z\in \mathbb{C}^*$, but this is equivalent to $\exp(sX)=\exp(sY)$.
	\item Let $\{X_n(z)\}\in M\mathfrak{g}, z\in \mathbb{C}$ be a sequence of elements such that  $\displaystyle\lim_{n\rightarrow \infty} X_n(z)=X \in M\mathfrak{g}$. Let $T\subset \mathbb{R}$ be a compact interval, $s\in T$. 
	As we have on $MG$ the compact-open topology, 
	$$\lim_{n\rightarrow \infty}(\exp{sX_n})=\exp(sX)\Leftrightarrow \forall K \subset \mathbb{C^*}:\exp\lim_{n\rightarrow \infty}(\exp{sX_n(K)})=\exp(sX)(K)\,.$$
	This assertion is correct as for every $z\in K:\exp\displaystyle\lim_{n\rightarrow \infty}(\exp{sX_n(z)})=\exp(sX)(z)$.
	\item The last assertion follows again from pointwise consideration and the validity of the assertion for finite semisimple Lie groups.\qedhere
\end{enumerate}
\end{proof}

\begin{theorem}[$ILB$-manifold]
\label{ILB-manifold}
$MG_{\mathbb{K}}$ carries the structure of an $ILB$-manifold. As an $ILB$-manifold, it is modelled on the $ILB$-system $\{M\mathfrak{g}_{\mathbb{K}}, A_n\mathfrak{g}_{\mathbb{K}}\}$. 
\end{theorem}

\begin{proof} This is a consequence of theorem~\ref{f0manifoldisilb} as the atlases for loop groups are $(0,b, C(n))$-tame.
\end{proof}

Obviously a similar result holds for 
\begin{enumerate}
\item non-compact real forms $\{MG_{D,\rho}; A_nG_{D,\rho} \}$,
\item the quotient spaces $MG_{\mathbb{R}}/\textrm{Fix}(M\rho)$, $\{MG_{D,\rho}/\textrm{Fix}(M\rho)$
\item the twisted versions of all described objects.
\end{enumerate}

\noindent To prove this, one has to check in every step that the restrictions defined by the involutions are compatible.

\begin{remark} The $ILB$-structure gives an easy interpretation of the properties of the exponential function:
the exponential map
$A_n\exp^{(n)}:A_\mathfrak{g}^{(n)} \longrightarrow A_G^{(n)}$ is defined by composition with the group exponential function $\exp: \mathfrak{g} \rightarrow G$. By the inverse function theorem for Banach spaces it is a local diffeomorphism. Thus for every $n\in \mathbb N$ there exist $U^{(n)}$, $V^{(n)}$ such that the map
$A_n\exp^{(n)}$ is a local diffeomorphism $A_n\exp^{(n)}: U^{(n)}\longrightarrow V^{(n)}$.
Of course in the limit $n\rightarrow \infty$, we get only that $\Mexp$ is a diffeomorphism of $\bigcap
U^{(n)}$ onto $\bigcap V^{(n)}$, which need no longer be open.

The obvious solution to this problem is the introduction of an additional assumption:
there are open neighborhoods 
$U^{\infty}\subset M\mathfrak{g}$ and $V^{\infty}\subset MG$, such that $U^{\infty}\subset \left(\bigcap U^{(n)}\right)\cap M\mathfrak{g}$ and 
$V^{\infty} \subset \left(\bigcap V^{(n)}\right)\cap MG$ for all $n$. This assures that the intersection is open. 

This corresponds to the additional condition of invertibility on an open subset in the Nash-Moser inverse function theorem. 
\end{remark}

We give some remarks about $1$-parameter subgroups. 

\begin{remark}
Let $g(t):=X\exp(tu)$ for $u\in X\mathfrak{g}_{\mathbb{K}}$ be a $1$-parameter subgroup in $XG_{\mathbb{K}}$, $X\in \{A_n , \mathbb{C}^*\}$, $\mathbb{K}\in \{\mathbb{R}, \mathbb{C}\}$. Then the following statements hold:
\begin{enumerate}
\item  $X\exp(tu)_{z_0}$ is a $1$-parameter group in $G_{ \mathbb{C}}$ for all $z_0\in X$.
\item If $A_n\subset A_{n+k}$ then the embedding $A_{n+k}G \hookrightarrow A_nG$ maps $1$-parameter subgroups onto $1$-parameter subgroups.
\end{enumerate}
\end{remark}

\begin{proof} Direct calculation. \end{proof}

\begin{remark}
As we have seen, the fact that the exponential function does not define a local diffeomorphism is responsible for several difficulties; so it is reasonable to try to use a setting in which the exponential function defines a local diffeomorphism. 
So let us try to take loops $f: S^1\longrightarrow G$ satisfying some regularity condition. In this case it is easy to
see that the exponential map defines always a local diffeomorphism, as a neighborhood of
the identity element of such a loop group is given by loops whose images lie in
a small neighborhood $V$ of the identity of the subjacent Lie group; this
neighborhood can be chosen in a way such that the group exponential is a
diffeomorphism from an open neighborhood $U$ in the Lie algebra onto it. But now other problems appear:
\begin{enumerate}
\item Suppose the functions to be $H^1$-Sobolev loops. In this setting, one can construct weak Hilbert symmetric spaces of compact and non-compact type. Nevertheless, one cannot define the double extension corresponding to the $c$- and $d$-part of the Kac-Moody algebra. As this extension is responsible for the structure theory, this setting is not useful for us. 
\item To be able to construct the extension corresponding to the derivative $d$, one needs loops that are $C^{\infty}$. For $C^{\infty}$-loops, it is possible to construct compact type symmetric spaces corresponding to the finite dimensional types $I$ and $III$, but for Kac-Moody symmetric spaces of $C^{\infty}$-loops there is no dualization: as the complexification of $c$ is not defined, we cannot complexify. As a consequence there are no symmetric spaces of the non-compact type. 
\end{enumerate}
The details of both theories are developed in~\cite{Popescu05}. Hence we have:

\begin{proposition}
The setting of holomorphic loops is the biggest setting such that Kac-Moody symmetric space of the compact type and of the non-compact type of the same regularity condition can be defined.
\end{proposition}

Let us mention in this context the short summary in \cite{Berger03} about infinite dimensional differential geometry, where the conflict between easy Hilbert space structures and  good metrics is addressed.
\end{remark}

\section{Polar actions on tame Fr\'echet spaces}
\label{PolaractionsontameFrechetspaces}

The isotropy representation of a (finite dimensional) Riemann symmetric space is a polar representation of the isotropy group on the tangential space. As a section one can choose any maximal flat subspace. We will see in chapter~\ref{chap:symm} that Kac-Moody symmetric spaces behave in a similar way. Nevertheless, there is a striking difference: there are orbits with finite codimension and orbits with infinite codimension. We will see, that the orbits with finite codimension correspond to gauge actions of tame loop groups on tame spaces.

Closely related is the theory of polar actions on Hilbert spaces, which is described in the article~\cite{terng95}.

The fundamental theorem due to C.-L.\ Terng states:
\index{$P(G,H)$-action}
\begin{theorem}
Define $P(G,H):=\{g\in H^1([0,1], G)|(g(0), g(1))\in H\subset G\times G\}$ and $V=H^0([0,1],\mathfrak{g})$. Suppose the $H$-action on $G$ is polar with flat sections. Let $A$ be a torus section through $e$ and let $\mathfrak{a}$ denote its Lie algebra. Then the gauge action of $P(G,H)$ on $V$ is polar with section $\mathfrak{a}$.
\end{theorem}

\begin{proof}see~\cite{terng95}. \end{proof}

Important special cases are the following: let $\Delta_{\sigma}\subset G\times G$ denote the $\sigma$-twisted diagonal subgroup of $G\times G$, that is: $(g,h)\in \Delta_{\sigma}$ iff $h=\sigma(g)$. We use the notation $\Delta=\Delta_{\textrm{Id}}$ for the non-twisted subgroup. 

\index{gauge action}
\begin{enumerate}
\item The gauge action of $H^1$-Sobolev loop groups $P(G, G\times G)\cong H^1([0,1],G)$ on their $H^0$-Sobolev loop algebras $H^0([0,1],\mathfrak{g})$ is transitive.
\item The gauge action of $P(G, \Delta_{\sigma})$ on $H^1([0,1],G)$  is polar with flat sections \cite{HPTT}.
\item The gauge action of $P(G, K\times K)$ on $H^1([0,1],G)$ where $K$ is the fixed point set of some involution of $G$ is polar with flat sections \cite{HPTT}.
\item The gauge action of $P(G, K_1\times K_2)$ on $H^1([0,1],G)$ where $K_i, i\in \{1,2\}$ are the fixed point groups of involutions of $G$ is polar with flat sections \cite{HPTT}.
\item The gauge action of Sobolev-$H^1$-loop groups $H^1(S^1,G)$ on their Sobolev $H^0$-loop algebras $H^0(S^1,\mathfrak{g})$-is polar \cite{PalaisTerng88}. 
\end{enumerate}

In this section we describe a similar theory for the loop groups $XG^{\sigma}$ on the tame  loop algebras $X\mathfrak{g}^{\sigma}$. As usual let $X\in \{A_n, \mathbb{C}^*\}$. Holomorphic functions on $A_n$ are supposed to be holomorphic in an open set containing $A_n$. From the embedding $XG^{\sigma}\hookrightarrow H^1([0,1],G)$ and $X\mathfrak{g}^{\sigma}\hookrightarrow H^0([0,1],\mathfrak{g})$ it is clear that the algebraic part of the theory works exactly the same in all regularity conditions. This means for example: sections for holomorphic actions  correspond to sections for the Hilbert actions and the associated affine Weyl groups are the same. Hence the crucial point is to check that the additional regularity restrictions fit together. This breaks down to two points:

\begin{enumerate}
\item One has to show locally that the additional regularity conditions are satisfied. 
\item One has to show globally that additional monodromy conditions are satisfied. 
\end{enumerate}

To define orthogonality on $X\mathfrak{g}$ we use the $H^0$-scalar product induced on $M\mathfrak{g}$ by the embedding into $H^0([0,1],\mathfrak{g})$. Hence we can define polar actions on Banach (resp.\ tame) spaces like that:

\begin{definition}
An action of a Lie group $G$ on a Fr\'echet space $F$ is called polar iff there is a subspace $S$, called a section, intersecting each orbit orthogonally with respect to some scalar product. 
\end{definition}

\begin{theorem}
\label{polaractiononmg}
The gauge action of $XG_{\mathbb{R}}^{\sigma}$ on $X\mathfrak{g}^{\sigma}$ is polar; an abelian subalgebra $\mathfrak{a} \subset \mathfrak{g}$ interpreted as constant loops is a section.
\end{theorem}

The proof consists of two parts:
\begin{enumerate}
\item We have to show that each orbit intersects the section $\mathfrak{a}$.  
\item We have to show that the intersection is orthogonal.
\end{enumerate}

\noindent The second part follows trivially from the embedding and Terng's result.

\noindent Thus we are left with proving the first assertion. We do this in a step-by-step way: first we study the action of $C^{k}$-loop groups on $C^{k-1}$-loop algebras ($k\in \{\mathbb{N}, \infty\}$). Then we proceed to the holomorphic setting of theorem~\ref{polaractiononmg}.

\begin{lemma}
\label{cinfty}
The gauge action of $L^{k}G^{\sigma}$ on $L^{k-1}\mathfrak{g}^{\sigma}$ is polar for $k\in \{\mathbb{N, \infty}\}$.
\end{lemma} 

\noindent This result is used without proof in~\cite{Popescu05} in order to show that all finite dimensional flats are conjugate. We do not know if a proof can be found in the literature. For completeness we give one: 

\begin{proof}[Proof of lemma~\ref{cinfty}]~
\begin{enumerate}
\item  {\bf Orthogonality} in $L^{k-1}\mathfrak{g}^{\sigma}$ is defined via the embedding into the space $H^0([0,1],\mathfrak{g})$ and the use of the $H^0$-scalar product. Hence orthogonality of the intersection between sections and orbits is covered by Terng's result.

\item {\bf Local regularity}~

\noindent Define the following spaces  

\begin{displaymath}
P(G,H)^k:=\{g\in C^k([0,1], G)|(g(0), g(1))\in H\subset G\times G\}\, .
\end{displaymath}

 Furthermore we use the equivalence~\cite{terng95}
\begin{align*}
P(G;e\times G)&\simeq H^0([0,1],\mathfrak{g})\, ,\\
h&\leftrightarrow -h'h^{-1}\, .
\end{align*}

Terng's polarity result~\cite{terng95} yields that the action of $P(G,\Delta_{\sigma})$ on $P(G; e\times G)$ defined by $(g(t), h(t))\mapsto g(t) h(t) g(0)^{-1}$ is polar with  a section of constant loops $\exp{t\mathfrak{a}}$ where $\mathfrak{a}$ is a maximal abelian subalgebra in $\mathfrak{g}$ (if $\sigma\not=0$ we restrict to $\mathfrak{a}_{\sigma}$ and omit the $\sigma$ in the notation~\cite{Kac90}). Thus for every $h(t)\in P(G;e\times G)$ there exist $g(t)\in P(G,\Delta_{\sigma})$ and $X\in \mathfrak{a}$, such that $g(t)h(t)g(0)^{-1}=\exp(tX)$.  

Rearranging this equation we deduce for any  loop $ g(t) \in P(G, \Delta_{\sigma})$ the explicit description $g(t):= \exp(tX)g(0)h(t)^{-1}$. Hence if $h(t)\in P^k(G;e\times G)$ then $g(t)\in  P^k(G,\Delta_{\sigma})$. Combining this with the orthogonality we obtain that the actions of $P^k(G,\Delta_{\sigma})$ on $P^k(G;e\times G)\simeq H^{k-1}([0,1],\mathfrak{g})$ and of $P^{\infty}(G,\Delta_{\sigma})$ on $P^{\infty}(G;e\times G)\simeq H^{\infty}([0,1],\mathfrak{g})$ are polar.

\item {\bf The periodicity relation:}~
We want to show that $L^{k}G^{\sigma}$ acts on $L^{k-1}\mathfrak{g}^{\sigma}$ with slice $\mathfrak{a}$ for $k\in \{\mathbb{N}, \infty\}$.

\begin{enumerate}
\item Let first $g\in L^{k}G^{\sigma}$ and $u\in L^{k-1}\mathfrak{g}^{\sigma}$. Then $g \cdot u= gug^{-1}-g' g^{-1}$ is in $L^{k-1}\mathfrak{g}^{\sigma}$. Thus $L^{k}G^{\sigma}$ acts on $L^{k-1}\mathfrak{g}^{\sigma}$.

\item We have to show that any $L^{k}G^{\sigma}$-orbit intersects the section $\mathfrak{a}$. This is equivalent to:
\noindent  For each $u\in L^{k-1}\mathfrak{g}^{\sigma}$, there is $X\in \mathfrak{g}$ and $g\in P^k(G,\Delta)$ such that
$\exp(tX)=g(t)h(t)g^{-1}(0)$ with $h'(t)=u(t)h(t)$ and the derivatives coincide; interpret in this last equation $u(t)$ as a quasi-periodic function on $\mathbb{R}$ (i.e. $u(t+2\pi)=\sigma u(t)$ and $h(t)$ as a function on $\mathbb{R}$). 

\noindent Using the first part, we find a function $g(t)\in P^k(G,\Delta_{\sigma})$. Hence, what remains is to check the closing condition of the derivatives: $g^{(n)}\cdot u(2\pi)= \sigma g^{(n)}\cdot u(0)$. We will prove that it is equivalent to the closing condition $g^{(n+1)}(2\pi)=\sigma g^{(n+1)}(0)$.

We start with the case $n=1$.
For this case, we have to show
$$\exp((t+2\pi)X) g_0 h(t+2\pi)^{-1}= \sigma (\exp(tX)g_0 h(t)^{-1})\,. $$ 

After rearranging, this is equivalent to the identity 
$$\sigma(g_0^{-1})\exp(2\pi X) g_0=\sigma(h(t)^{-1})h(t+2\pi)\,.$$
As the left side is a constant we find:
$$\left(\sigma(h(t)^{-1})h(t+2\pi)\right)'=0\,.$$ 
Hence: $-\sigma(h(t)^{-1}h'(t) h(t)^{-1})h(t+2\pi)+ \sigma(h(t)^{-1})h'(t+2\pi)=0$.
Rearranging this equality we get
$$\sigma(u(t))= - \sigma(h'(t))\sigma( h(t)^{-1})= -h'(t+2\pi)h(t+2\pi)^{-1}=u(t+2\pi)\,$$
which is the desired periodicity condition. 

For $n\not=0$ we use induction. If $g$ is $k$-times differentiable then $u$ is $k-1$-times differentiable. 
This proves the lemma.\qedhere
\end{enumerate}
\end{enumerate}
\end{proof}

\begin{proof}[Proof of theorem~\ref{polaractiononmg}]
To prove the theorem, we have to further strengthen the used regularity conditions to holomorphic functions. 
The description in the proof of lemma~\ref{cinfty} shows that the group of analytic loops $L_{an}(S^1, G)$ acts polarly with section $\mathfrak{a}$ on the algebra $L_{an}(S^1, \mathfrak{g})$ of analytic loops. 
\begin{enumerate}
\item {\bf The case of holomorphic loops on $\mathbb{C}^*$}~
For the specialization to holomorphic maps we use the description:
$$H_{\mathbb{C}^*}(\mathbb{C},\mathfrak{g}_{\mathbb{C}})_{\mathbb{R}}:=\{f:\mathbb{C}\longrightarrow \mathfrak{g}_{\mathbb{C}}|f(z+i\mathbb{Z})=f(z), f{i\mathbb{R}}\subset \mathfrak{g}_{\mathbb{R}}\}\,.$$
\noindent Identifying $it \leftrightarrow t$ in this (resp.\ the above) description we get an embedding: 
  $$H_{\mathbb{C}^*}(\mathbb{C},\mathfrak{g}_{\mathbb{C}})_{\mathbb{R}}\hookrightarrow H^{\infty}(S^1,\mathfrak{g})$$
This shows that there are no problems concerning the monodromy. So we have only to check the regularity aspect. 
For $g\in MG$ and $u\in M\mathfrak{g}$, $g\cdot(u)=gug^{-1}-g'g^{-1}\in M \mathfrak{u}$. 
On the other hand, using the description in the proof of lemma~\ref{cinfty}, we get for $u\in M\mathfrak{g}$ a transformation function $g(t):= \exp(tX)g(0)h(t)^{-1}$. A priori this function is in $L(S^1,G)$; but $\exp(tX)$ can be continued to a holomorphic function on $\mathbb{C}$, $g(0)$ is a constant and $h(t)^{-1}$ is a solution of the differential equation: $h'(t)=u(t)h(t)$; if $u(t)$ is defined on $\mathbb{C}^*$, this equation has a solution on the universal cover of $\mathbb{C}^*$, that is $\mathbb{C}$. So $g(t)$ is defined on $\mathbb{C}$, but has perhaps nontrivial monodromy; this is, of course, not possible, as the embedding tells us that $g(t)\subset L_{an}G$. 
\item {\bf The case of holomorphic loops on $A_n$}~
This case is exactly similar. $\mathbb{C}$ is replaced by $A'$ (compare subsection~\ref{A_n}).
\end{enumerate}
Hence theorem~\ref{polaractiononmg} is proved.
\end{proof}

Thus we have proven that $\sigma$-actions and diagonal actions are polar. Those two cases corresponds to the isotropy representation of Kac-Moody symmetric spaces of types $II$ and $IV$: the diagonal action is induced by the isotropy representation of Kac-Moody symmetric spaces in the non-twisted case, the $\sigma$-action is induced by the isotropy representation of $\sigma$\ndash twisted one.

The isotropy representations of Kac-Moody symmetric spaces of type $I$ and $III$ correspond to the Hermann examples. 
A holomorphic version of the Hermann examples~\cite{HPTT} can be defined in exactly the same way: 

Let $XG_{\mathbb{R}}^{\sigma}$ be a simply connected loop group, $\rho$ an involution such that $X\mathfrak{g}_{\mathbb{R}}^{\sigma}=\mathcal{K}\oplus \mathcal{P}$ is the decomposition into the $\pm 1$-eigenspaces of the involution induced by $\rho$ on $X\mathfrak{g}_{\mathbb{R}}^{\sigma}$. Let $XK_{\mathbb{R}}\subset XG_{\mathbb{R}}^{\sigma}$ be the subgroup fixed by $\rho$.  

\begin{theorem}
\label{polaractiononp}

The gauge action of $XK_{\mathbb{R}}$ on $\mathcal{P}$ is polar.
\end{theorem}

\begin{proof}
The proof is like the one of theorem~\ref{polaractiononmg}. One starts with a similar result for polar actions on Hilbert action~\cite{terng95} and checks then step by step that the introduced higher regularity conditions fit together. 
\end{proof}

\section[Tame structures and $ILB$-structures]{Tame structures and $ILB$-structures on Kac-Moody algebras}
\label{holomorphicstructuresonkacmoodyalgebras}

In this section we will describe explicit realizations as central extensions of holomorphic loop algebras of the abstract affine geometric Kac-Moody algebras which we introduced in definition~\ref{geometricaffinekacmoodyalgebra}.

\noindent Let $X\in \{A_n, \mathbb{C}^*\}$. As usual, holomorphic functions on $A_n$ are understood to be holomorphic in an open set containing $A_n$.

\begin{definition}[holomorphic affine geometric Kac-Moody algebra]
Define $\widehat{X\mathfrak{g}}$ to be the realization  of $\widehat{L}(\mathfrak{g}, \sigma)$ with $L(\mathfrak{g}, \sigma)$ replaced by $X\mathfrak{g}^{\sigma}$.
\end{definition}

Thus an element of a Kac-Moody algebra can be represented by a triple
$\left(f(z), r_c, r_d \right)$, where $f(z)$ denotes a $\mathfrak{g}_{\mathbb{C}}$-valued holomorphic function on $X$ and  $\{r_c, r_d\}\in \mathbb{C}$.

We equip those algebras with the norms $\|(f,r_c, r_d)\|_n:= \sup_{z\in A_n} |f_z| +(r_c \bar{r}_c+r_d \bar{r}_d)^{\frac{1}{2}}$. Thus we use the supremum norm on the loop algebra and complete it with an Euclidean norm on the double extension defined by $c$ and $d$.

\begin{lemma}[Banach- and Fr\'echet structures on Kac-Moody algebras]~
\label{banachandfrechetstructureonkacmoodyalgebras}
\begin{enumerate}
	\item For each $n$, the algebras $\widehat{A_n\mathfrak{g}}_{\mathbb{R}}^{\sigma}$ and $\widehat{A_n\mathfrak{g}}_{\mathbb{C}}^{\sigma}$ equipped with the norm $\|\hspace{3pt} \|_n$ are Banach-Lie algebras,
	\item $\widehat{M\mathfrak{g}}_{\mathbb{R}}^{\sigma}$ and $\widehat{M\mathfrak{g}}_{\mathbb{C}}^{\sigma}$ equipped with the sequence of norms $\|\hspace{3pt} \|_n$ are tame Fr\'echet-Lie algebras. 
\end{enumerate}
\end{lemma}

\begin{proof}
Let $\mathbb{F}\in \{\mathbb{R}, \mathbb{C}\}$.
\begin{enumerate}
\item As a consequence of lemma~\ref{mgfrechetagbanach}, $A_n\mathfrak{g}^{\sigma}$ is a Banach space. Thus $\widehat{A_n\mathfrak{g}}^{\sigma}$ is Banach.

To prove that $ad(f+r_cc+r_dd)$ is continuous, we use  $[f+r_cc+r_dd,g+s_cc+s_dd]= [f,g]+r_d[d,g]-s_d[d,f]=[f,g]_0+\omega(f,g)c+r_d izg'-s_diz f'$. 
Continuity follows from the continuity of $\frac{d}{dz}$, which is a consequence of the Cauchy-inequality and the boundedness of multiplication on compact domains.

\item $M\mathfrak{g}^{\sigma}$ is a tame Fr\'echet space as a consequence of lemma~\ref{mgfrechetagbanach}.  Thus $\widehat{M\mathfrak{g}}^{\sigma}$ is tame as a direct product of tame spaces (lemma~\ref{constructionoftamespaces}). To prove tameness of the adjoint action, we need tame estimates for the norms. Those estimates follow directly from the Banach space situation:
\begin{alignat*}{2}
&\|ad(f+r_cc+r_dd)(g+s_cc+s_dd)\|_n =\\
&= \|[f,g]_0+\omega(f,g)c+r_d izg'-s_diz f'\| \leq\\
&\leq 2\|f\|_n \sup_{z\in A_n}|g(z)|+\|\omega(f,g)c\|+\|r_d izg'\|+\|s_diz f'\|_n\leq\\
&\leq 2\|f\|_n \|g(z)\|_n+2\pi \|f\|_n \| g' \|_n+|r_d| \|z\|_n\|g'\|_n+|s_d| \|z\|_n\| f'\|_n\leq\\
&\leq  2\|f\|_n \|g\|_n + 2\pi\|f\|_n \frac{e^{n+1}}{e-1}  \|g\|_{n+1}+e^n\|f\|_n\|\frac{e^{n+1}}{e-1} \|g\|_{n+1} + e^n \|s_d\|\frac{e^{n+1}}{e-1} \| f\|_{n+1} \leq\\
&\leq\left(2+2\pi\frac{e^{n+1}}{e-1}+ 2\frac{e^{2n+1}}{e-1}  \right) \|g\|_{n+1}\|f\|_{n+1}\leq\\
&\leq 6\pi e^{2n+1}\|g\|_{n+1}\|f\|_{n+1}\,.
\end{alignat*}

Thus $ad(\widehat{g})$ is $(1, 0, 6\pi e^{2n+1}\|g\|_{n+1})$-tame.\qedhere
\end{enumerate}
\end{proof}

\noindent This result shows that the tame structure on the Kac-Moody algebra is preserved by the adjoint action. 

\noindent For analytic details and the Cauchy-inequalities see for example~\cite{Berenstein91}.

\begin{theorem}
The system $\{\widehat{M\mathfrak{g}}^{\sigma},\widehat{A_n\mathfrak{g}}^{\sigma}\}$ is an $ILB$-system.

\end{theorem}

\begin{proof}
According to lemma~\ref{nganginverselimit}, the system $\{M\mathfrak{g}^{\sigma}, A_n\mathfrak{g}^{\sigma}\}$ is an $ILB$-system. The same is true for  $\{\widehat{M\mathfrak{g}}^{\sigma},\widehat{A_n\mathfrak{g}}^{\sigma}\}$ as it is a direct product of an $ILB$-system with a finite dimensional subspace.
\end{proof}

Nevertheless the adjoint action on $\widehat{M\mathfrak{g}}^{\sigma}$ does not induce an adjoint action on each algebra $\widehat{A_n\mathfrak{g}}^{\sigma}$, but only an $ILB$-(1,0)-regular map. This is a consequence of lemma~\ref{banachandfrechetstructureonkacmoodyalgebras}. This means that $\widehat{A_n\mathfrak{g}}^{\sigma}$ is not a Lie algebra as the bracket is not defined for all elements.

\begin{remark}
Of course we could also define a weak Lie algebra structure with the Lie bracket defined only on a (dense) subset of $A_n\mathfrak{g}^{\sigma}\times A_n\mathfrak{g}^{\sigma}$.
\end{remark}

\section{Kac-Moody groups}
\label{Kac-Moody groups}

\subsection{Construction}
\index{Kac-Moody group}

In this section we describe the construction of affine Kac-Moody groups. Affine Kac-Moody groups can be realized as extensions of loop groups. Our presentation follows the outline of~\cite{PressleySegal86} - nevertheless we use a different functional analytic setting. Furthermore using a technical result of B.\ Popescu we prove that Kac-Moody groups of holomorphic loops carry a structure as tame Fr\'echet manifolds.

As usual $G_{\mathbb{C}}$ is a complex semisimple Lie group and $G$ its compact real form. As the constructions are valid in Kac-Moody groups defined with respect to various different regularity conditions, we use the regularity-independent notation $L(G_{\mathbb{C}}, \sigma)$ for the complex loop group and $L(G, \sigma)$ for its real form of compact type. To define groups of polynomial or analytic loops, we use the fact that every compact Lie group is isomorphic to a subgroup of some unitary group. Hence we can identify it with a matrix group. Similarly the complexification can be identified with a subgroup of some general linear group~\cite{PressleySegal86}. Groups of polynomial loops are then defined with respect to this representation. 

\noindent Kac-Moody groups are constructed in two steps.
\begin{enumerate}
\item The first step consists in the construction of an $S^1$-bundle in the real case (resp.\ a $\mathbb{C}^*$-bundle in the complex case) that corresponds via the exponential map to the central term $\mathbb{R}c$ (resp.\ $\mathbb{C}c$) of the Kac-Moody algebra.
\item In the second step we construct a semidirect product with $S^1$ (resp.\ $\mathbb{C}^*$). This corresponds via the exponential map to the $\mathbb{R}d$-\/ (resp.\ $\mathbb{C}d$-) term
\end{enumerate}

Study first the extension of $L(G, \sigma)$ with the short exact sequence:
\begin{displaymath}
1 \longrightarrow S^1\longrightarrow \widetilde{L}(G,\sigma) \longrightarrow L(G, \sigma) \longrightarrow 1\,.
\end{displaymath}

There are various groups $X$ that fit into this sequence. We need to define $\widetilde{L}(G,\sigma)$ in a way that its tangential Lie algebra at $e\in \widetilde{L}(G, \sigma)$ is isomorphic to $\widetilde{L}(\mathfrak{g}, \sigma)$.

As described in~\cite{PressleySegal86} this $S^1$-bundle is best represented by triples: take triples $(g(z), p(z,t), w)$ where $g(z)$ is an element in the loop group, $p(z,t)$ a path connecting the identity to $g(z)$ and $w\in S^1$  (respective $w\in \mathbb C^*$) subject to the relation of equivalence: $(g_1(z), p_1(z,t), w_1) \sim (g_2(z), p_2(z,t), w_2)$ iff $g_1(z)= g_2(z)$ and $w_1= C_{\omega}(p_2*p_1^{-1})w_2$. The term $w_1= C_{\omega}(p_2*p_1^{-1})w_2$ defines a twist of the bundle. Here we put: 
\begin{displaymath}C_{\omega}(p_2*p_1^{-1})=e^{\int_{S\left(p_2*p_1^{-1}\right)}\omega}\end{displaymath}
 where $S(p_2*p_1^{-1})$ is a surface bounded by the closed curve $p_2*p_1^{-1}$ and $\omega$ denotes the $2$-form used to define the central extension of $L(\mathfrak{g}, \sigma)$. The law of composition is defined by 
\begin{displaymath}(g_1(z), p_1(z,t), w_1)\cdot (g_2(z), p_2(z,t), w_2)=(g_1(z)g_2(z),\ p_1(z,t)*g_1(z)\cdot p_2(z,t), w_1 w_2)\,.\end{displaymath}

If $G$ is simply connected it can be shown that this object is a well defined group independent of arbitrary choices made in the construction iff $\omega$ is integral. This condition is satisfied by our definition of $\omega$~\cite{PressleySegal86}, theorem 4.4.1.
If $G$ is not simply connected, the situation is a little more complicated: let $G=H/Z$ where $H$ is a simply connected Lie group and $Z=\pi_1(G)$. Let $(LG)_0$ denote the identity component of $LG$. We can describe the  extension using the short exact sequence:
\begin{displaymath}1\longrightarrow S^1\longrightarrow \widetilde{LH}/Z\longrightarrow (LG)_0\longrightarrow 1\end{displaymath}see~\cite{PressleySegal86}, section 4.6.\, .

\noindent In case of complex loop groups, the $S^1$-bundle is replaced by a $\mathbb{C}^*$-bundle.

Hence, we can now give the definition of Kac-Moody groups:

\begin{definition}[Kac-Moody group]~
\begin{enumerate}
\item The real Kac-Moody group $\widehat{MG}_\mathbb{R}$ is the semidirect product of
 $S^1$ with the $S^1$-bundle $\widetilde{MG}_{\mathbb R}$.

\item The complex Kac-Moody group $\widehat{MG}_{\mathbb C}$ is its complexification: a semidirect product
of $\mathbb C^*$ with $\widetilde{MG}_{\mathbb C}$-bundle over
$MG$. 
\end{enumerate}
The action of the semidirect $S^1$ (resp.\ $\mathbb{C}^*$)-factor is in both cases defined by a shift of
the argument: 
\begin{displaymath}
\mathbb C^* \ni w_d: MG \rightarrow MG: f(z)\mapsto f(zw_d)\, .
\end{displaymath}
\end{definition}

\begin{remark}
Remark that in the compact case the shift is by elements $w_d=e^{i\varphi_d}$ only. Hence the action preserves the unit circle $S^1$. Thus one can use function spaces on $S^1$ thus yielding more general Kac-Moody groups than the groups of holomorphic loops, we study. Nevertheless those groups have no complexification in the same regularity class.
\end{remark}

The next aim is to prove that Kac-Moody groups are tame Fr\'echet manifolds. To this end we use a result of B.\ Popescu~\cite{Popescu05} stating that fiber bundles whose fiber is a Banach space over tame Fr\'echet manifolds are tame.
  
We start with the definition of tame fiber bundles:

\begin{definition}[tame Fr\'echet fiber bundle]
\index{tame Fr\'echet fibre bundle}
A fiber bundle $P$ over $M$ with fiber $G$ is a tame Fr\'echet manifold $P$ together with a projection map $\pi: P\longrightarrow M$ satisfying the following condition:

For each point $x\in M$ there is a chart $\varphi: U\longrightarrow V\subset F$ with values in a tame Fr\'echet space $F$ such that there is a chart $\varphi: \pi^{-1}(U)\longrightarrow  G\times U \subset G \times F$ such that the projection $\pi$ corresponds to a projection of $U\times F$ onto $U$ in each fiber.
\end{definition}

\noindent The following lemma is proved in~\cite{Popescu05}.
\begin{lemma}
Let $P$ be a fiber bundle over $M$ whose fiber is a Banach manifold; then $P$ is a tame Fr\'echet manifold.
\end{lemma}

\noindent This result contains the important corollary:

\begin{corollary}
The Kac-Moody groups $\widehat{MG}_{\mathbb{R}}$, $\widehat{MG}_{\mathbb{C}}$, the quotient spaces $\widehat{MG}_{\mathbb{C}}/\widehat{MG}_{\mathbb{R}}$, $\widehat{MG}_D$ and the quotient spaces $\widehat{MG}_{\mathbb{R}}/\textrm{Fix}(\rho)$ and $\widehat{MG}_{D}/\textrm{Fix}(\rho)$ are tame Fr\'echet manifolds.
\end{corollary}

\noindent Next we prove:

\begin{theorem}
\label{hatdiffeovectorspace}
$\widehat{MG}_{\mathbb{C}}/\widehat{MG}_{\mathbb{R}}$ is diffeomorphic to a vector space.
\end{theorem}

\begin{proof}
By theorem~\ref{mgc/mgrtamefrechet}, we know that $MG_{\mathbb{C}}/MG_{\mathbb{R}}$ is diffeomorphic to a vector space.  As $MG_{\mathbb{R}}$ is a subgroup of $MG_{\mathbb{C}}$, the quotient is well defined. To prove the theorem we check the decomposition
\begin{displaymath}\widehat{MG}_{\mathbb{C}}/\widehat{MG}_{\mathbb{R}}\simeq MG_{\mathbb{C}}/MG_{\mathbb{R}}\times (\mathbb{R^+})^2\,.\end{displaymath}

\noindent To this end we use the description of the elements in $\widehat{MG}_{\mathbb{C}}/\widehat{MG}$ as $4$-tuples. Two $4$-tuples $(g(z), p(z,t), r_c, r_d)$ and $(g'(z), p'(z,t), r'_c, r'_d)$ describe the same element of $\widehat{MG}_{\mathbb{C}}/\widehat{MG}$ iff there exists an element $(h(z), q(z,t), s_c, s_d) \in MG_{\mathbb{R}}$ such that 
\begin{displaymath}(g(z), p(z,t), r_c, r_d)=(g'(z)+h(z), p'(z,t)+q(z,t), r'_c+s_c, r'_d+s_d)\,.\end{displaymath}

Hence the equivalence classes for $g(z)$ are elements of $MG_{\mathbb{C}}/MG_{\mathbb{R}}$. The extension of $MG_{\mathbb{R}}$ lies in $S^1$. Thus $r_c$ and $r_d$ are defined up to an element in $S^1$. So we obtain a description of $\widehat{MG}_{\mathbb{C}}/\widehat{MG}_{\mathbb{R}}$ as a $(\mathbb{R}^+)^2$-bundle over $MG_{\mathbb{C}}/MG_{\mathbb{R}}$. 

As $MG_{\mathbb{C}}/MG_{\mathbb{R}}$ is diffeomorphic to a vector space (and hence simply connected), this bundle is trivial. The diffeomorphism can be described by mapping  $(g(z), p(z,t), r_c, r_d)$ onto $(g(z), r_c, r_d)$, hence by forgetting the path $p(z,t)$.
\end{proof}

\noindent Now we investigate the quotients $\widehat{MG}/\widehat{\textrm{Fix}(\sigma)}$.

The group $\widehat{\textrm{Fix}(\sigma)}$ consists of elements $(g, p, r_c, r_d)\in \widehat{MG}$ such that
 $\{g, p\in \textrm{Fix}(\sigma), r_c\in \pm 1, r_d \in \pm 1\}$. 
  As $\{r_c, r_d\} \in \pm 1$, this is a covering with four leaves. The details of the argument follow the description found in~\cite{Popescu05} for the smooth $C^{\infty}$-setting.

\noindent Let $H \subset MG_{\mathbb{C}}$ be a real form of non-compact type.

\noindent The description of the space $H/\textrm{Fix}(\sigma) \subset MG_{\mathbb{C}}/\textrm{Fix}(\sigma)$ follows similarly. This yields a proof  of the following theorem:

\begin{theorem}
\label{typeIIIdiffeovectorspace}
The space $H/\textrm{Fix}(\sigma)$ is diffeomorphic to a vector space.
\end{theorem}

In section~\ref{loopgroups} we introduced the pointwise exponential function $\Mexp:\widehat{M\mathfrak{g}}\longrightarrow \widehat{MG}$.
This exponential function can be extended to an exponential function on the Kac-Moody algebra $\widehat{M\mathfrak{g}}$.
\begin{displaymath}
\widehat{\Mexp}:\widehat{M\mathfrak{g}}\longrightarrow \widehat{MG}
\end{displaymath} 
Via $\widehat{\Mexp}$ the central extension $\mathbb{C}c$ corresponds to the fiber of the $\mathbb{C}^*$-bundle and the
 $\mathbb{C}d$-term corresponds to the $\mathbb{C}^*$-factor of the semidirect product.

To study more precisely the properties of the exponential function, we introduce the notion of curves in a Kac-Moody group. As a Kac-Moody group $\widehat{MG}_{\mathbb{R}}$ is locally a (topologically) direct product of the loop group $MG_{\mathbb{R}}$ with $\mathbb R^2$, a path: $\widehat{\gamma}: (-\epsilon, \epsilon)\rightarrow \widehat{MG}$ is locally described by three components: $\widehat{\gamma}(t)=(\gamma(t), \gamma_c(t), \gamma_d(t))$,
with $\gamma(t)$ taking values in $MG_{\mathbb{R}}$ and $\gamma_c(t)$, $\gamma_d(t)$ taking values in $\mathbb R$. For every $z\in \mathbb C^*$, $\gamma$ defines a path $\gamma_z(t): (-\epsilon, \epsilon)\rightarrow G$ by setting $\gamma_z(t):= [\gamma(t)](z)$.
A path $\gamma: (-\epsilon, \epsilon) \rightarrow MG$, $t \rightarrow \gamma (t)$ is differentiable (respective smooth) iff the map $\delta: (-\epsilon, \epsilon) \times \mathbb C^* \rightarrow G_{\mathbb C}$, $(t, z) \rightarrow \delta(t,z)$ such that $\delta(t,z)=\gamma_z(t)$ is differentiable (respective smooth).

For the group $\widetilde{MG}$ we can describe the exponential function as follows:

\begin{proposition}
The exponential function $\widetilde{\Mexp}:\widetilde{M\mathfrak{g}}\longrightarrow \widetilde{MG}$ is defined by
\begin{displaymath}
\widetilde{\Mexp}(f+r_c c)=\left(\Mexp (f),(t\longrightarrow exp tu)|_0^1 ,r^{ir_c}\right)
\end{displaymath}
\end{proposition}

For the straightforward proof one has to check that this defines a $1$-parameter subgroup whose differential at the identity is $f+r_c c$.

\subsection{The Adjoint action}
\label{theadjointaction}

Similarly to finite dimensional simple Lie groups a Kac-Moody group admits an adjoint action on its Lie algebra:

\begin{example}[Adjoint action]
\label{adjointaction}
\index{Adjoint action}

With $x=\{w,(g,p,z)\}$, the Adjoint action of $\widehat{MG}^{\sigma}$ on $\widehat{M}\mathfrak{g}^{\sigma}$ is described by the following formulae
\begin{alignat*}{1}
\textrm{Ad}(x)u&:= gw(u)g^{-1}+\langle g w(u)g^{-1},g'g^{-1}\rangle c\\
\textrm{Ad}(x)c&:= c\\
\textrm{Ad}(x)d&:= d - g'g^{-1}+ \textstyle\frac{1}{2}\langle g'g^{-1}, g'g^{-1}\rangle c\,.
\end{alignat*} 
Here $\omega(u)$ denotes the shift of the argument by $\omega$.
\end{example}

\noindent For the proof compare~\cite{HPTT}, \cite{PressleySegal86}, and \cite{Kac90}.

\begin{proof}~ 
\begin{itemize}
\item[-] $c$ generates the center, thus  $\textrm{Ad}(x)c:= c$.
\item[-] $\textrm{Ad}(x)u$ follows by integrating the $Ad$-action.
\item[-] To calculate $\textrm{Ad}(g) (d)$, we use the $Ad$-invariance of the Lie bracket. As $[d,v]=v'$ we get for all $v \in L_{alg}\mathfrak{g}^{\sigma}$:
\begin{alignat*}{2}
  &   gv'g^{-1}+\langle g v'g^{-1},g'g^{-1}\rangle c =\\
= &   \textrm{Ad}(g)(v')=\textrm{Ad}(g) [d,v]= [\textrm{Ad}(g)(d), \textrm{Ad}(g) (v)]=\\
= &  [h+ \mu c + \nu d, gvg^{-1}+\langle g vg^{-1},g'g^{-1}\rangle c]=\\
= &  [h, gvg^{-1}]+\nu[d, gvg^{-1}]                           =\\
= &  hgvg^{-1}-gvg^{-1}h+\omega(h, (gvg^{-1})')c+\nu g'vg^{-1}+\nu gv'g^{-1}+\nu gv(g^{-1})'\,,
\end{alignat*}
 with $h \in L_{alg}\mathfrak{g}^{\sigma}$ and $\{\mu, \nu\} \in \mathbb{R}$.

\noindent To get equality we have to choose $\nu=1$, $h=-g'g^{-1}$. This gives us
 \begin{alignat*}{2} 
&  hgvg^{-1}-gvg^{-1}h+\omega(h, (gvg^{-1})')c+\nu g'vg^{-1}+\nu gv'g^{-1}+\nu gv(g^{-1})'=\\
=&-g'vg^{-1}+gvg^{-1}g'g^{-1}+\omega(-g'g^{-1}, (gvg^{-1})')c+ g'vg^{-1}+ gv'g^{-1}+ gv(g^{-1})'=\\
=&\omega(-g'g^{-1}, (gvg^{-1})')c+  gv'g^{-1}\,.
\end{alignat*}

\noindent Thus we are left with the calculation of $\mu$. To this end we use the property that $Ad$ acts by isometries. This results in $\mu=\frac{1}{2}\langle g'g^{-1}, g'g^{-1}\rangle$.\qedhere

\end{itemize}
\end{proof}

\noindent More details can be found in~\cite{PressleySegal86,HPTT,Popescu05, Popescu06}.

\begin{theorem}[Polarity of the Adjoint action]
\label{polarityofadjointaction}
\index{polarity of adjoint action}
Let $H_{l, r}\subset \widehat{X\mathfrak{g}}^{\sigma}$, $\{l, r\} \in \mathbb{R}\backslash \{0\}$ denote the intersection of the sphere with radius $-l^2$ with the horosphere $r_d=r$. The restriction of the Adjoint action to $H_{l, r}$ is polar.
\end{theorem}

\begin{proof} The restriction of the Adjoint action to $H_{l, r}$ coincides with the gauge action on $X\mathfrak{g}$. Hence, theorem \ref{polarityofadjointaction} is a direct consequence of theorem~\ref{polaractiononmg}.
\end{proof}

C.-L.\ Terng describes how to associate an affine Weyl group to this gauge action~\cite{terng95}. This is exactly the affine Weyl group of the Kac-Moody group $\widehat{MG}$.

\noindent This theorem gives a complete description of the Adjoint action iff $r_d\not=0$. 

\noindent Surprisingly in the remaining case $r_d=0$ the situation is different:
now the Adjoint action is reduced to the equations:
 \begin{alignat*}{1}
\textrm{Ad}(x)u&:= gw(u)g^{-1}+\langle g w(u)g^{-1},g'g^{-1}\rangle c\\
\textrm{Ad}(x)c&:= c
\end{alignat*}

Calculate the orbit of the constant function $u\equiv0$. $u$ is fixed by the Adjoint action as $\textrm{Ad}(x)u:= g 0 g^{-1}+\langle g0g^{-1},g'g^{-1}\rangle c=0$. Hence iff we can  describe the restriction of the Adjoint action to $X\mathfrak{g}^{\sigma}$ as some kind of polar action then the associated Weyl group has to be necessarily of spherical type. Furthermore the action is clearly not proper Fredholm. Hence the Hilbert-space version is not covered by Terng's results. 

\noindent We use the regularity independent notation:

We define a  flat of finite type to be a flat $\mathfrak{t}\subset L(\mathfrak{g},\sigma)$ such that $\mathfrak{t}$ is the restriction of a flat in $\widehat{L}(\mathfrak{g},\sigma)$. Hence all flats of finite type are conjugate in $\widehat{L}(G,\sigma)$ and as the orbits of $\widehat{L}(G,\sigma)$ and $L(G,\sigma)$ coincide on $L(\mathfrak{g},\sigma)$ also in $L(G,\sigma)$.  Hence any flat of finite type in $L(\mathfrak{g}, \sigma)$ is isomorphic to $\mathfrak{t}_0\subset \mathfrak{g}$, where we choose $\mathfrak{g}$ to denote the subalgebra of constant loops. Using the usual notion for regular and singular elements we find that the associated Weyl group is the spherical Weyl group of $\mathfrak{g}$.

\begin{remark}
From a geometric point of view this different behaviour is related to the fact that the hyperplane defined by $r_d=0$ corresponds to the spherical building at infinity while the space $r_d\not=0$ corresponds to the spherical building at infinity. For further details confere \cite{freyn09}, \cite{Freyn10a}, \cite{Freyn10d}.
\end{remark}

\chapter{Kac-Moody symmetric spaces}
\label{chap:symm}

\section{Foundations}

Kac-Moody symmetric spaces are tame Fr\'echet Lorentz manifolds. In this foundational section, we will review the differential geometry of tame Fr\'echet manifolds following the presentation in~\cite{Hamilton82} and discuss extensions of some results of pseudo-Riemannian geometry to the tame Fr\'echet setting.

\subsection{Differential geometry of tame Fr\'echet manifolds}

Let $M$ be a tame Fr\'echet manifold whose charts take values in a Fr\'echet space $F$ such that its tangent space at a point $T_fM$ is isomorphic to $F_1$. As the exponential map is not in general a diffeomorphism, we cannot automatically suppose $F=F_1$.

Having defined fibre bundles in section~\ref{Kac-Moody groups}, we
turn now to the special case of vector bundles, which is the most important one for differential geometry.
Denote by $P$ a vector bundle over $M$ with fiber $V$. $M$, $P$ and $V$ are supposed to be tame Fr\'echet spaces and denote by $TP$ its tangential bundle.

A vector field on $M$ is a smooth section of $TM$. 
As an example, let $M:=MG$. Then $TM$ is a $M\mathfrak{g}$-bundle over MG. Charts can be chosen to be $U\times M\mathfrak{g}$, where $U$ is a chart of $MG$. A vector field on $MG$ is defined locally to be $(f,\phi(f))$, with $f\in U$ and $\phi: U\longrightarrow M\mathfrak{g}$ a smooth map. 
Following the finite dimensional theory, we define:

\begin{definition}[Vertical bundle]
\index{vertical bundle}
The vertical bundle $T_vP\subset TP$ consists of the vertical vectors, that is the vectors $v\subset TP$ such that $v \subset (d\pi)^{-1}(0,x)$ with $x\in M$.
\end{definition}

A connection on $TP$ consists of the assignment of a complementary tame Fr\'echet subbundle $T_hP$ of $TP$ (i.e.: such that $TP=T_vP\oplus T_hP$), that is:

\begin{definition}[Connection]
\index{connection}
A connection $\Gamma$ on $TP$ is the assignment of a complementary subspace of horizontal vectors, such that in terms of any coordinate chart of $P$ with values in $(U\subset F)\times G$ the subspace of horizontal vectors at $T_fP$ consists of all $(h,k)\in F_1\times G$ such that $k=\Gamma(f)\{g,h\}$, where $\Gamma$ is represented in any local chart by a smooth map $\Gamma: (U\subset F)\times G\times F_1 \longrightarrow G$, which is bilinear in $g$ and $h$.  
\end{definition} 

If $P$ is the tangent bundle of $M$, then $G=F_1$. We call a connection symmetric if $\Gamma$ is symmetric in $\{g,h\}$. 

In the finite dimensional case, in order to define curvature, one would now define a tensor field on $M$. 
As dual spaces of Fr\'echet spaces are (with the trivial exception of Banach spaces) not Fr\'echet spaces, this approach is not possible for Fr\'echet manifolds.  In contrast we are forced to use explicit coordinate dependent descriptions in terms of component functions.

\begin{definition}[curvature]
\index{curvature of Fr\'echet manifold}
The curvature of a connection on a vector bundle $P$ over $M$ is the trilinear map $R: P\times TM\times TM \longrightarrow P$ such that
$$R(f)\{g,h,k\}:= D\Gamma\{g,h,k\}- D\Gamma(f)\{g,k,h\}-\Gamma(f)\{\Gamma(f)\{g,h\},k\}+\Gamma{f}\{\Gamma(f)\{g,k\},h\}\,.$$
\end{definition}

We now want to define a metric on $TM$. Again, this is only possible in a coordinate description:

\begin{definition}[metric]
\index{metrics}
Let $M$ be a tame Fr\'echet manifold and $TM$ its tangential bundle. A metric on $TM$ is a smooth bilinear map: $g:TM\times TM \longrightarrow \mathbb{R}$. Smooth means that $g$ can be described in any local chart as a smooth map.
\end{definition}

\noindent Clearly, $M$ is not complete with respect to $g$; so $g$ is only a weak metric.

Following the finite dimensional convention, we define the index of $g$ to be the dimension of the maximal subspace on which $g$ is negative definite.

As usual a connection is  \emph{compatible} with the metric $g$, iff 

$$\frac{d}{dt} g\left(V,W\right)=g\left(\frac{DV}{dt},W\right)+g\left(V, \frac{DW}{dt}\right)$$
for any vector fields $V$ and $W$ along a curve $c:I\longrightarrow M$. 

\begin{definition}[Levi-Civita connection]
\index{Levi-Civita connection}
Let $(M,g)$ be a tame Fr\'echet manifold. A Levi-Civita connection is a symmetric connection which is compatible with the metric.
\end{definition}

\begin{lemma}
If a Levi-Civita connection exists, it is well defined and unique.
\end{lemma}

This result follows analogous to in the finite dimensional case via explicit calculation.

In contrast to this result, the existence of a Levi-Civita connection seems not to be clear in general for tame Fr\'echet manifolds.

Nevertheless for the special case of Lie groups, B.\ Popescu proves in~\cite{Popescu05}:

\begin{theorem}[Existence of Levi-Civita connection on Fr\'echet Lie group]
\label{existanceoflevicivitaconnection}
Any tame Fr\'echet Lie group $G$ admits a unique left invariant connection such that $\nabla_{X}Y=\frac{1}{2}[X,Y]$ for any pair of left invariant vector fields $X$ and $Y$. It is torsion free. If $G$ admits a biinvariant (pseudo-) Riemannian metric, then $\nabla$ is the corresponding Levi-Civita connection.
\end{theorem}

\begin{lemma}[curvature of tame Fr\'echet Lie group]
\index{curvature of Fr\'echet Lie group}
\label{curvatureoffrechetliegroup}
The curvature tensor of a tame Fr\'echet Lie group is given by:
\begin{displaymath}
R(f)\{g,h,k\}=\frac{1}{4}[[g,h],k]\, .
\end{displaymath}
\end{lemma}

The proof consisting of purely algebraic manipulations follows the finite dimensional pattern.

\begin{lemma}[Bianchi-identity and symmetry properties]
\index{Bianchi-identity}
For a pseudo-Riemann tame Fr\'echet manifold with Levi-Civita connection, the following identities hold:
\begin{itemize}
\item[-]\hspace{3pt} $ R(f)\{g,h,k\}+R(f)\{h,k,g\}+R(f)\{k,g,h\}=0$ (Bianchi-identity), 
\item[-]\hspace{3pt} $\langle R(f)\{g,h,k\},l \rangle=-\langle R(f)\{h,g,k\},l \rangle$,
\item[-]\hspace{3pt} $\langle R(f)\{g,h,k\},l \rangle=-\langle R(f)\{g,h,l\},k \rangle$,
\item[-]\hspace{3pt} $\langle R(f)\{g,h,k\},l \rangle=\langle R(f)\{l,k,g\},h \rangle$.
\end{itemize}
\end{lemma}  

The proof consists in algebraic manipulations, exactly like in the finite dimensional case. There are no good notions of Ricci curvature or scalar curvature on infinite dimensional manifolds as contraction in general do not exist. Remark that in special cases one could force the existence of Ricci curvature by imposing finiteness conditions on the trace. On the other hand sectional curvature is defined on any $2$-dimensional subspace:

\begin{definition}[Sectional curvature]
\index{section curvature}
If $|g\wedge h|^2=\langle g, g \rangle \langle h, h\rangle -\langle g,h\rangle ^2 \neq 0 $, we define the sectional curvature, to be 
\begin{displaymath}K_f(g,h)=\frac{\langle R(f)\{g,h,g\},h \rangle}{|g\wedge h|^2}\,.\end{displaymath}
\end{definition}

For a Riemannian manifold the condition $|g\wedge h|^2 \not=0$ is always fulfilled. Nevertheless for pseudo Riemannian manifolds there are degenerate planes. For degenerate planes we define:

\begin{definition}[$0$-sectional curvature]
\index{sectional curvature, degenerate planes}
 If $|g\wedge h|^2=\langle g, g \rangle \langle h, h\rangle -\langle g,h\rangle ^2 = 0$ suppose without loss of generality that $g$ is a $0$-vector and that $h$ is not a $0$-vector. Then
\begin{displaymath}
K_{f,N}(g,h)=\frac{\langle R(f)\{g,h,g\}, h\rangle}{\langle h, h\rangle} 
\end{displaymath}
\end{definition}

A consequence of this phenomenon is the theorem of Kulkarni. In the next section we describe a version for pseudo Riemannian tame Fr\'echet manifolds.

\subsection{A Kulkarni type theorem for tame Fr\'echet manifolds}

In the last section, we saw that one can define a connection and a metric which work locally exactly like their finite dimensional counterparts. As a consequence all local purely algebraic results will carry over to the infinite dimensional case.

A nice structure result for finite dimensional Lorentz
geometry is the theorem of Kulkarni. It states that for pseudo-Riemannian
manifolds that are not Riemannian, bounded sectional curvature implies constant curvature.  This theorem generalizes to tame Fr\'echet manifolds:

\begin{theorem}[infinite dimensional Kulkarni-type]
\index{theorem of Kulkarni-type for Fr\'echet manifolds}
\label{Kulkarni}
Let $M$ be a Lorentzian tame Fr\'echet manifold with Levi-Civita connection. Then the
following conditions are equivalent:
\begin{enumerate}
  \item \hspace{3pt}  $K_f(g,h)$ is constant.
  \item \hspace{3pt} $a\leq K_f(g,h)$ or $K_f(g,h) \leq b$ for some $a, b\in \mathbb R$.
  \item \hspace{3pt} $a \leq K_f(g,h) \leq b$ on all definite planes for some
    $a\leq b \in \mathbb R$.
  \item \hspace{3pt} $a \leq K_f(g,h) \leq b$ on all indefinite planes for some
    $a\leq b \in \mathbb R$.
\end{enumerate}
\end{theorem}

This theorem yields the following  corollary:

\begin{corollary}
A Lorentzian tame Fr\'echet manifold admitting a flat subspace isometric to $\mathbb{R}^k$  with $k>1$ is either flat or its curvature is unbounded from above and from below.
\end{corollary}

For the finite dimensional proof of this theorem and more generally finite dimensional Lorentz geometry see~\cite{Oneill83}. A different version is stated in \cite{Beem81}.
The proof described in ~\cite{Oneill83} is based on comparing sectional curvatures on various $2$-dimensional subspaces: Clearly the first statement implies the three other statements. The crucial point is to understand the behaviour of sectional curvature close to degenerate planes. The proof of the tame Fr\'echet version is a straight forward generalization, so we do not detail it.

\subsection{Some remarks about Lorentz symmetric spaces}

\index{Lorentz symmetric spaces}
Even in the finite dimensional situation there is an important difference between Riemannian symmetric spaces and pseudo-Riemannian symmetric spaces:
For finite dimensional Riemannian symmetric spaces there is a complete classification available. The central point of this classification is that each finite dimensional irreducible Riemannian symmetric space with the exception of $\mathbb{R}$ is either a compact simple Lie group or a quotient of a real simple Lie group by a certain subgroup. Hence a classification of irreducible symmetric spaces can be given by classifying complex, simple Lie groups.

 The most important fact for us is that finite dimensional (pseudo-) Riemannian symmetric spaces are not classified; there is recent work by I.\ Kath and M.\ Olbrich~\cite{Kath04,Kath06} which gives a good description of the structure and classifies pseudo-Riemannian symmetric spaces of index $1$ and $2$. The surprising difficulty of the classification of pseudo-Riemannian symmetric spaces in comparison to the Riemannian case has its roots in two facts:

\begin{enumerate}
\item There is no splitting theorem of a pseudo-Riemannian symmetric space into a direct sum of ``simple'' factors.
\item There are pseudo-Riemannian symmetric spaces that have no semisimple groups of isometries.
\end{enumerate}

The subclass of pseudo-Riemannian symmetric spaces corresponding to semisimple Lie groups is well understood. A classification is achieved in the paper~\cite{Berger57}.

\subsection{Tame Fr\'echet symmetric spaces}

We use the definition:

\begin{definition}[tame Fr\'echet symmetric space]
\index{tame Fr\'echet symmetric space}
\label{tamesymmetricspace}
A tame Fr\'echet manifold $M$ with a weak metric having a Levi-Civita connection is called a symmetric space, iff for all $p\in M$ there is an involutive isometry $\rho_p$, such that $p$ is an isolated fixed point of $\rho_p$.
\end{definition}

It is not clear if the isometry group of a tame Fr\'echet manifold is in general a tame Fr\'echet Lie group.

\begin{lemma}[geodesic symmetry]
\index{geodesic symmetry}
\label{geodesicsymmetry}
Let $M$ be a tame Fr\'echet pseudo-Riemannian symmetric space.
For each $p\in M$ there exists a normal neighborhood $N_p$ of $p$ such that $s_p$ coincides with the geodesic symmetry on all geodesics through $p$ in $N_p$.
\end{lemma}

As the exponential map for Fr\'echet manifolds is in general not a diffeomorphism, it is not true that  for each $p$ there is an open neighborhood $N_p$, such that there is for all $q \in N_p$ a geodesic segment connecting $p$ and $q$. Thus the notion of geodesic symmetries at a point $p$ is only defined for points in $N_p\cap \exp_p{T_pM}$.

\begin{proof}[Proof of lemma~\ref{geodesicsymmetry}:]
Let $\gamma(t)$ be a geodesic through $p$ and $\mu(t):=\rho_p(\gamma(t))$ its image under the isometry $\rho_p$.
\begin{itemize}
\item[-] As $\rho_p$ is an isometry we find that $\mu(t)$ is a geodesic.
\item[-] $(d\rho)(\dot{\gamma}(t)|_{t=0})=\dot{\mu}(t)|_{t=0}$. If $\dot{\gamma}(t)|_{t=0}=-\dot{\mu}(t)|_{t=0}$ for all geodesics $\gamma(t)$, we are done. So suppose by contradiction, there exists a $\gamma(t)$ such that $\dot{\gamma}(t)|_{t=0}\not=-\dot{\mu}(t)|_{t=0}$.
As $\rho_p$ is an isometry, we have $|\dot{\gamma}(t)|_{t=0}=|\dot{\mu}(t)|_{t=0}$. Hence, there are two cases:
    \begin{itemize}
			\item[-] There is a geodesic such that $\dot{\gamma(0)}=\dot{\mu(0)}$. Then $\gamma(t)=\mu(t)$ so $p$ is no isolated fixed point.
			\item[-] $\textrm{dim } \textrm{span}\{\dot{\gamma}(0),\dot{\mu}(0)\}=2$. Then take the geodesic $\nu(t):=\textrm{exp}(\dot{\gamma}(0)+\dot{\mu}(0))$. $\rho_p$ is the identity on $\nu(t)$, so again $p$ is no isolated fixed point. This is again a contradiction, so the lemma is proven.\qedhere
		\end{itemize}
\end{itemize} 
\end{proof}

\begin{corollary}
For each $p$, the involution $\rho_p$ induces $-Id$ on the tangent space $T_pM$.
\end{corollary}

The corollary is a straight forward application of lemma~\ref{geodesicsymmetry}.

\begin{lemma}
\index{locally symmetric space}
A tame Fr\'echet symmetric space with a Levi-Civita connection is locally symmetric, that is it has parallel curvature. 
\end{lemma}

The proof of the finite dimensional version of this lemma is based on local algebraic manipulations. A proof can be found in~\cite{Helgason01}. It generalizes directly to the tame Fr\'echet setting.

\begin{definition}[Kac-Moody symmetric space]
\index{Kac-Moody symmetric space}
An (affine) Kac-Moody symmetric space $M$ is a tame Fr\'echet Lorentz symmetric space such that its isometry group $I(M)$ contains a transitive subgroup isomorphic to an affine geometric Kac-Moody group $H$, and the intersection of the isotropy group of a point with $H$ is a loop group of compact type.
\end{definition}

\begin{lemma}
Let $M$ be a Kac-Moody symmetric space with isometry group $I(M)$; let $H$ be the Kac-Moody subgroup of $I(M)$. For $x\in M$, let $K_x=\textrm{Fix}(x) \cap (\widehat{H})_0$. 
\begin{enumerate}
\item For $x, y\in M$ we have $K_x \simeq K_y$.
\item $M$ is isomorphic to the quotient $M\simeq(\widehat{H})_0/K_x$.\qedhere
\end{enumerate}
\end{lemma}

\begin{proof}~
\begin{enumerate}
\item Transitivity of $H$ assures the existence of an element $h_{xy}\in H$ such that $h_{xy}(x)=y$. Then $K_y=hK_xh^{-1}$.
\item The second assertion follows as $H$ acts transitively on $M$.\qedhere
\end{enumerate}
\end{proof}

Recall the alternative definition of symmetric spaces proposed by Loos~\cite{Loos69I}, p.~69:

\begin{definition}[symmetric space in the sense of Loos]
\index{symmetric space in the sense of Loos}
A symmetric space is a manifold $M$ with a differentiable multiplication $\mu:M\times M\longrightarrow M$, $\mu(x,y)=x\cdot y$ such that
\begin{enumerate}
\item $x\cdot x=x$,
\item $x\cdot (x\cdot y)=y$,
\item $x\cdot (y\cdot z)=(x\cdot y) \cdot (y\cdot z)$,
\item every $x$ has a neighborhood $U$ such that $x\cdot y=y$ implies $y=x$ for all $x\in U$.
\end{enumerate}
\end{definition}

Following~\cite{Loos69I} we define the multiplication to be the involutive isomorphisms:
\begin{displaymath}
 x\cdot y:=\sigma_x(y)
\end{displaymath}

\begin{theorem}
A Kac-Moody symmetric space is a symmetric space in the sense of Loos.
\end{theorem}

\begin{proof}
We have to verify that the geodesic symmetry satisfies the four axioms. The first three axioms are immediate:
\begin{enumerate}
\item $x\cdot x=\sigma_x(x)=x$,
\item $x\cdot (x\cdot y) =\sigma_x(\sigma_x(y))=\sigma_x^2(y)=y$,
\item $x\cdot (y\cdot z)=\sigma_x(\sigma_y(z))= \sigma_{\sigma_x(y)}(\sigma_x(z))=(x\cdot y) \cdot (x\cdot z)$.
\end{enumerate}
The last axiom follows as we supposed $x$ to be an isolated fixed point for $\sigma_x$ (see definition \ref{tamesymmetricspace}). 
\end{proof}

Thus we describe now how to construct the OSAKA associated to $M$. Let $\widehat{H}$ be the affine Kac-Moody group acting transitively on $M$ and $\widehat{\mathfrak{h}}$ its tangential geometric affine Kac-Moody algebra. Let $\rho_*$ be the involution induced on $\widehat{\mathfrak{h}}$.

\begin{theorem}
The pair $(\widehat{\mathfrak{h}}, \rho_*)$ is an OSAKA.
\end{theorem}

\begin{proof}~
\begin{enumerate}
\item $\mathfrak{h}$ is a geometric affine Kac-Moody algebra.
\item $\rho_*$ is an involution of the second kind; hence we can suppose $\rho_*(c)=-c$ and $\rho_*(d)=-d$. The fixed point algebra $\textrm{Fix}(\rho_*)\subset \widehat{\mathfrak{h}}$ is a loop algebra. Let $\mathfrak{h}=\bigoplus_{i} \mathfrak{h}_i\oplus \mathbb{R}c\oplus \mathbb{R}d$ be the decomposition of $\mathfrak{h}$ into ideals invariant under $\rho_*$. 

\begin{enumerate}
 \item If $\mathfrak{h}_i$ is abelian then the Cartan-Killing form of the subjacent Lie algebra vanishes --- hence also the averaged Cartan-Killing form. Hence, the  maximal subgroup of compact type is trivial. 
 \item If $\mathfrak{h}_i$ is of compact type then the fixed point algebra of any involution of the second kind is of  compact type as it is a subalgebra of  compact type~\cite{Heintze09}.
 \item If $\mathfrak{h}_i$ is of non-compact type then there is the Cartan involution whose fixed algebra is the maximal compact subalgebra $\mathcal{K}$. Hence, if $\rho_*$ is the Cartan involution we are done. Suppose $\rho_*$ it is not the Cartan involution. Then there are two possibilities:
 suppose first there is an element $x\in \mathcal{K}\backslash \textrm{Fix}(\rho_*) $. Then the resulting space is not Lorentzian as the Cartan Killing form is no longer positive definite on the loop part of the complement of $\textrm{Fix}(\rho)$. Thus $\mathcal{K}\subset \textrm{Fix}(\rho_*)$. But then $\textrm{Fix}(\rho_*)$ is not of compact type. \qedhere
 \end{enumerate}
\end{enumerate} 
\end{proof}

\begin{definition}[irreducible Kac-Moody symmetric space]
\index{irreducible Kac-Moody symmetric space}
A Kac-Moody symmetric space is called irreducible iff its OSAKA is irreducible.
\end{definition}

In the next sections we investigate the three types of irreducible Kac-Moody symmetric spaces.

\section{ILB-symmetric spaces}

In this short section we extend the connection between tame Fr\'echet structures and ILB-structures to include symmetric spaces. We start with some definitions: 

\begin{definition}
Let $\{F_k,F\}$ be an ILB-system. A ILB-bilinear form $b$ on $\{F_k,F\}$ consists of a set of bilinear forms: $b_k:F_k\times F_k \longrightarrow \mathbb{R}$ and $b:F\times F\longrightarrow \mathbb{R}$ such that $b_k|_{F_{k+1}\times F_{k+1}}\equiv b_{k+1}$ and $b_k|_{F\times F}\equiv b$.
\end{definition}

\begin{definition}
A pseudo-Riemannian ILB-manifold $M$ is an ILB-manifold together with a nondegenerate ILB-bilinear form on each tangent space $T_pM$ varying smooth with the basepoint.
\end{definition}

\begin{definition}[ILB-symmetric space]
A pseudo-Riemannian ILB-manifold $M$ modelled on an ILB-system $\{F_k,F\}$ is called an ILB-symmetric space  iff for all $p\in M$ there is an involutive isometry $\rho_p$ such that $p$ is an isolated fixed point of $\rho_p$.
\end{definition}

\begin{proposition}
Kac-Moody symmetric spaces are ILB-symmetric spaces.
\end{proposition}

This is a consequence of proposition \ref{ILB-manifold}.

\section{Kac-Moody symmetric spaces}

From now on let $M$ denote a Kac-Moody symmetric space. A transitive subgroup $H$ of its isometry group $I(M)$ --- called the Kac-Moody isometry group --- is a real form of a Kac-Moody group $\widehat{MG}$.  Following the finite dimensional theory we define the scalar product on a Kac-Moody symmetric space via the scalar product on its Kac-Moody isometry group. 

Suppose $H$ is a real form of a Kac-Moody group $\widehat{MG}$. We first need an $Ad$-invariant scalar product on $\widehat{MG}_{\mathbb{C}}$. This induces a scalar product on each real form. 

\index{Lorentz metric of Kac-Moody symmetric space}
The scalar product on $\widehat{XG}$, $X\in \{A_n, \mathbb{C}^*\}$ can be defined in the following way:

We start with a scalar product $\langle\cdot,\cdot \rangle$ on the compact real form  $\mathfrak{g}$ defined as follows~\cite{Heintze08}, \cite{Popescu05}, \cite{HPTT}, \cite{Gross00}:
\begin{eqnarray*}
\langle u, v \rangle &= &\frac{1}{2\pi} \int\langle u(t), v(t) \rangle dt,\\
\langle c, d \rangle &= &-1,\\
\langle c, c \rangle &= &\langle d, d \rangle \hspace{5pt} =\hspace{5pt} \langle f, c \rangle \hspace{5pt} = \hspace{5pt} \langle f, d \rangle \hspace{5pt}= \hspace{5pt}0\, .
\end{eqnarray*}

By direct computation one checks the following result:
\begin{lemma}
The scalar product $\langle\cdot ,\cdot\rangle$ is $Ad$-invariant.
\end{lemma}

For a proof see~\cite{Gross00}, \cite{HPTT}.

As in the finite dimensional case we call a vector $v$
\begin{itemize}
	\item[-] space-like iff $|v| > 0$,
	\item[-] light-like iff $|v| = 0$,
	\item[-] time-like  iff $|v| < 0$.
\end{itemize}

\noindent The Lie algebra $M\mathfrak{g}$ lies completely in the space-like part, while for example the direction $c+d$ is a time-like vector, and $c$ and $d$ are light-like vectors.

If $\mathfrak{g}$ is simple this scalar product is essentially unique: While the definition $\langle d, d\rangle:=0$ is arbitrary and not forced by the Adjoint action, $d$ is not a distinguished element; it is only defined up to an element $f\in \widetilde{L\mathfrak{g}}$. So suppose we choose the definition $\langle \widetilde{d}, \widetilde{d}\rangle=-k^2<0$ for $k\in \mathbb{R}$. Then we can find an element $f\in M\mathfrak{g}$ such that $\langle f, f \rangle= k^2$ and define $\overline{d}:= \widetilde{d}+ f$. Then we get $\langle \overline{d}, \overline{d} \rangle =0$. If in contrast $\langle \widetilde{d}, \widetilde{d}\rangle=k^2>0$, choose $\overline{d}=\widetilde{d}+ \frac{k^2}{2}c$; again we find $\langle \overline{d}, \overline{d} \rangle= 0 $.

Now we extend $\langle\cdot,\cdot\rangle$ to the complexification $\mathfrak{g}_{\mathbb{C}}$ of $\mathfrak{g}$:
We identify $\widehat{M\mathfrak{g}}_{\mathbb{C}}=\widehat{M\mathfrak{g}}_{\mathbb{R}}\oplus i \widehat{M\mathfrak{g}}_{\mathbb{R}}$ and put $\langle ix, iy \rangle:= \langle x,y\rangle$ and $\langle x, iy\rangle=0$. We have to check that this scalar product is invariant under $\widehat{MG}$. The crucial point here is to check $Ad$-invariance under the action of $\exp{id}$. We have $Ad(\exp{ird})(v(t)):= v(te^{-r})$, $v(t)\in M\mathfrak{g}$, which changes the radius. $Ad$-invariance follows like that:

\noindent  We use the equivalence $\langle u,v \rangle=2\pi  \textrm{Res}(uv')$. Thus 
 \[
 \begin{array}{rl}
 
 Ad(e^{ird}) \langle v,u \rangle &= \langle v(te^{-r}), u(te^{-r}) =\rangle\\
 & = 2\pi Res (v(te^{-r})u(te^{-r})=\\
 & = 2\pi Res (\sum_n (\sum_k v_k u_{n-k}r(n-k)) r^{n}t^{n})\\
 & = \sum_k v_k u_{-k-1}  =\\
 & = 2\pi Res(uv)= \langle u, v \rangle\,.
 \end{array}
 \]

By restriction the scalar product $\langle\cdot ,\cdot \rangle$ induces a scalar product on any real form of $\mathfrak{g}_{\mathbb{C}}$. The restriction to the compact real form  $\mathfrak{g}$ of $\mathfrak{g}_{\mathbb{C}}$ is the Lorentzian scalar product, we started with. 

\begin{definition}[Rank]
\index{rank of Kac-Moody symmetric space}
The rank of a Kac-Moody symmetric space $M$ is the maximal $n\in \mathbb{N}$, such that there exists a flat totally geodesic submanifold $A\subset M$ with signature $(n-1,1)$.
\end{definition}

\begin{lemma}
The rank $r$ of a Kac-Moody symmetric space satisfies $r\geq 3$.
\end{lemma}

\begin{proof}
Each flat with signature $(n-1,1)$ has to contain the direction spanned by $c$, by $d+ f$ for some $f\in M\mathfrak{g}$ and an element in $M\mathfrak{g}$. 
\end{proof}

\noindent Now theorem~\ref{existanceoflevicivitaconnection} guarantees the existence of an associated Levi-Civita Connection $\nabla$.

\section{Kac-Moody symmetric spaces of the compact type}

\subsection{Foundations}

\index{Kac-Moody symmetric space of the compact type}
In this part we want to equip Kac-Moody groups of the compact type with an $Ad$-invariant metric, in order to obtain symmetric spaces.  Of course these spaces are --- as they are infinite dimensional
--- not compact. So we want to call a Kac-Moody symmetric space ``of the compact
type''' if it is associated to a real form of compact type of a Kac-Moody group or to a
quotient of such a real form  (or, as we will see equivalently, if $\langle R(f)\{g,h,g\}, h\} \rangle \geq 0$).

\noindent A consequence of theorems~\ref{existanceoflevicivitaconnection} and \ref{curvatureoffrechetliegroup} is:

\begin{theorem}[Levi-Civita connection and curvature of Kac-Moody groups]
\label{levicivita}
$(\widehat{MG}, g)$ admits a unique Levi-Civita connection $\nabla$. For
$g, h$ left invariant vector fields, $\nabla_{g}h=\frac{1}{2}[g,h]$. The curvature tensor is given by
\begin{displaymath}
R(f)\{g,h,k\}=\frac{1}{4}[[g,h],k]\, .
\end{displaymath}
\end{theorem}

\begin{corollary}
\label{curvatureestimatecompacttype}
The metric of a Kac-Moody group of compact type satisfies the following curvature estimate: 
\begin{displaymath}\langle R(f)\{g,h,g\}, h\}\rangle \geq 0\, .\end{displaymath}
\end{corollary} 

\begin{proof}
The calculation is similar to the one used in the finite dimensional case for compact Lie groups:
\begin{displaymath}
\langle R(f)\{g,h,g\}, h\}\rangle =\langle[g,[h,g]],h\rangle=\langle[h,g],[h,g]\rangle\geq 0\, .
\end{displaymath}
The last inequality follows as $[h,g]$ is in the derived algebra.
\end{proof}

 Remark that in the finite dimensional case the condition $\langle R(f)\{g,h,g\}, h\}\rangle \geq 0$ is equivalent to nonnegative sectional curvature. This is no longer true in the Kac-Moody case as the metric is now Lorentzian and hence the denominator may get infinite~(see also theorem \ref{Kulkarni}).

\noindent Nevertheless, using the Jacobi equation, 
\begin{displaymath}
\frac{D^2 J}{dt^2}+R(f)\{\gamma'(t),J(t),\gamma'(t)\}=0, \quad
\end{displaymath}
for a geodesic $\gamma(t)$ and a Jacobi field $J(t)$, we see that the resulting symmetric spaces have ``compact type behaviour''.

\subsection{Kac-Moody symmetric spaces of type II}
\index{Kac-Moody symmetric space of type II}

In this section we describe the Kac-Moody analogue of finite dimensional irreducible ``type II'' Riemannian symmetric spaces. Riemannian symmetric spaces of type $II$ are compact Lie groups equipped with their (unique up to a global scaling factor) invariant metric. Accordingly Kac-Moody symmetric spaces of type $II$ are Kac-Moody groups of compact type equipped with their invariant metric.

Let $G_{\mathbb{C}}$ be a simple complex Lie group and $G$ its compact real form (unique up to conjugation). Let $\widehat{MG}$ be the Kac-Moody group of compact type associated to $G$. 

Let us fix some notation: For an element $f:\mathbb{C}\longrightarrow G_{\mathbb{C}}\in MG_{\mathbb{R}}$ we denote by $f^*$ the element: \begin{displaymath}f^*(z):=\overline{f\left(\frac{1}{\overline{z}}\right)}\, .\end{displaymath}

We choose now the group $\widehat{M(G\times G)}$ as the Kac-Moody symmetric group and define the symmetry $\widehat{M\rho}$ by the mapping
\begin{alignat*}{1}
\widehat{M\rho}: \widehat{M(G\times G)}&\longrightarrow \widehat{M(G\times G)},\\
\left\{\left(f_1(z),p_1(z,t)\right),\left(f_2(z),p_2(z,t)\right),r_c, r_d\right\}&\mapsto 
\left\{\left(f^*_2(z), p_2^*(z,t)\right), \left(f^*_1(z), p^*_1(z,t)\right),-r_c, -r_d\right\}\,.
\end{alignat*}

The action of the Kac-Moody isometry group $\widehat{M(G\times G)}$ on $\widehat{MG}$ is defined as follows:

\begin{lemma}
$\widehat{M\rho}$ is an involution of the second kind.
\end{lemma}
\begin{proof}
As $f^{**}(z)=f(z)$, $(\widehat{M\rho})^2(r_c)=\widehat{M\rho}(-r_c)=r_c$, $(\widehat{M\rho})^2(r_d)=\widehat{M\rho}(-r_d)=r_d$, we find that $\widehat{M\rho}$ is an involution; as $\widehat{M\rho}(r_c)=-r_c$ it is of the second kind. 
\end{proof}

\begin{lemma}
The pair $\left(\widehat{M(\mathfrak{g\times g})}, d\widehat{M\rho}\right)$ is an irreducible OSAKA of type $II$. 
\end{lemma}

\begin{proof}
The fixed points of $\widehat{M\rho}$ consist of elements that have a description of the form: 
\begin{displaymath}\left(\left(f_1(z),p_1(z,t)\right),\left(f_2(z),p_2(z,t)\right),r_c, r_d\right)\,,\end{displaymath}
such that  $f_1(z)=f^*_2(z), p_1(z,t)=p_2^*(z,t)$ and $r_c=r_d\equiv 0$. Hence the fixed point group is isomorphic to the loop group $MG$, which is a loop group of compact type. 
The tangential space of $\widehat{M(G\times G)}$ at the identity is $\widehat{M(\mathfrak{g}\times\mathfrak{g})}$. The involution 
 \begin{displaymath}
 \widehat{M\rho}:\widehat{M(G\times G)}\longrightarrow \widehat{M(G\times G)}
 \end{displaymath} 
 induces an involution 
\begin{displaymath}
\widehat{M\rho}:\widehat{M(\mathfrak{g}\times\mathfrak{g})}\longrightarrow \widehat{M(\mathfrak{g}\times\mathfrak{g})}\, .
\end{displaymath}
Its fixed point algebra is the Lie algebra of the fixed point group, hence $M\mathfrak{g}$; this is a loop algebra of the compact type, whence the result.
\end{proof}

\begin{theorem}
We have the isomorphism:
\begin{displaymath}
\widehat{MG}\cong \widehat{M(G\times G)}/MG\, .
\end{displaymath}
\end{theorem}

\begin{proof}
Let $\Delta^{*}:=(f(z),p(z,t),f^*(z),p^*(z,t))\subset M(G\times G)$. We can describe the quotient space $\widehat{M(G\times G)}/MG$ as the set of equivalence classes 
\begin{displaymath}\left(\left(f_1(z),p_1(z,t)\right),\left(f_2(z),p_2(z,t)\right),r_c, r_d\right)\cdot \Delta^*\,,\end{displaymath} 
Each equivalence class contains exactly one representative of the form $\left(\left(f_3(z),p_3(z,t)\right),\left(0,0\right),r_c, r_d\right)$ by choosing 
the element $(f_2(z), f_2(z)^{-1})\subset \Delta$.
\end{proof}

The Lorentzian structure on $\widehat{M(G)}$ is $\widehat{M(G\times G)}$ invariant, iff the scalar product on $\widehat{M\mathfrak{g}}$ is biinvariant. Thus by example~\ref{adjointaction} the resulting Lorentzian structure on $\widehat{MG}$ is biinvariant.

The classification of Kac-Moody symmetric spaces of type $II$ coincides thus with the classification of complex affine Kac-Moody algebras. We get the list:
\begin{displaymath}
\widetilde{A}_n,\widetilde{B}_n, \widetilde{C}_n, \widetilde{D}_n,\widetilde{E}_{6,7,8},
\widetilde{F}_4,\widetilde{G}_2, \widetilde{A}_1^{'},\widetilde{C}_1^{'},\widetilde{B}_n^t, \widetilde{A}_n^t,\widetilde{F}_4^t,\widetilde{G}_2^t\, .
\end{displaymath}

\begin{lemma}[rank of Kac-Moody symmetric spaces of type II]
Let $\mathfrak{g}$ be a simple Lie algebra of rank $r$. The Kac-Moody symmetric space  $\widehat{MG}$ has rank $r+2$.
\end{lemma}

\begin{proof}
The standard maximal flat is a direct product of the standard torus in $G$ with $(S^1)^2$, where one $S^1$-factor corresponds to $c$, the other to $d$ --- compare~\ref{sect:structure of finite dim flats}.
\end{proof}

\subsection{Kac-Moody symmetric spaces of type I}
\index{Kac-Moody symmetric space of type II}

The aim of this section is to prove the following theorem:

\begin{theorem}[Kac-Moody symmetric space of type I]
\label{Kac-Moody symmetric space of type I}
The space $\widehat{MG}^{\sigma}/MG^{\sigma, \rho}$ carries an $\widehat{MG}$-invariant Lorentzian metric such that it is a symmetric space.
\end{theorem}

\begin{lemma}
$\widehat{MG}/MG^{\rho}$ is associated to an OSAKA of type $I$.
\end{lemma}

\begin{proof}
$\widehat{MG}^{\sigma}$ is a Kac-Moody group, $Fix(\rho)=MG^{\sigma, \rho}$ is a loop group of the compact type --- whence the result.
\end{proof}

We denote by $\pi$ the projection
\begin{displaymath} \pi: \widehat{MG}\longrightarrow \widehat{MG}/\widehat{MG}^{\rho}, f \mapsto f\cdot \widehat{MG}^{\rho}\,.\end{displaymath}

For notational convenience we suppress the involution $\sigma$. For the proof of theorem~\ref{Kac-Moody symmetric space of type I} we need a lemma:

\begin{lemma}
\label{identificationtangentspace}
Let $\widehat{M\mathfrak{g}}^{\rho}_{-1}$ denote the eigenspace of $\rho$ to the eigenvalue $-1$. There is an equivalence
\begin{displaymath}T_{\pi(e)}\widehat{MG}/\widehat{MG}^{\rho}\simeq \widehat{M\mathfrak{g}}^{\rho}_{-1}\,.\end{displaymath}
\end{lemma}

\begin{proof}~
\begin{itemize}
\item[-] The inclusion $T_{\pi(e)}\widehat{MG}/\widehat{MG}^{\rho}\subset \widehat{M\mathfrak{g}}^{\rho}_{-1}$ follows as for every curve $\gamma(t)$ in $\widehat{MG}/\widehat{MG}^{\rho}$ there is a curve $\gamma'(t)$ in $\widehat{MG}$, such that $\pi(\gamma'(t)=\gamma(t)$. So let $\gamma'(0)=e$ and $\gamma(0)=\pi(e)$; then $d\pi(\dot{\gamma}'(0))=\gamma(0)$. 
\item[-] The inclusion $T_{\pi(e)}\widehat{MG}/\widehat{MG}^{\rho}\supset \widehat{M\mathfrak{g}}^{\rho}_{-1}$ is clear: let $X\in \widehat{M\mathfrak{g}}^{\rho}_{-1}$, $\gamma(t):=\pi\circ exp(tX)$ is a curve in $\widehat{MG}/\widehat{MG}^{\rho}$; thus $\widehat{MG}/\dot{\gamma(t)}:= d\pi()T_{\pi(e)}\widehat{MG}^{\rho}$.\qedhere
\end{itemize}
\end{proof}

\begin{proof}[Proof of theorem~\ref{Kac-Moody symmetric space of type I}:]

To define a metric on $\widehat{MG}/\widehat{MG}^{\rho}$, we use the projection
$$ \pi: \widehat{MG}\longrightarrow \widehat{MG}/\widehat{MG}^{\rho}\,.$$

Lemma~\ref{identificationtangentspace} shows that $\pi$ induces the identification of $T_{\pi(e)}\widehat{MG}/\widehat{MG}^{\rho}\simeq \widehat{M\mathfrak{g}}^{\rho}_{-1}$.
We can thus define a metric on $T_{\pi(e)}\widehat{MG}/\widehat{MG}$ by 
$$\langle x,y\rangle_{T_{\pi(e)}\widehat{MG}/\widehat{MG}^{\rho}} := \langle \pi^{-1}(x), \pi^{-1}(y)\rangle_{\widehat{M\mathfrak{g}}^{\rho}_{-1}}.$$
This scalar product is $Ad(\widehat{MG}^{\rho})$-invariant, thus well defined. By left translation it induces an $\widehat{MG}$-invariant metric on $\widehat{MG}/\widehat{MG}^{\rho}$.

Via the projection map $\pi$, the connection on $\widehat{MG}$ induces a connection on the quotient $\widehat{MG}/\widehat{MG}^{\rho}$, making the quotient into a symmetric space. 
\end{proof}

\begin{lemma}
The curvature tensor of a Kac-Moody symmetric space of type $I$ satisfies the curvature estimate
\begin{displaymath}\langle R(f)\{g,h,g\}, h\}\rangle \geq 0\, .\end{displaymath}
\end{lemma}

\begin{proof}
With respect to this metric the projection map $\pi$ is a Riemannian submersion. Hence the estimates of the curvature tensor of the group $\widehat{MG}$ (lemma \ref{curvatureestimatecompacttype}) are preserved.

\end{proof}

\section{Kac-Moody symmetric spaces of the non-compact type}
\index{Kac-Moody symmetric space of the non-compact type}
\index{Kac-Moody symmetric space of type III}\index{Kac-Moody symmetric space of type I}

Dual symmetric spaces can be constructed in the canonical way,
known from finite dimensional symmetric spaces.

\noindent The complexification of the $Ad$-invariant Lorentzian 
metric on $\widehat{M\mathfrak{g}}_{\mathbb{R}}$ is an $Ad$-invariant metric on $\widehat{M\mathfrak{g}}_{\mathbb{C}}$ and induces thus a left invariant metric on  $\widehat{MG}_{\mathbb
C}$. 

\begin{theorem}[Kac-Moody symmetric spaces of type IV]
\index{Kac-Moody symmetric space of type IV}

\label{KMspaceoftypeIV}
Let $G_{\mathbb{C}}$ be a simple complex Lie group and $\widehat{MG}_{\mathbb{C}}^{\sigma}$ be an associated complex (twisted) Kac-Moody group. Let $\widehat{MG}_{\mathbb{R}}^{\sigma}$ be the maximal compact subgroup.
The quotient space
\begin{displaymath} M=\widehat{MG}_{\mathbb{C}}^{\sigma}/\widehat{MG}_{\mathbb{R}}^{\sigma}\end{displaymath} 
is a Kac-Moody symmetric space of type $IV$.
\end{theorem}

There is one principal point to study, namely the OSAKA associated to a Kac-Moody symmetric space of type $IV$:

\begin{lemma}
\label{OSAKAoftypeIV}
The OSAKA associated to a Kac-Moody symmetric space of type $IV$ is an OSAKA of type $IV$.
\end{lemma}

The problem is as follows:
Studying the pair $\left(\widehat{MG}_{\mathbb{C}}^{\sigma}, \widehat{MG}_{\mathbb{R}}^{\sigma}\right)$ we see that the associated pair of Lie algebras consists of $\left(\widehat{M\mathfrak{g}_{\mathbb{C}}}^{\sigma}, \widehat{M\mathfrak{g}_{\mathbb{R}}}^{\sigma}\right)$. As $\widehat{M\mathfrak{g}_{\mathbb{C}}}^{\sigma}$is not a real geometric Kac-Moody algebra and $\widehat{M\mathfrak{g}_{\mathbb{R}}}^{\sigma}$ is not a loop algebra of the compact type but a Kac-Moody algebra, it is not an OSAKA in the sense of our definition. Nevertheless it is an OSAKA up to complexification of the $c$- and $d$-extension.

\begin{proof}[Proof of Lemma \ref{OSAKAoftypeIV}]
For simplicity of notation we suppress the superscripts $\sigma$. Call two OSAKAs $(\mathfrak{g}, \mathfrak{h})$ and $(\mathfrak{g}', \mathfrak{h}')$ equivalent if
\begin{displaymath}
\mathfrak{g}/\mathfrak{h}\cong \mathfrak{g}'/\mathfrak{h}'\, .
\end{displaymath}

Let $\oplus_s$ denote the semidirect product described by the usual formulas for affine Kac-Moody algebras. We have the equivalences
\begin{align*}
\left(\widehat{M\mathfrak{g}_{\mathbb{C}}}, \widehat{M\mathfrak{g}_{\mathbb{R}}}\right)&= 
\left(\widehat{M(\mathfrak{g}_{\mathbb{R}}\oplus i\mathfrak{g}_{\mathbb{R}})}, \widehat{M\mathfrak{g}_{\mathbb{R}}}\right)=\\
&=\left(M(\mathfrak{g}_{\mathbb{R}}\oplus i\mathfrak{g}_{\mathbb{R}})\oplus_s(\mathbb{R}+i\mathbb{R})c\oplus_s(\mathbb{R}\oplus i\mathbb{R})d, M\mathfrak{g}_{\mathbb{R}}\oplus_s i\mathbb{R}c\oplus_s\mathbb{R}d\right)=\\
&=\left(M(\mathfrak{g}_{\mathbb{R}}\oplus i\mathfrak{g}_{\mathbb{R}})\oplus_s \mathbb{R}c\oplus_s\mathbb{R}d, M\mathfrak{g}_{\mathbb{R}}\right)\, .
\end{align*}
The last pair is an OSAKA.
\end{proof}

\begin{proof}[Proof of theorem \ref{KMspaceoftypeIV}]
The proof follows the pattern of theorem~\ref{Kac-Moody symmetric space of type I}.
\end{proof}

\begin{theorem}[Kac-Moody symmetric spaces of type III]
\index{Kac-Moody symmetric space of type III}
Let $H$ be a real form of non-compact type of $\widehat{MG}_{\mathbb{C}}$. The $Ad$-invariant scalar product on $\widehat{MG}_{\mathbb{C}}$ is restricted to an $Ad$-invariant scalar product on $H$. Let $\textrm{Fix}(\rho)$ be the maximal compact subgroup. The projection 
\begin{displaymath}\pi: H \longrightarrow H/\textrm{Fix}(\rho)\end{displaymath}
induces a unique left invariant Lorentzian metric on $H/\textrm{Fix}(\rho)$. With respect to this metric, the space $M=H/\textrm{Fix}(\rho)$ is a symmetric space. 
\end{theorem}

\begin{lemma}
\label{curvatureestimatenon-compacttype}
The curvature tensor of a Kac-Moody symmetric space of the non-compact type satisfies:
\begin{displaymath}\langle R(f)\{g,h,g\}, h\}\rangle\geq 0\end{displaymath}
\end{lemma}

\begin{proof}
Use the expression~\ref{levicivita}:
\begin{displaymath}
R(f)\{g,h,k\}=\frac{1}{4}[[g,h],k]\, .
\end{displaymath}
The Dualization switches the sign.

The result follows now from corollary \ref{curvatureestimatecompacttype}, stating that
$\langle R(f)\{g,h,g\}, h\}\rangle \geq 0$ for Kac-Moody symmetric spaces of the compact type.

\end{proof}

\noindent Theorems~\ref{hatdiffeovectorspace} and \ref{typeIIIdiffeovectorspace} include the following important corollary.

\begin{corollary}
Kac-Moody symmetric spaces of the non-compact type are diffeomorphic to a vector
space.
\end{corollary}

This last fact suggest the existence of a nice boundary at infinity of the
Kac-Moody symmetric spaces; the existence of rigidity of quotient spaces and a Mostow-type behaviour seem quite probable.

\section{Kac-Moody symmetric spaces of the Euclidean type}
\index{Kac-Moody symmetric space of Euclidean type}
\label{KMSSofEuclideantype}

In the finite dimensional situation the simply connected $n$-dimensional Euclidean symmetric space is isometric to $\mathbb{R}^n$.  This space is highly reducible: It consists of $n$ copies of $\mathbb{R}$, the $1$-dimensional symmetric space. It is non-compact and furthermore diffeomorphic to a vector space. In addition there is a ``compact dual'', namely the torus $T^n:=(S^1)n$; in dimension $1$ this is a circle $S^1$. The fundamental group of $S^1$ is $\mathbb{Z}$, the fundamental group of a $n$-dimensional torus is $\mathbb{Z}^n$. Hence we have the description
\begin{displaymath}
T^n\cong \mathbb{R}^n/ \mathbb{Z}^n \quad\textrm{and}\quad S^1\cong \mathbb{R}/\mathbb{Z}\, .
\end{displaymath}

In both cases the OSLA is isomorphic to $\mathfrak{g}\cong SO(1)\ltimes \mathbb{R}$, the involution is the inversion $x\mapsto -x$. The fixed subgroup is \begin{displaymath}
SO(1):=\{\pm 1\}\, .
\end{displaymath} 

In terms of Lie algebras and the exponential function we can describe the ``duality'' between
\begin{displaymath}
\mathbb{R}\Leftrightarrow S^1
\end{displaymath}
as follows:

Start with the space of ``compact type'', namely $S^1$. Its tangential space $T_p(S^1)$ --- the $\mathcal{P}$-component of the Cartan decomposition --- is isomorphic to $\mathbb{R}$. Its complexification is isomorphic to $\mathbb{C}$. $T_p(S^1)$ is embedded into $\mathbb{C}$ as the imaginary axis.  Dualization exchanges the subspace $i\mathbb{R}$ by the subspace $\mathbb{R}$.
The exponential mapping in the first case (writing the symmetric space $\mathbb{R}\cong \mathbb{R}^+$) is 

\begin{displaymath}
\exp: i\mathbb{R}\longrightarrow S^1,\quad ix\mapsto \exp(ix)=\cos(x)+i\sin(x)\, ,
\end{displaymath}
while it is in the second case
\begin{displaymath}
\exp: \mathbb{R}\longrightarrow \mathbb{R},\quad x\mapsto \exp(x)=\cosh(x)+\sinh(x)\, .
\end{displaymath}

We have furthermore the complex relation

\begin{displaymath}
\exp: \mathbb{C}\longrightarrow \mathbb{C}^*,\quad z\mapsto \exp(z)\, .
\end{displaymath}

Hence we can desribe the symmetric space of non-compact type also as 
\begin{displaymath}
M=\mathbb{C}^*/S^1\, .
\end{displaymath}
similar to the description of finite dimensional symmetric spaces of type $4$.
 
The theory of Euclidean affine Kac-Moody symmetric spaces is similar:

An affine geometric Kac-Moody algebra of Euclidean type is the double extension of a loop algebra $\widehat{L}(\mathfrak{g}_{\mathbb{R}})$ where $\mathfrak{g}$ denotes a finite dimensional real abelian Lie algebra. We omit the involution $\sigma$ in the notation as $\sigma$ is the identity for Euclidean Kac-Moody algebras. This real abelian Lie algebra is a real form of a complex Lie algebra $\widehat{L}(\mathfrak{g}_{\mathbb{C}})$ with $\mathfrak{g}_{\mathbb{C}}=\mathfrak{g}_{\mathbb{R}}\oplus i \mathfrak{g}_{\mathbb{R}}$. Hence, we have two possible real forms: $\widehat{L}(\mathfrak{g}_{\mathbb{R}})$ and $\widehat{L}(i\mathfrak{g}_{\mathbb{R}})$. The loop algebras are isomorphic as are the extensions. Nevertheless, the associated Kac-Moody groups are different. 

\begin{definition}[Kac-Moody symmetric space of Euclidean type - 1. case]
Let $\mathfrak{g}_{\mathbb{C}}:=\mathbb{C}^n$ and let $G_{\mathbb{C}}\cong (\mathbb{C}^*)^n$ be the associated Lie group. We put
\begin{displaymath}{(MG_{\mathbb{C}})}_{\mathbb{R}}:=\{f:\mathbb{C}^*\longrightarrow G_{\mathbb{C}}|f(S^1)\subset i\mathfrak{g}_{\mathbb{R}}\subset \mathfrak{g}_{\mathbb{C}}\}\, .\end{displaymath}
Then the space $\widehat{M\mathfrak{g}_{\mathbb{C}}}_{\mathbb{R}}$ is a Kac-Moody symmetric space of the Euclidean type.  
\end{definition}

Spaces of this type can occur as a subspace of an indecomposable Kac-Moody symmetric space having subspaces of compact type.

\begin{definition}[Kac-Moody symmetric space of Euclidean type - 2. case]
Let $\mathfrak{g}_{\mathbb{C}}:=\mathbb{C}^n$ and let $G_{\mathbb{C}}\cong (\mathbb{C}^*)^n$ be the associated Lie group.
\begin{displaymath}{(MG_{\mathbb{C}})}_{\mathbb{R}}:=\{f:\mathbb{C}^*\longrightarrow G_{\mathbb{C}}|f(S^1)\subset i\mathfrak{g}_{\mathbb{R}}\subset \mathfrak{g}_{\mathbb{C}}\}\end{displaymath}
Then the space $M=\widehat{M\mathfrak{g}_{\mathbb{C}}}/\widehat{M\mathfrak{g}_{\mathbb{C}}}_{\mathbb{R}}$ is a Kac-Moody symmetric space of the Euclidean type.  
\end{definition}

Spaces of this type can occur as a subspace of an indecomposable Kac-Moody symmetric space having subspaces of non-compact type.

\begin{lemma}
A Kac-Moody symmetric space of Euclidean type is irreducibe iff $\mathfrak{g}\cong \mathbb{C}$.
\end{lemma}

\begin{lemma}
A Kac-Moody symmetric space of Euclidean type is flat.
\end{lemma}

The most important feature of Kac-Moody symmetric spaces of the Euclidean type is that their exponential maps behave well: 

\begin{proposition}
The exponential map describes a tame diffeomorphism $\Mexp: \widehat{U}\longrightarrow \widehat{V}$, where $\widehat{U}\subset \widehat{M\mathfrak{g}}$ and $\widehat{V}\subset \widehat{MG}$.
\end{proposition}

\noindent This is a direct consequence of corollary~\ref{examplesofliegroupswithdiffexp}, stating that the exponential map of a loop group associated to a Lie groups whose universal cover is biholomorphically equivalent to $\mathbb{C}^n$ is a local diffeomorphism.

\noindent In connection to this result we also want to mention the following theorem of Galanis~\cite{Galanis96}, that describes in some sense the inverse situation:

\begin{theorem}[Galanis]
 Let $G$ be a commutative Fr\'echet Lie group modelled on a Fr\'echet space $F$. Assume the group is a strong exponential Lie group (i.e. a group such that $\exp$ is a local diffeomorphism). Then it is a projective limit Banach Lie group.
\end{theorem}

As the loop group part of an Euclidean Kac-Moody group is exponential we get immediately that an Euclidean Kac-Moody symmetric space is strong exponential. Furthermore it carries a projective limit Banach structure modelled via the exponential maps on the tangential Kac-Moody algebra.

\begin{theorem}
Let $G$ be a simple compact Lie group, $T\subset G$ a maximal torus. The space $\widehat{MT}$ is a Kac-Moody symmetric space of Euclidean type. 
\end{theorem}

As is the case for finite dimensional symmetric spaces of the Euclidean type the isometry group is much bigger, namely a semidirect product of an Euclidean Kac-Moody group with the isotropy group of a point. 

\noindent The group $\widetilde{MG}$ is a Heisenberg group~\cite{PressleySegal86}, chapter 9.5. .

\section{Duality}

Similar to simply connected finite dimensional Riemannian symmetric spaces also simply connected Kac-Moody symmetric spaces come in pairs. Those pairs share various structure properties. For the finite dimensional blueprint see~\cite{Helgason01}, chapter V.

We summarize the most important properties in the following theorem:

\begin{theorem}[Duality of Kac-Moody symmetric spaces]
\index{duality of Kac-Moody symmetric spaces}
There is a duality relation between Kac-Moody symmetric spaces of the compact type and Kac Moody symmetric spaces of the non-compact type. Especially there are dualities between Kac-Moody symmetric spaces of type:
\begin{center}\begin{tabular}{ccc}
type I&$\Leftrightarrow$& type III\\
type II&$\Leftrightarrow$&type IV
\end{tabular} 
\end{center}
Especially:
\begin{enumerate}
  \item  Dual Kac-Moody symmetric spaces have the same rank.
	\item  Dual Kac-Moody symmetric spaces have the same isotropy representation.
	\item  The curvature tensors of dual Kac-Moody symmetric spaces differ only by the sign. 
\end{enumerate}
\end{theorem}

\begin{proof}
\begin{enumerate}
\item Dual Kac-Moody symmetric spaces have the same $\mathcal{P}$. Hence the dimension of maximal abelian subalgebras and hence the rank is the same.
\item Dual Kac-Moody symmetric spaces have the same loop group of compact type as isotropy group and the same $\mathcal{P}$. Let $K:\mathcal{P}\longrightarrow \mathcal{P}$ be the isotropy representation of a Kac-Moody symmetric space of the compact type. Imposing $\mathbb{C}$-linearity we can define a complexification
\begin{displaymath}
K:\mathcal{P}\oplus i\mathcal{P}\longrightarrow \mathcal{P}\oplus i \mathcal{P}
\end{displaymath}
As $K$ is real the imaginary part and the real part are preserved. Furthermore the subrepresentations
$K:\mathcal{P}\longrightarrow \mathcal{P}$ and $K:i\mathcal{P}\longrightarrow i\mathcal{P}$ are isomorphic. But the second one is the isotropy representation of the dual Kac-Moody symmetric space of non-compact type.
\item This is a consequence of Lemmata \ref{curvatureestimatecompacttype} and \ref{curvatureestimatenon-compacttype}.
\end{enumerate}
\end{proof}

\section{The isotropy representation of Kac-Moody symmetric spaces}

In this section we show that the isotropy representations of Kac-Moody symmetric spaces coincide with the polar representation described by Christian Gro\ss ~\cite{Gross00}.

\subsection{The finite dimensional blueprint}

A reference for the material in this subsection is \cite{Berndt03}.

\begin{definition}[polar action]
\index{polar action}
Let $M$ be a Riemannian manifold.  An isometric action $G:M\longrightarrow M$ is called polar if  there exists a complete, embedded, closed submanifold $\Sigma \subset M$ that meets  each orbit orthogonally.
\end{definition}

\begin{definition}[polar representation]
\index{polar representation}
 A polar representation is a polar action on an Euclidean vector space, acting by linear transformations. 
\end{definition}

\begin{example}[adjoint representation]
\index{Adjoint representation}
Let $G$ be a compact simple Lie group, $\mathfrak{g}$ its Lie algebra. The adjoint representation is polar. Sections are the Cartan subalgebras.
\end{example}

\begin{example}[isotropy representation]
\index{isotropy representation}
 Let $(G, K)$ be a Riemannian symmetric pair, $M=G/K$ the corresponding symmetric space and $m=eK\in M$ (hence $K=I(M)_m$). The action of $K$ on $M$ induces an action on $T_mM$. Let $\mathfrak{g}=\mathfrak{k}\oplus \mathfrak{p}$ be the corresponding Cartan decomposition (i.e.\ $\mathfrak{k}=Lie(K)$). Then $\mathfrak{p}\cong T_mM$.  Using this isomorphism, we get an action of $K$ on $\mathfrak{p}$. This action is called the isotropy representation of $M$. 
\end{example}

\begin{example}[type II-symmetric spaces]
 For symmetric spaces of type II, that is compact Lie groups, the isotropy representation coincides with the Adjoint representation.
\end{example}

\begin{lemma}
The isotropy representation of a symmetric space is polar. Each Cartan subalgebra in $\mathfrak{p}$ is a section. 
\end{lemma}

\begin{lemma}
 Suppose $M$ and $M^D$ are dual symmetric spaces. Then the isotropy representations of $M$ and $M^D$ are isomorphic.
\end{lemma}

\begin{example}[The $n$\ndash sphere $S^n$ and the $n$\ndash hyperbolic space $\mathbb{H}^n$]
 The $n$\ndash sphere $S^n:=\{x\in \mathbb{R}^{n+1}| |x|=1\}\cong SO(n+1)/SO(n)$. The Lie group $SO(n+1)$ acts transitively on $S^n$. The stabilizer of a point $p$ is isomorphic to $SO(n)$. 

To define the hyperbolic space we use $\mathbb{R}^{n,1}$, that is $(\mathbb{R}^{n+1}, g_{\textrm{Lorentz}})$ where $g_{\textrm{Lorentz}}$ is defined by
\begin{displaymath}
g_{\textrm{Lorentz}}=\left(\begin{array}{cccc}1&0&\dots&0\\ 0&\ddots&\ddots&\vdots\\\vdots&\ddots&1&0\\0&\dots&0&-1  \end{array}\right)\, .
 \end{displaymath}
The hyperbolic $n$\ndash is defined by
\begin{displaymath}
\mathbb{H}^n:=\{x\in \mathbb{R}^{n-1}||x|=-1, x_{n+1}>0\}\cong SO(n,1)/SO(n) \, .
\end{displaymath}
The isotropy representation of both those spaces is the representation:
\begin{displaymath}
 SO(n):\mathbb{R}^n\longrightarrow \mathbb{R}^n,\quad (A,x)\mapsto A\cdot x\, .
\end{displaymath}
\end{example}

\begin{theorem}[Dadok]
\index{theorem of Dadok}
 Every polar representation on $\mathbb{R}^n$ is orbit equivalent to the isotropy representation of a symmetric space.
\end{theorem}

\subsection{The isotropy representation of Kac-Moody symmetric spaces}
\index{isotropy representation of Kac-Moody symmetric space}
\begin{definition}
Let $M$ be a Kac-Moody symmetric space, $G$ the Kac-Moody isometry group, $p\in M$ and $K_p:=\{g\in G| g(p)=p\}$ the isotropy group of $p$.

The representation
\begin{displaymath}
  K_p:T_pM\longrightarrow T_pM 
\end{displaymath}
is called the isotropy representation of $M$.
\end{definition}

\begin{theorem}
The isotropy representation of a Kac-Moody symmetric space is essentially equivalent to a $P(G,H)$-action.
\end{theorem}

Recall that two representations $G:V\longrightarrow V$ and $G: V'\longrightarrow V'$ are essentially equivalent if there is a injective map $\varphi: V\longrightarrow V'$ which is $G$-equivariant with dense image~\cite{PressleySegal86}, p.~167.

\begin{theorem}
Let $M$ and $M^D$ be dual Kac-Moody symmetric spaces. The isotropy representations of $M$ and $M^D$ are isomorphic.
\end{theorem}

\begin{lemma}
 For a Kac-Moody symmetric space of type II and type IV, the isotropy representation coincides with the Adjoint action. 
\end{lemma}

We have seen that the Adjoint action is isometric; hence the length of vectors is preserved. Furthermore the $d$-coefficient $r_d$ is preserved. Hence the 
Adjoint action induces an action on the space 
\begin{displaymath}H_{r,\ell}:=\{x\in \widehat{L}(\mathfrak{g}, \sigma)\|x\|=-\ell, r_d=r\}\, .\end{displaymath}
This action is essentially equivalent to the gauge action of $H^1$-loop groups on the Hilbert space of $H^0$-loops in $\mathfrak{g}$ and hence polar~\cite{terng95}.

For the following definition see for the case of Lie groups~\cite{Berndt03}, for the case of affine Kac-Moody groups~\cite{Gross00}. 

\begin{definition}[$s$-representation]
\index{$s$-representation}
Let $\mathcal{G}$ be a simple or affine Kac-Moody Lie algebra of compact type, $\rho$ an involution and $\mathcal{G}=\mathcal{K}\oplus \mathcal{P}$ the decomposition into the $\pm 1$-eigenspaces of $\rho$. Let $K$ be the Lie group resp.\ affine Kac-Moody group associated to $\mathcal{K}$.
The representation
\begin{displaymath}
S:K\longrightarrow Aut(\mathcal{P})
\end{displaymath}
is called $S$-representation.  
\end{definition}

\begin{lemma}
 For a Kac-Moody symmetric space of type I and type III, the isotropy representation is an $s$-representation. 
\end{lemma}

\begin{proof}
by direct calculation.
\end{proof}

By the following theorem of Christian Gross~\cite{Gross00}, p.~340 we get polarity.
\begin{theorem}
Suppose $K$ is a closed subgroup in $L(G,\sigma)$, then its affine $s$-representation on $\mathcal{P}$ is polar.
\end{theorem}

The converse result, a Dadok-type theorem, is not known for affine Kac-Moody symmetric spaces and polar representations on Hilbert case.

\section{Lie triple systems}
\index{Lie triple system}
Recall the definition~\cite{Helgason01}, p.~224 
\begin{definition}[Lie triple system]
Let $\mathfrak{g}$ be a real Lie algebra and $\mathfrak{m}$ be a subspace. $\mathfrak{m}$ is called a Lie triple system if for any three elements $f,g,h\in \mathfrak{m}$ we have:
\begin{displaymath}
[f, [g,h]]\in \mathfrak{m}\, .
\end{displaymath}
\end{definition}

\begin{proposition}
Let $M$ be a Kac-Moody symmetric space such that $M=\widehat{MG}^{\sigma}/Fix(\rho)$, where $Fix(\rho)$ is the isotropy group of a point $p\in M$. $\rho$ induces the decomposition: $\widehat{MG}^{\sigma}=\mathcal{P}\oplus \mathcal{K}$. Identify $T_pM\cong \mathcal{P}$. Then $\mathcal{P}$ has the structure of a Lie triple system.
\end{proposition}

\begin{proof}
$\mathcal{P}$ is the $-1$-eigenspace of the derivative $d\rho:\widehat{M\mathfrak{g}}^{\sigma}\longrightarrow \widehat{M\mathfrak{g}}^{\sigma}$ --- whence the result.
\end{proof}

In the finite dimensional theory of symmetric spaces one has an equivalence between Lie triple systems and totally geodesic submanifolds. We conjecture this to be true in the situation of affine Kac-Moody symmetric spaces as well. While it is easy to check that a totally geodesic subspace defines a Lie triple system, the converse is not clear.

\chapter{Geometry of Kac-Moody symmetric spaces}
\label{chap:geometry}

\section{The structure of finite dimensional flats}
\label{sect:structure of finite dim flats}
\index{flats are conjugate}
Let $M$ be a finite dimensional Riemannian symmetric space. It is well-known that all maximal flats are conjugate~\cite{Helgason01}. The proof of this result may be given in several steps: The first step consists in proving that all maximal flats containing a special point $p_0\in M$ are conjugate under the action of the isometry group. The second step generalizes this result to all flats using homogeneity. Let us describe this in a little more detail: For irreducible symmetric spaces of type $II$, that is compact Lie groups, maximal flats through the identity element correspond bijectively to maximal tori. Hence for flats containing the identity element the first step is equivalent to all maximal tori beeing conjugate; for all flats the result follows by transitivity. For symmetric spaces $G/K$ of type $I$, one has to study the decomposition $\mathfrak{g}=\mathfrak{k}\oplus \mathfrak{p}$ into the $\pm 1$-eigenspaces and prove that all maximal flats in $\mathfrak{p}$ are conjugate. For the dual symmetric spaces of non-compact type one can deduce the assertion via the isotropy representation: the isotropy representation is isomorphic for a symmetric space of compact type and its non-compact dual. Thus all maximal abelian subalgebras in $\mathfrak{p}$ are conjugate and thus via the exponential map also all maximal flats through one point $p_0$ are conjugate. For details~\cite{Helgason01}.

To understand the infinite dimensional situation we start with the investigation of flats in a Kac-Moody group of compact type $\widehat{MG}$ and its Kac-Moody algebra $\widehat{M\mathfrak{g}}$. Then we study flats in Kac-Moody symmetric spaces of type $I$. Third we turn to the Kac-Moody symmetric spaces of  non-compact type.

\subsection{Flats in Kac-Moody symmetric spaces of type \protect\boldmath$II$}

In a Kac-Moody algebra $\widehat{L}(\mathfrak{g}, \sigma)$ there are two types of maximal abelian subalgebras: infinite dimensional ones and finite dimensional ones. 

\begin{enumerate}
\item The first class consisting of infinite dimensional subalgebras --- called subalgebras of infinite type --- consists of subalgebras that are contained in $\widetilde{L}(\mathfrak{g}, \sigma)$. For the affine Kac-Moody groups of smooth loops this is proved by B.\ Popescu~\cite{Popescu06}; his argument applies as well to the situation of $\widehat{A_n\mathfrak{g}}_{\mathbb{R}}$ and to $\widehat{M\mathfrak{g}}_{\mathbb{R}}$.

\item The second class --- called subalgebras of finite type --- consists of subalgebras that are not contained in $\widetilde{L}(\mathfrak{g}, \sigma)$. Subalgebras of this type always contain an element of the form $f+c+d$, where $f\in L(\mathfrak{g}, \sigma)$. Furthermore they contain $c\mathbb{R}$. Thus a subalgebra of finite type $\widehat{\mathfrak{a}}$ can be described as $\widehat{\mathfrak{a}}=\mathfrak{a}\oplus \mathbb{R}c\oplus{\mathbb{R}}(d+f)$, where $\mathfrak{a}$ is an abelian subalgebra of $M\mathfrak{g}$. .B.\ Popescu proves  in the setting of affine Kac-Moody algebras of $C^{\infty}$-loops that $\textrm{dim}(\widehat{\mathfrak{a}})=\textrm{rank}(\mathfrak{g})+2$. 

The crucial observation in his proof is that any element $v\in \mathfrak{a}$ satisfies the following Lax equation:
\begin{displaymath}[v, d+f]=0 \Leftrightarrow v'=[v,f]\Leftrightarrow v'=[v,f]_0, \omega(v, f)=0\,.\end{displaymath}

\noindent Thus one has to find the space of solutions of this Lax equation under the condition $\omega(v, f)=0$. The proof applies as well to the setting of holomorphic loops.\index{Lax equation}  
\end{enumerate}

\begin{remark}
Those two different classes of flats reappear in different guises: a prominent example is the study of co-adjoint actions of loop resp.\ affine Kac-Moody groups. The orbits of the co-adjoint action of loop groups $L(G, \sigma)$ have infinite codimension, while the orbits of affine Kac-Moody groups $\widehat{L}(G, \sigma)$ have finite codimension provided one studies an orbit such that $r_d\not=0$. Identifying the dual loop algebra $L^*(\mathfrak{g}, \sigma)$ with the $\{r_d=0\}$-section of $\widehat{L}^*(\mathfrak{g}, \sigma)$, the loop group action coincides with the action of the affine group.   
For further details and the definitions we omitted see the recent reference~\cite{Khesin09}, remark 1.20.
\end{remark}

From our point of view, the important type of flats are the finite dimensional ones.
By definition an abelian subalgebra $\widehat{\mathfrak{a}}\subset \widehat{M\mathfrak{g}}$ of finite type contains an element $\widetilde{u}+d$ with $\widetilde{u}\subset \widetilde{M\mathfrak{g}}$.

\noindent Thus every finite dimensional flat intersects the sphere of radius $-l^2, l\in \mathbb{R}\backslash \{0\}$. We will now prove the converse, namely that any element in a sphere of radius $-l^2$ in an affine Kac-Moody algebra is contained in some finite dimensional flat.

To this end we identify the two sheets of the intersection of the  sphere of radius $l$ with the planes $r_d=\pm r\not=0$ , called $H_{l,r}$, with two copies of the tame Fr\'echet space $M\mathfrak{g}$. Under this identification the adjoint action restricted to $H_{l,r}$ coincides with the gauge action. Hence, we can use the analytic results of section~\ref{PolaractionsontameFrechetspaces}.

As consequence of theorem~\ref{polaractiononmg} we get:

\begin{lemma}\label{finiteflatscoverunitsphere}
Let $H_{l,r}:=\{u \in \widehat{M\mathfrak{g}}| \|u\|=l, r_{d}=r\not=0\}$. For every $u\in H_{l,r}$ there is an abelian subalgebra $\widehat{\mathfrak{a}}$ containing $u$. This abelian subalgebra is $\widehat{\mathfrak{a}}=\{\mathfrak{a}+\mathbb{R}c+\mathbb{R}(d+u)\}$ where $\mathfrak{a}= \widehat{\mathfrak{a}}\cap M{\mathfrak{g}}$ is a finite dimensional abelian subalgebra whose dimension is the rank $\mathfrak{g}$.
\end{lemma}

\noindent Note that this lemma is stated for Kac-Moody algebras of $C^{\infty}$-loops in~\cite{Heintze06}, in~\cite{Popescu05} and in~\cite{Popescu06}.

\begin{proof}
For the proof one has to check that an element $v\in M \mathfrak{g}$ commutes with an element of the form $\widetilde{u}+d$,  iff $[v,u]=v'$. This differential equation has the solution $v(t)=\textrm{Ad}\phi(t) v_0$ for  a solution $\phi(t)$ of the differential equation $\phi'(t)= u(t) \phi (t)$~\cite{Guest97}. 

\noindent To get closed loops we need the condition $v(2\pi)=\textrm{Ad}\phi(2\pi) v_0=v_0$. Furthermore for two solutions $v$ and $v'$ we need that $[v,v']=0$; hence we get the condition $[v(0), v'(0)]=0$ -- thus for every element $\widetilde{u}+d$ there is an abelian subalgebra $\mathfrak{a}\subset \mathfrak{g}=\textrm{Lie}(G)$, such that $v(0)\in \mathfrak{a}$ for every element $v(t)$ in the flat containing $\widetilde{u}+d$. 
\end{proof}

\begin{lemma}
\label{conjugateabeliansubalgebras}
All finite dimensional abelian subalgebras in $\widehat{M\mathfrak{g}}$ are conjugate by elements in $\widehat{MG}$.
\end{lemma}

\begin{proof}
B.\ Popescu proves a similar result for the real Kac-Moody algebras $\widehat{L^{\infty}\mathfrak{g}}$ of compact type, constructed with $C^{\infty}$ loops~\cite{Popescu06}). Embedding 
$$\widehat{M\mathfrak{g}}\hookrightarrow \widehat{L^{\infty}\mathfrak{g}}\,,$$
we find that all finite abelian subalgebras in $\widehat{M\mathfrak{g}}$ are conjugate by elements in $\widehat{L^{\infty}G}$; we have to check that the conjugating element can be chosen to be in $\widehat{MG}$. 
This is done  --- following the blueprint of Popescu's proof --- in two steps:
\begin{enumerate}
\item First, we have to check that a flat $\widehat{\mathfrak{a}}$ is conjugate to a flag $\widehat{\mathfrak{b}}$, such that $\widehat{\mathfrak{b}}= \mathfrak{b}\oplus \mathbb{R}c \oplus \mathbb{R}(d+x)$ with $x$ constant by an element in $\widehat{MG}$.
This is a direct consequence of lemma~\ref{polaractiononmg}, stating that the gauge action of $MG$ on $M\mathfrak{g}$ is polar.
\item Second, for a given $x$ we have to check that all flats of the form of $\widehat{\mathfrak{b}}$ are conjugate. This relies on the solution of a differential equation being holomorphic iff the equation is holomorphic.\qedhere
\end{enumerate} 
\end{proof}

\begin{corollary}
Every Cartan subalgebra contains the $1$\ndash dimensional subspace $\mathbb{R}c$. Hence the residuum of $c$ is infinite dimensional.  
\end{corollary}

\begin{definition}[Flat of exponential type]
\index{flat of exponential type}
A flat in $\widehat{\mathfrak{a}}\subset\widehat{MG}$ such that $e\in \widehat{\mathfrak{a}}$ is called of exponential type iff it is in the image of the exponential map. 
\end{definition}

\begin{theorem}\label{finiteflatsconjugate}
All finite dimensional flats $\widehat{\mathfrak{a}}\subset \widehat{MG}$ of exponential type are conjugate. 
\end{theorem}

\begin{proof}
This is a direct consequence of lemma~\ref{conjugateabeliansubalgebras}.
\end{proof}

Equivalent theorems hold for the Kac-Moody algebras $\widehat{A_nG}$ and their associated Kac-Moody groups $\widehat{A_nG}$:

\begin{lemma}
\label{conjugateabeliansubalgebrasinA_ng}
All finite dimensional abelian subalgebras in $\widehat{A_n\mathfrak{g}}$ are conjugate by elements in $\widehat{A_nG}$.
\end{lemma}

\begin{theorem}
\label{finiteflatsconjugateinA_nG}
All finite dimensional flats $\widehat{\mathfrak{a}}\subset \widehat{A_nG}$  are conjugate. 
\end{theorem}

As the exponential map $\widehat{A_n\exp}:\widehat{A_n\mathfrak{g}}\longrightarrow \widehat{A_nG} $ is  a local diffeomorphism, every flat containing $e$ is of exponential type. Thus we can omit this condition.

Study now the $ILB$-system: $\lim \widehat{A_{n}G} =\widehat{MG}$. As $\widehat{MG} \subset \widehat{A_nG}$ for each $n\in \mathbb{N}$, any flat in $\widehat{MG}$ is contained in a flat in $\widehat{A_{n}G}$. 

\noindent This observation contains the result:

\begin{theorem}
Any two finite flats in $\widehat{MG}$ are conjugate by elements in $\widehat{A_nG}$ for every $n\in \mathbb{N}$.
\end{theorem}

\subsection{Flats in Kac-Moody symmetric spaces of type \protect\boldmath$I$}

Let $M$ be an affine Kac-Moody symmetric space of compact type and
let $\widehat{L}(\mathfrak{g}, \sigma)=\mathcal{K}\oplus \mathcal{P}$ be the decomposition of its associated Kac-Moody algebra into the $\pm 1$-eigenspaces of the involution defining $M$.

\begin{lemma}
\label{subalgebrasinp}
All subalgebras of finite type contained in $\mathcal{P}$ are conjugate. 
\end{lemma}

B.\ Popescu proves a similar theorem for the affine Kac-Moody algebra of $C^{\infty}$-loops~\cite{Popescu06}. His proof generalizes to the holomorphic situation.

As we did for Kac-Moody symmetric spaces of type $II$ we call a flat in $M$ of ``exponential type'' iff it is in the image of the exponential map.

\begin{theorem}
All flats of finite exponential type are conjugate by elements in the isotropy group of $M_{\mathbb{C^*}}$.
\end{theorem}

\begin{proof}
Via the exponential map we use lemma~\ref{subalgebrasinp}.
\end{proof}

Analogously to the case of type $I$ we study the inverse limit structure: let $M_n$ be the Kac-Moody symmetric space corresponding to $\widehat{A_nG}$. We have then a series of spaces $M_n$ such that $\lim M_n = M$ and get the result:

\begin{theorem}
All flats of finite type in $M_n$ are conjugate by elements in the isotropy group of $M_n$.
\end{theorem}

\noindent As $M\subset M_n$ for all $n$, we find

\begin{theorem}
All flats of finite type in $M$ are conjugate by elements in the isotropy group of $M_n$ for all $n$.
\end{theorem}

\subsection{Flats in Kac-Moody symmetric spaces of non-compact type}

Let $M$ be an affine Kac-Moody symmetric space of  non-compact type; we use the canonical identification of $T_eM\simeq \mathcal{P}$. Then there is an exponential map $\widehat{\Mexp}: \mathcal{P}\longrightarrow M$. Let $M_c$ denote the affine Kac-Moody symmetric space of compact type that is dual to $M$. Its tangential space is isomorphic to $\mathcal{P}$. The isotropy representation is the same -- thus on the Lie algebra level, abelian subspaces in  $T_eM$ correspond bijectively to abelian subspaces in $T_eM_c$. Especially, the adjoint action is transitive on the space of all flats. 

Via the exponential map, this shows that flats of exponential type (defined as in the case of Kac-Moody symmetric spaces of compact type) are conjugate. Investigation of the inverse limit construction yields that all flats are conjugate by elements in the inverse limit groups.

\section{Geometry of Kac-Moody symmetric spaces of type \protect\boldmath$II$}

The Kac-Moody symmetric spaces of type $II$ are the compact real form of affine Kac-Moody groups. As every affine Kac-Moody group has a --- up to conjugation unique --- compact real form, their classification coincides with the classification of affine Kac-Moody groups resp.\ algebras. In this section we want to explain that their geometry is modelled according to the geometry of finite dimensional Riemannian symmetric spaces of type $II$. To understand the geometry of those spaces we need to focus on two building blocks:

\begin{enumerate}
 \item Tori of finite type: they describe the structure of flat subspaces. The various types of those subalgebras are distinguished by their tori. 
 \item Residua: the residua describe the geometry of the space around singular hyperplanes. In case of type $II$\ndash symmetric spaces all residua are of type $SU(2)$.
\end{enumerate}

All residua of finite dimensional Riemannian symmetric spaces resp. affine Kac-Moody symmetric spaces of type $II$ correspond to the compact Lie group $SU(2)$:

\begin{example}[$SU(2)$]
\index{$SU(2)$}
The special unitary group of rank $1$ is defined by
\begin{displaymath} 
 SU(2):=\left\{\left(\begin{array}{cc} \alpha&\beta\\ -\overline{\beta}&\overline\alpha \end{array}\right)\mid \alpha, \beta\in \mathbb{C}, |\alpha|^2+|\beta|^2=1\right\}.
\end{displaymath}
\end{example}

Hence we can understand $SU(2)$ as the $3$\ndash dimensional unit sphere in $\mathbb{C}^2$. To understand $SU(2)$ as a Riemannian symmetric space we need a metric. The metric is defined in two steps: First we define the unique $Ad$\ndash invariant metric on the Lie algebra (resp.\ the tangent space of the symmetric space). Then in a second step define the metric on the whole Lie group by left translation.

Any $Ad$\ndash invariant scalar product on a simple Lie algebra is a multiple of the Cartan-Killing form:
\begin{definition}
The Cartan-Killing form $B(X,Y)$ of a simple Lie algebra is defined by 
\begin{displaymath}
B(X,Y)=\textrm{trace}\left(\textrm{ad}(X) \circ\textrm{ad}(Y)\right).
\end{displaymath}
\end{definition}

For a Lie algebra in matrix representation we can use a second expression, which is easier to evaluate: 
\begin{displaymath}
B(X,Y)=\textrm{trace}\left(XY\right).
\end{displaymath}

Doing the calculations for $SU(2)$ shows that $SU(2)$ together with its $Ad$\ndash invariant metric is a $3$\ndash sphere. As a coset space we can describe this space as
\begin{displaymath}
 SU(3)\cong SU(2)\times SU(2)/\Delta(SU(2)) \cong S^3\, .
\end{displaymath}

\subsection{A finite dimensional example: \protect\boldmath$SU(n)$}
\index{$SU(n)$}
Let us now describe the finite dimensional blueprint for the easiest example  $M=SU(n)$ of unitary $n\times n$\ndash matrices. This space has a maximal flat $H_0$ of purely imaginary diagonal matrices:

\begin{displaymath}
 H_0:=\{\textrm{diag}(a_1, \dots, a_n)| a_k=e^{i x_k}\ \textrm{and}\ \prod a_k=1)\}\, .
\end{displaymath}

As Weyl group we get  the symmetric group acting as the permutation group of the entries of the flat: Define $H_{12}$ as the fixed torus of the permutation $\sigma_{12}\subset Sym(n)$. We perform the same steps as before: Start by defining

\begin{displaymath}
SU_{12}(n):=\left\{\left(\begin{array}{c|ccc}aA&0&0&0\\ \hline 0&a_3&0&0\\ 0&0&\ddots&0\\0&0&0&a_n   \end{array}\right)\mid A\in SU(2),\quad a^2\prod_{k\geq 3} a_k=1 \right\} \, .
\end{displaymath}

We get the embedding (well-known from Lie theory):

\begin{displaymath}
SU_{12}(n, \mathbb{R}) \hookrightarrow SU(n)
\end{displaymath}

We want to study the subspace $SU_{12}(n, \mathbb{R})$. As $SU(2)\cong S^3$  we can note the space:

\begin{displaymath}
SU_{12}(n)=\left\{\left(\begin{array}{c|ccc}  a_1 S^3&0&0&0\\ \hline 0&a_3&0&0\\ 0&0&\ddots &0\\0&0&0&a_n   \end{array}
\right) \right\} \,.
\end{displaymath}

In this form we can describe explicitly the geometric structure of $SU_{12}(n)$: It is the direct product of a $3$\ndash dimensional sphere $S^3$ with an $(n-2)$\ndash dimensional abelian subspace. Maximal abelian subspaces in $S^3$ are $1$\ndash dimensional and elements in $S^3$ are orthogonal to our flat. Hence maximal flats in $SU_{12}(n)$ are $(n-1)$\ndash dimensional. Especially the flat $H_0$ is a maximal flat in $SU_{12}(n)$. Flats in this space corresponding to flats containing our singular hyperplane $H_{12}$ are exactly those that pass through the identity $x\in SU(2)=S^3$. The space of geodesics through one point in a $3$\ndash sphere is exactly $\mathbb{R}P(2)$, $2$\ndash dimensional projective space. As a consequence we note:

\begin{lemma}
 The space of maximal flats containing $H_{12}$ (or more generally $H_{ij}$) is isomorphic to $\mathbb{R}P(2)$.
\end{lemma}

The marked Dynkin diagram of this symmetric space is hence the following:

$\begin{picture}(80,20)
\multiput(3,2)(25,0){6}{\circle{5}} 
\multiput(7,2)(25,0){2}{\line(2,0){17}}\multiput(82,2)(25,0){2}{\line(2,0){17}}\multiput(59,2)(25,0){1}{\dots}
\multiput(0,7)(25,0){6}{\small $2$\normalsize}
\end{picture} $

\subsection{An infinite dimensional example: \protect\boldmath$\widehat{MSU}(n)$}
\index{$\widehat{MSU}(n)$}
Let us now turn to the infinite dimensional example $\widehat{MSU}(n)$: $\widehat{MSU}(n)$ is the Kac-Moody symmetric space of ``compact type'' of type $\widetilde{A}_{n-1}$, hence the Kac-Moody analogue of the unitary groups $SU(n)$. Recall that $SU(n)$ has as Weyl group the symmetric group $Sym(n)$; around each singular hyperplane $F_{ij}$ we found as a local building block a  $SU(2)$\ndash subgroup (hence geometrically and topologically a $3$\ndash sphere). The set of other flats containing a singular hyperplane $F_{ij}$ forms a complex projective space $\mathbb{C}P(1)$. We want to recover similar features in $\widehat{MSU}_n$.

We start with the construction of the symmetric space:

\begin{definition}[$\widehat{MSU}_n$]
The loop group is defined by: 
\begin{displaymath}
MSU_n:=\left\{f\in MSL(n,\mathbb{C})|f(S^1)\subset SU(n)  \right\}\, .
\end{displaymath}
The Kac-Moody group $\widehat{MSU}_n$ is the torus extension of $MSU(n)$\, .
\end{definition}

Analogously the Kac-Moody algebra is defined by

\begin{definition}[$\widehat{M\mathfrak{su}}_n$]
The loop algebra is defined by: 
\begin{displaymath}
M\mathfrak{su}_n:=\left\{f\in M\mathfrak{sl}(n,\mathbb{C})|f(S^1)\subset \mathfrak{su}(n)  \right\}\, .
\end{displaymath}
The Kac-Moody algebra $\widehat{M\mathfrak{su}}_n$ is the torus extension of $M\mathfrak{su}(n)$.
\end{definition}

Recall the $\textrm{Ad}$\ndash invariant Lorentz scalar product we have defined on the Kac-Moody algebra $\widehat{M\mathfrak{su}}_n$: 
\begin{align*}
 \langle f,g\rangle &:=\frac{1}{2\pi}\int_{S^1}\langle f(t), g(t)\rangle dt\, ,\\
\langle c,d\rangle&:=-1\, ,\\
\langle c,c\rangle&:=\langle d,d\rangle=0\, ,\\
\langle c,f\rangle&:=\langle d,f\rangle=0\, .
\end{align*}

For elements $(f,r_c,r_d)$ and $(g, s_c, s_d)$ such that $f$, $g\in M\mathfrak{su}_n$ and $r_c,s_c$ denote the coefficients of the $c$\ndash extension and $r_d,s_d$ denote the coefficients of the $d$\ndash extension we get:
\begin{displaymath}
 \langle (f,r_c,r_d),(g, s_c, s_d)\rangle=\frac{1}{2\pi}\int_{S^1}\langle f(t), g(t)\rangle dt - r_cs_d-r_ds_c\ .
\end{displaymath}

Let us remark that the $c,d$\ndash extension is orthogonal onto the loop group $MSU(n)$. This metric gives a metric on the whole space $\widehat{MSU}(n)$ by left translation.

Let us first investigate the geometry on the subgroup of constant loops
\begin{displaymath}
 \varphi: SU(n)\hookrightarrow \widehat{MSU}(n)\, . 
\end{displaymath}

\begin{lemma}
 $\varphi$ is an isometry of $SU(n)$ onto its image $\varphi(SU(n))$ equipped with the subspace metric inherited of $\widehat{MSU}(n)$.
\end{lemma}

\begin{proof}
 We study first $d\varphi:\mathfrak{su}(n)\hookrightarrow \widehat{M\mathfrak{su}}(n)$. Let $f_0,g_0\in \mathfrak{su}(n)$ and suppose $f=\varphi(f_0),g=\varphi(g_0)\in M\mathfrak{su}(n)$, hence $f\equiv f_0$ and $g\equiv g_0$ are constant. We calculate
\begin{displaymath}
 \langle f,g\rangle=\frac{1}{2\pi}\int_{S^1}\langle f,g\rangle dt=\frac{1}{2\pi}\langle f_0,g_0\rangle \int_{S^1} 1 dt=\langle f_0, g_0\rangle\, .
\end{displaymath}
Thus $d\varphi$ is an isometry. Furthermore the $c,d$\ndash extension is orthogonal to $d\varphi(\mathfrak{su}(n))$.

The (surjective) finite dimensional exponential map $\exp: \mathfrak{su}(n)\longrightarrow SU(n)$ commutes with the embedding $\varphi$:
As $c,d$ are direct factors over the subgroup of constant loops we find that the exponential function $\widehat{\Mexp}$ factors into 
\begin{displaymath}
 \widehat{\Mexp}(f\equiv f_0,ir_c,ir_d)=\left(\Mexp(f), \exp(ir_c), \exp(ir_d)\right)\, .
\end{displaymath}

Hence we get the relation
\begin{displaymath}
\widehat{\Mexp}|_{d\varphi(\mathfrak{su(n)}} d\varphi(\mathfrak{su}(n))= \varphi \circ \exp \mathfrak{su}(n)\ .
\end{displaymath}
On the subgroup $\varphi(SU(n))\subset \widehat{MSU}(n)$ the metric can be constructed by left translation with elements in $\varphi(SU(n))$. Hence left translation commutes with $\varphi$. Thus $\varphi$ is an isometry.
\end{proof}

\begin{lemma}
\label{totally_geodesic}
The submanifold $SU(n)\subset \widehat{MSU}(n)$ is a totally geodesic submanifold but not a strongly totally geodesic submanifold.
\end{lemma}

By definition a submanifold $M\subset N$ is strongly totally geodesic if any geodesic $\gamma(t)\subset N$ such that $\gamma(t)\cap M$ contains more than one point is contained completely in $SU(n)$. It is totally geodesic if any geodesic $\gamma(t)$ whose tangential vector is at some point $p\in M$ contained in $T_pM$ stays in $M$ for all time.

\begin{proof}[Proof of lemma \ref{totally_geodesic}]
We have to prove that any geodesic whose tangential vector is $T_{Id}SU(n)$ is contained in $SU(n)$ (remark that by homogeneity it is enough to check this condition for one point which we choose to be the identity). Now $T_{Id}SU(n)\cong \mathfrak{su}(n)$ consists of constant functions. Hence the result follows from the definition of the pointwise exponential function. Nevertheless there are geodesics in $\widehat{MSU}(n)$ hitting $SU(n)$ at least in two points, but not contained completely in $SU(n)$. Choose as the two points $Id$ and $-Id$ and take a geodesic $\gamma_0(t):=\exp(t\alpha_i)\subset T\subset SU(n)$ for some $\alpha_i\in \mathfrak{t}\subset\mathfrak{su}(n)$ such that $\gamma(2\pi)=Id$ and $\gamma(\pi)=-Id$. Then define the geodesic $\gamma(t)$ by $\gamma(t):=\widehat{f}(t)\gamma_0(t)\widehat{f}(t)^{-1}$. Especially 
\begin{align*}
 \gamma(\pi)&=\widehat{f}(\pi)\gamma_0(\pi)\widehat{f}(\pi)^{-1}=\\
&=\widehat{f}(\pi)\gamma_0(\pi)\widehat{f}(\pi)^{-1}=\\
&=\widehat{f}(\pi)(-Id)\widehat{f}(\pi)^{-1}=-Id\subset T\subset SU(n).
\end{align*}
\end{proof}

Hence as we know in detail the geometry of $SU(n)$, its tori, and its $SU(2)$ subgroups, we know that we will have exactly the same structure in the Kac-Moody symmetric space $\widehat{MSU(n)}$. Furthermore, as all flats of finite type are conjugate, for any torus $\widehat{T}^*\subset \widehat{MSU}(n)$ we can find an element $\widehat{f}\subset \widehat{MSU}(n)$ such that $\widehat{T}^*=\widehat{f}\widehat{T}_0\widehat{f}^{-1}$. The subgroup 
\begin{displaymath}
 SU(n)^*=\widehat{f}SU(n)\widehat{f}^{-1}
\end{displaymath}
 is another (weakly totally geodesic) submanifold of $\widehat{MSU}(n)$ which is isometric to $SU(n)$. Hence as long as we think about what is happening in a subspace of a Kac-Moody symmetric space of type $\widehat{MSU}(n)$ that is contained in a $SU(n)$ subgroup we can just think in $SU(n)$ which in the end can be thought of as a matrix group. On any torus $\widehat{f}\varphi(T)\widehat{f}^{-1}$ the structure we find is clearly the one of $\varphi(T)$, hence the one of a torus $T\subset SU(n)$. Thus $T$ carries a Riemannian metric and admits an isometric action of $Sym(n)$, the spherical Weyl group of $SU(n)$. To describe the geometry of the whole Kac-Moody groups we need thus to understand the extension by $c$ and $d$. 

So in a next step we want to investigate the structure of maximal tori (of finite type):

As all tori are conjugate we can choose one special example:

Put $\widehat{H}_0=H_0\oplus S^1\oplus S^1$. Here $H_0$ denotes the maximal torus of $SU(n)$.
The two $S^1$\ndash factors correspond to the $c$- and the $d$-extension of the Lie algebra. Over $SU(n)\subset MSU(n)$, embedded as the subgroup of constant loops, this bundle is trivial. $\widehat{H}_0$ has its Lie algebra $\widehat{\mathfrak{h}}_0=\mathfrak{h}\oplus i \mathbb{R}_c \oplus i\mathbb{R}_d$. The exponential map $\Mexp|_{\widehat{H}_0}$ is just a direct product
\begin{displaymath}
 \widehat{\Mexp}(f,ir_c,ir_d)=\left(\exp(f), \exp(ir_c), \exp(ir_d)\right)
\end{displaymath}
where $\exp(f)$ denotes the finite dimensional exponential function.

On $\widehat{H}_0$ we have a Lorentz scalar product: We start by describing its structure on $\widehat{\mathfrak{h}}_0$, the Lie algebra of $\widehat{H}_0$. On the subspace $\mathfrak{h}_0$ the scalar product is given by
\begin{displaymath}
 \langle u(t),v(t)\rangle_{\widehat{MSU}(n)}=\frac{1}{2\pi}\int_{S^1}\langle u(t),v(t)\rangle_{SU(n)} dt= \langle u(0), v(0)\rangle_{SU(n)}\equiv \langle u, v\rangle_{SU(n)}\, .
\end{displaymath}
Moreover we have $\langle c, c \rangle=\langle d, d \rangle= 0$ and $\langle c,d\rangle=-1$. Hence $c$ and $d$ are lightlike. $c+d$ is timelike and the finite dimensional torus $H_0$ is in the spacelike part.

Let us look at $\widehat{\mathfrak{h}}_0$ in the special case of $\widehat{MSU}(2)$. Here the torus is $3$\ndash dimensional; hence we have a chance to make some nice pictures:

The Weyl group acts by isometries with respect to this metric. For example its action fixes both sheets of the sphere of radius $-1$. In example we can use:
\begin{align*}
 \widehat{\mathfrak{h}}_{-1}&:=\{x\in \widehat{\mathfrak{h}}_0||x|=-1\},\\
 \widehat{\mathfrak{h}}_{0}&:=\{x\in \widehat{\mathfrak{h}}_0||x|=0\},\\
 \widehat{\mathfrak{h}}_{1}&:=\{x\in \widehat{\mathfrak{h}}_0||x|=1\}.\\
\end{align*}

Each of those spaces consists of two sheets:

\begin{align*}
\widehat{\mathfrak{h}}_{-1,+}&:=\{x\in \widehat{\mathfrak{h}}_0||x|=-1, r_d \geq 0\},\\
\widehat{\mathfrak{h}}_{-1,-}&:=\{x\in \widehat{\mathfrak{h}}_0||x|=-1, r_d \leq 0\},\\
\widehat{\mathfrak{h}}_{0,+}&:=\{x\in \widehat{\mathfrak{h}}_0||x|=0, r_d \geq 0\},\\
\widehat{\mathfrak{h}}_{0,-}&:=\{x\in \widehat{\mathfrak{h}}_0||x|=0, r_d \leq 0\},\\
\widehat{\mathfrak{h}}_{1,+}&:=\{x\in \widehat{\mathfrak{h}}_0||x|=1, r_d \geq 0\},\\
\widehat{\mathfrak{h}}_{1,-}&:=\{x\in \widehat{\mathfrak{h}}_0||x|=1, r_d \leq 0\}.
\end{align*}

Each of those sheets is $2$\ndash dimensional.

Besides the length of a vector the adjoint action preserves the coefficient $r_d$ of $d$. Hence we can intersect the six sheets $\widehat{\mathfrak{h}}_{\pm 1,0,\pm}$ with planes  such that $r_d$ is constant; for explicitness we choose $r_d=\pm 1$. The resulting sections are $1$\ndash dimensional; they are preserved by the Adjoint action of the Kac-Moody group. Hence the affine Weyl group acts on those sections. In pictures \ref{fig:h=0}, \ref{fig:h=-1} and \ref{fig:h=1} we show the structure of $\widehat{\mathfrak{h}}$ together with the $1$\ndash dimensional intersections. All those spaces are isometric to $\mathbb{R}$.
Let us suppose we have applied this isometry. Then we can describe the action of the affine Weyl group of $\widehat{MSU}(2)$ explicitly:
We have
\begin{displaymath}
 W_{\widehat{MSU}(2)}:=\{s_0,s_1|s_0^2=s_1^2=\textrm{Id}\}\, .
\end{displaymath}

We can implement the two elements $s_0$ and $s_1$ explicitly as the reflections at $0$ and $1$ in $\mathbb{R}$:
\begin{align*}
 s_0(x)&=-x,   &x \in \mathbb{R},\\
 s_1(x)&=1-(x-1)=2-x\ &x\in \mathbb{R}.
\end{align*}

The set of singular elements on this $1$\ndash dimensional space consists just of the integers $\mathbb{Z}\subset \mathbb{R}$. The linear continuation to $\widehat{t}$ is straight forward: Singular elements consist of $2$\ndash dimensional hyperplanes: The hyperplane containing $z_0\subset \mathbb{Z}$ is generated by $c$ and the line through the origin containing $z_0$. Hence all Cartan subalgebras in $\widehat{MSU}(2)$ are Lorentz spaces, isometric to $\mathbb{R}^{2,1}$. 

They are pieced together exactly as in the finite dimensional situation: in projective spaces around singular elements. Nevertheless, there is one special feature: All Cartan subalgebras contain the central direction $c$. Hence as the ``local'' building block around $c$ we find a infinite dimensional projective space.

\begin{figure}[p:]
\centerline{\includegraphics[width=0.32\textwidth]{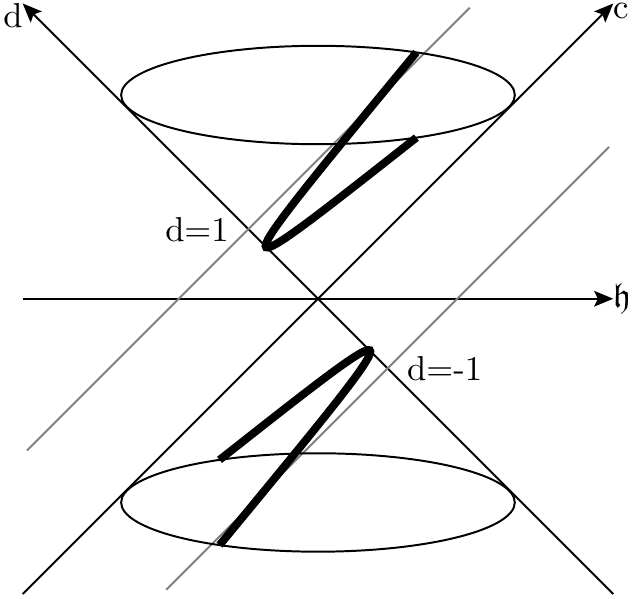}}
 \caption{This figure shows $\widehat{\mathfrak{h}}_{0,\pm 1}$. The fat lines are the intersections of the lightcone with the planes $r_d=\pm 1$\label{fig:h=0}}
\end{figure}

\begin{figure}[p:]
\centerline{\includegraphics[width=0.32\textwidth]{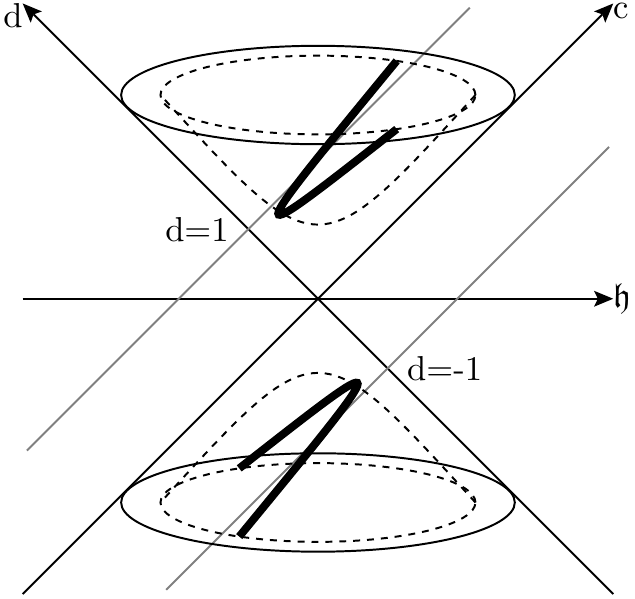}}
 \caption{This figure shows $\widehat{\mathfrak{h}}_{-1,\pm 1}$. The dotted lines represent the sphere of radius $-1$. The fat lines are the intersections with the planes $r_d=\pm 1$\label{fig:h=-1}}
\end{figure}

\begin{figure}[p:]
\centerline{\includegraphics[width=0.32\textwidth]{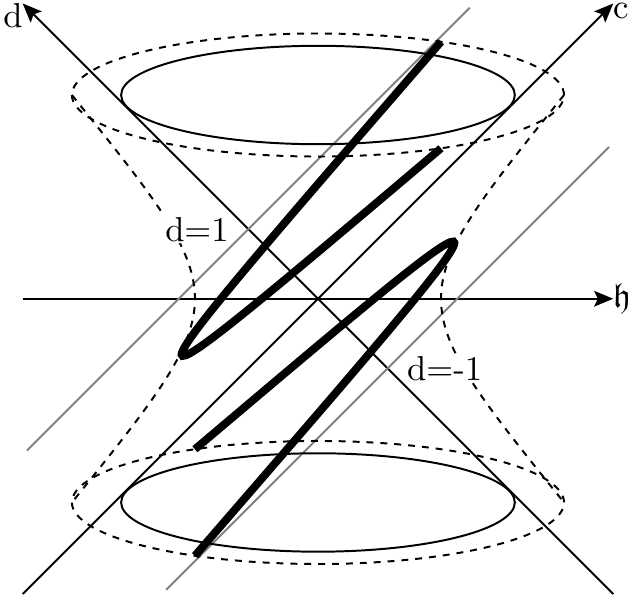}}
 \caption{This figure shows $\widehat{\mathfrak{h}}_{1,\pm 1}$. The dotted lines represent the sphere of radius $1$. The fat lines are the intersection with the planes $r_d=\pm 1$\label{fig:h=1}}
\end{figure}

Let us now go back to $\widehat{MSU}(n)$. Here the structure is quite similar: For every $SU(2)\in SU(n)$ we find an $\widehat{MSU}(2)$\ndash subspace. The geometric structure is exactly the same.


\section{Geodesics}
\label{geodesics}

\noindent The most important fact about the structure of geodesics is that in Fr\'echet Kac-Moody symmetric spaces geodesic
completeness fails even locally.  Hence there is no chance to get a Hopf-Rinow-type
theorem. Global failure of geodesic completeness is a typical feature of Lorentz geometry, local failure is a special feature of the functional analytic setting. In a Banach Lie group for example the inverse function theorem assures the local geodesic completeness. 

There are various results describing criteria for impossibility of Lorentz manifolds to be
geodesically complete. For details~\cite{Beem81}.

The failure of global geodesic completeness of finite dimensional Lorentz manifolds  leads
directly to the fact that Kac-Moody symmetric spaces are not locally
geodesically complete (compare the construction for $SL(2,\mathbb{C})$, example~\ref{sl2chasnodiffeomorphicexponentialmap}).

Let us first investigate the behaviour of the geodesics of the (Riemannian) symmetric space $MG_{\mathbb{R}}$  equipped with the metric 
\begin{displaymath}
\langle f,g\rangle=\frac{1}{2\pi}\int_{0}^{2\pi}\langle f(t), g(t)\rangle dt\, .
\end{displaymath}

In this space geodesics are intimately related to the geodesics of the finite dimensional Lie group $G_{\mathbb{C}}$

\begin{proposition}
$\gamma(t) \subset MG$ is a geodesic iff $\gamma_z(t) \subset G$ is a geodesic
for every $z \in \mathbb{C}^*$.
\end{proposition}

\begin{corollary}
Let $f$ and  $g\in MG$. A necessary and sufficient condition for the existence of a geodesic segment $\gamma([0,1])$ connecting $f$ and $g$ is:
For all $z\in \mathbb{C}^*$ there is a geodesic in $G_{\mathbb{C}}$ connecting $Id$ and $fg^{-1}(z)$.
\end{corollary}

The $d$-extension twists those geodesics by a function $\varphi:\mathbb{C}\longrightarrow \mathbb{C}$. This gives rise to the following theorem for $\widehat{MG}$:

\begin{theorem}[Geodesic through $e$]
Let $\pi: \widehat{MG}\longrightarrow MG$ be the bundle projection. Let
$\dot{\gamma}:=(f, r_c, r_d) \in M\mathfrak{g}$ and $\gamma(t)$ be the associated
geodesic. Then $(\pi \gamma)_{\varphi(z,r_d)}(t)$ is a geodesic in $G$.
\end{theorem}

As in finite dimensions geodesics in a Kac-Moody symmetric space are related to the group structure:

\begin{theorem}[$1$-parameter subgroups]
A curve in $\widehat{MG}$ through the identity element $e$ is a geodesic iff it is a $1$\ndash parameter
subgroup.
\end{theorem}

\begin{proof}
The proof is as in the finite dimensional case.
\end{proof}

\begin{corollary}
 There is a geodesic containing $x\in \widehat{MG}$ iff $x$ lies in the image of the exponential map $\Mexp$. 
\end{corollary}

Let us now give some remarks about the use of ILB-structure: We can describe the tame Fr\'echet symmetric space as an inverse limit of Banach symmetric spaces. Obviously in each of those Banach spaces the exponential function describes a local diffeomorphism. Hence in each of those Banach spaces any two sufficiently close points can be connected by a geodesic. We call a geodesic in one of those Banach symmetric spaces a $B$-geodesic. Every geodesic in the Fr\'echet symmetric space defines a geodesic in some Banach space.

\section{Hermitian Kac-Moody symmetric spaces?}

Several finite dimensional Riemannian symmetric spaces admit a complex structure, making them into hermitian symmetric space~\cite{Helgason01}, p.~352. As the hermitian structure of a hermitian symmetric space is Kaehler, those space provide nontrivial examples of Kaehler manifolds and have various applications in complex analysis. Hence we have to investigate if there are hermitian Kac-Moody
symmetric spaces. 
Following the finite dimensional blueprint we make the definition:

\begin{definition}
Let $M$ be an affine Kac-Moody symmetric space. $M$ is an Hermitian Kac-Moody symmetric space if
\begin{enumerate}
 \item $M$ is a complex manifold with a Hermitian structure.
 \item Each point $p\in M$ is an isolated fixed point of an involutive holomorphic isometry $\sigma_p$ of $M$.
\end{enumerate}
\end{definition}

The answer to the existence of hermitian Kac-Moody symmetric spaces is negative:

\begin{theorem}
There are no hermitian affine Kac-Moody symmetric spaces.
\end{theorem}

\begin{proof}
Let $p\in M$, $g_p:T_pM\times T_pM\longrightarrow \mathbb{R}$ the metric and  $J_p:T_pM\longrightarrow T_pM$ be the complex structure. Then 
\begin{displaymath}
g_p(v,w)=g_p(J_p(v), J_p(w))
\end{displaymath}
On a Kac-Moody symmetric space we have especially:
\begin{displaymath}
-2=g_p(c+d,c+d)=g_p(J_p(c+d), J_p(c+d))
\end{displaymath}
Hence we have a $2$-dimensional subspace of negative index ---  this is contradiction as we started with a Lorentzian metric. 
\end{proof}

This is similar to the theory of finite dimensional Lorentzian manifolds. 
Alternatively we could argue via the structure of the isometry group, generalizing a result of~\cite{Helgason01}, p.~382 which states:

\begin{theorem}
Let $G/K$ be a hermitian symmetric space. Then the center $Z_K$ of $K$ is analytically isomorphic to $S^1$.
\end{theorem}

The question of the existence of a para-hermitian structure on Kac-Moody symmetric spaces is open.

\chapter{Open problems and various open ends}
There are various deep open question concerning the geometry of Kac-Moody symmetric spaces themselves and their relations to other objects of mathematics. We just list some of them without going into details.

\begin{enumerate}
 \item {\bf Mostow rigidity:} Study quotients of Kac-Moody symmetric spaces of the non-compact type. We
 conjecture the existence of a Mostow-type theorem for Kac-Moody symmetric spaces of the non-compact type, if the rank $r$ of each irreducible factor satisfies $r\geq 4$; in the finite dimensional situation the main ingredients for the
  proof of Mostow rigidity are the spherical buildings which are associated to the universal covers $\widetilde{M}$ and
  $\widetilde{M}'$ of two homotopy equivalent locally symmetric spaces $M=\widetilde{M}/\Gamma$ and $M'=\widetilde{M'}/\Gamma'$ (suppose the rank $r$ of each de Rham factor satisfies  $\textrm{rank}(M)\geq 2$). To prove Mostow rigidity one has to show that a homotopy equivalence of the quotients lifts to a quasi isometry of the universal covers and induces a building isomorphism. This step is done via a description of the boundary at infinity. By rigidity results of Jacques Tits this building isomorphism is known to introduce a group isomorphism which in turn leads to an isometry of the quotients.

 Hence to prove a generalization of Mostow rigidity to quotients of Kac-Moody symmetric spaces of the non-compact type along these lines, one needs a description of the boundary of Kac-Moody symmetric spaces of the non compact type. In view of their Lorentz structure, those spaces are not CAT(0). Consequently, a direct generalization of the finite dimensional ideas to construct a boundary seems difficult. Nevertheless as for each point in the symmetric space a boundary can be defined, a complete definition using the action of the Kac-Moody groups seems within reach. We note that a generalization or adaption of the methods developed 
in \cite{Caprace09} and in \cite{Gramlich09} might lead to a proof of Mostow rigidity based on local\ndash global methods. By work of Andreas Mars this is possible for algebraic Kac-Moody groups satisfying various technical assumptions in case of rings with sufficiently many units~\cite{Mars10}.

\item{\bf Holonomy:} Study the holonomy of infinite dimensional Lorentz manifolds. More precisely: Is there a generalization of the Berger holonomy theorem to Kac-Moody symmetric spaces? In the finite dimensional Riemannian case, Berger's holonomy theorem tells us the following: Let $M$ be a simply connected manifold with irreducible holonomy. Then either the holonomy acts transitively on the tangent sphere or $M$ is  symmetric.  
Kac-Moody symmetric spaces are Lorentz manifolds and have a distinguished, parallel lightlike vector field, corresponding to $c$. 
In relativity, similar (four dimensional) manifolds are known as Brinkmann waves~\cite{Brinkmann25}. Hence it seems that Kac-Moody symmetric spaces should be understood as a kind of infinite dimensional Brinkmann wave. The conjectured holonomy theorem states then, that the holonomy of an infinite dimensional Brinkmann wave is either transitive on the horosphere around $c$ or the space is a Kac-Moody symmetric space.

\item{\bf Characterization:} Is there some intrinsic characterization of Kac-Moody symmetric spaces? Probably, the existence of sufficiently many finite dimensional flats and the closely related Fredholm property of the induced polar representations will be part of this characterization. From an algebraic point of view, it would be very interesting, to investigate if there are symmetric spaces associated to Lorentzian Kac-Moody algebras - nevertheless, the absence of good explicit realizations seems to be a serious impediment. From a functional analytic point of view, a generalization to other classes of maps, i.e.\ holomorphic maps on Riemann surfaces or symplectic maps would be interesting. From the finite dimensional blueprint the development of an infinite dimensional version of the theory of non-reductive pseudo-Riemannian symmetric spaces as developed in the finite dimensional case by 
I.\ Kath and M.\ Olbricht in \cite{Kath06} seems promising.  In view of the characterization of Kac-Moody symmetric spaces as infinite dimensional Brinkmann waves, we would like to understand, how the existence of this special direction is related to other structure properties. It is hence natural to ask if there is a class of infinite dimensional Lorentz symmetric spaces without such a distinguished direction? 
\end{enumerate}

\appendix
\backmatter
\bibliographystyle{alpha}
\bibliography{Doktorarbeit1}
\printindex
\end{document}